\documentclass[10pt, a4paper, final]{article}

\usepackage{mathtools}
\usepackage{amsmath, amsthm, amssymb}
\usepackage{lmodern}
\usepackage[utf8]{inputenc}
\usepackage[T1]{fontenc}
\usepackage[margin=3cm]{geometry}
\usepackage{enumitem}
\usepackage{comment}
\usepackage[notcite, notref  ]{showkeys}

\usepackage[normalem]{ulem}%%%
\usepackage{cases}%%%
\usepackage{hyperref}
\usepackage{bbm}
\usepackage{float}
\usepackage{tikz}

\usepackage[backend=biber, sorting=nyt, style=numeric, url=false, isbn=false, eprint=false, maxbibnames=4, giveninits=true]{biblatex}
\renewbibmacro{in:}{\ifentrytype{article}{}{\printtext{\bibstring{in}\intitlepunct}}}
\DeclareFieldFormat
[article,inbook,incollection,inproceedings,patent,thesis,unpublished]
{title}{#1\isdot}
\DeclareFieldFormat[article]{volume}{\mkbibbold{#1}}
\DeclareFieldFormat[article]{number}{\mkbibbold{#1}}
\newenvironment{system}[1]
{%before
\left\{\begin{aligned}{#1}} 
{ \end{aligned}\right.}

\newcommand\dd{\mathrm{d}}
\newcommand\vect[1]{{#1}}

\newtheorem{thm}{Theorem}[section]
\newtheorem{theorem}[thm]{Theorem}
\newtheorem{lem}[thm]{Lemma}
\newtheorem{lemma}[thm]{Lemma}
\newtheorem{prop}[thm]{Proposition}
\newtheorem{proposition}[thm]{Proposition}

\theoremstyle{definition}
\newtheorem{defi}{\textbf{Definition}}
\newtheorem{defn}[defi]{\textbf{Definition}}
\newtheorem{definition}[defi]{\textbf{Definition}}

\theoremstyle{remark}
\newtheorem{rem}[thm]{Remark}
\newtheorem{remark}[thm]{Remark}
\newtheorem{assumption}{\textbf{Assumption}}

\newcommand\ep{\varepsilon}
\newcommand\R{\mathbb{R}}

\newcounter{stepnum}

\allowdisplaybreaks

\newcommand{\writefoot}[1]{
    \renewcommand{\thefootnote}{}
    \footnotetext{\hspace{-16.5pt}\scriptsize#1}
    \renewcommand{\thefootnote}{\arabic{footnote}}
}
\numberwithin{equation}{section}

\begin{document}
\writefoot{\small \textbf{AMS subject classifications (2020).} Primary: 35K40, 35C07, 35K57, 35K58; Secondary: 92D25. \smallskip}
\writefoot{\small \textbf{Keywords.} 
Reaction-diffusion system, spreading speed, pulsating traveling wave, anisotropic propagation, stability analysis, singular limit, Lyapunov function.\smallskip
}
\writefoot{\small \textbf{Acknowledgements:} 
The initial part of the research was conducted when QG was a JSPS International Research Fellow (FY2017; Graduate School of Mathematical Sciences, The University of Tokyo and MIMS, Meiji University). The research was partially supported by JSPS KAKENHI 16H02151 and 21H00995. QG was partially supported by a PEPS-JCJC grant from CNRS (2019) and by ANR grant “Indyana” Number ANR-21-CE40-0008.
}

\begin{center}
    \begin{minipage}{0.9\textwidth}
	\centering
    \LARGE{\bf Front propagation in hybrid reaction-diffusion epidemic models with spatial heterogeneity. \\ Part II: Pulsating traveling waves }\bigskip
    \end{minipage}

    \Large
    Quentin Griette\medskip \\
    \normalsize
    {\it Laboratoire de Math\'ematiques Appliqu\'ees du Havre, Universit\'e Le Havre Normandie, \\
    EA 3821, FR CNRS 3335 \\ 
    25, rue Philippe Lebon, 76600 Le Havre, France}\\
    {\tt quentin.griette@univ-lehavre.fr}
    \bigskip

    \Large
    Hiroshi Matano \medskip \\ 
    \normalsize
    {\it Meiji Institute for Advanced Study of Mathematical Sciences, Meiji University, \\
    4-21-1 Nakano, Tokyo 164-8525, Japan 
    }\\
    {\tt matano@meiji.ac.jp} \\ 
	\bigskip

	\today
\end{center}

\begin{abstract}
    We consider a two-species reaction-diffusion system in one space dimension that is derived from an epidemiological model in a spatially periodic environment with two types of pathogens: the wild type and the mutant. 
    The system is of a hybrid nature, partly cooperative and partly competitive, but neither of these entirely. 
As a result, the comparison principle does not hold.
In the previous work, we studied the propagation properties of the solutions to the Cauchy problem for this system and showed, among other things,  that the spreading speeds of the fronts to the right and to the left directions, denoted  by $ c^*_R$ and  $ c^*_L$, can be characterized by using certain principal eigenvalues, and studied the homogenization limit as the spatial period $L$ tends to $0$, and also discussed the long-time behavior of solutions behind the fronts. 

	In the present paper   we prove the existence of pulsating traveling waves in the right direction (resp. left direction) with speed $c$ for any  $c\geq c^*_R$ (resp. $c\geq c^*_L$), where $c^*_R$ and $c^*_L$ denote the aforementioned spreading speeds in the right and left directions.  We also prove that the leading edge of any traveling wave has the exponential decay rate that is anticipated from formal linear analysis, thus extending part of the results of Hamel 2008 to systems of equations. 
	Finally, we present an example in which the two speeds $c^*_R$ and $c^*_L$ are different. This is done by considering a multi-scale singular limit problem. This result highlights a marked difference between our system and scalar KPP type equations. 
\end{abstract}

%\maketitle

\section{Introduction}

In this paper we consider the following reaction-diffusion system:
\begin{equation}\label{eq:main-sys}
	\begin{system}
		\relax &u_t=\big(\sigma(x) u_{x}\big)_x+\big(r_u(x)-\kappa_u(x)(u+v)\big)u+\mu_v(x)v-\mu_u(x)u, & t>0, x\in\R, \\
		\relax &v_t=\big(\sigma(x) v_{x}\big)_x+\big(r_v(x)-\kappa_v(x)(u+v)\big)v+\mu_u(x)u-\mu_v(x)v, &t>0, x\in\R.
	\end{system}
\end{equation}
Here $u(t, x)$, $v(t,x) $ stand for the density of a population of individuals living in  a periodically heterogeneous environment, 
  $r_u(x)\in\mathbb{R} $ and $ r_v(x) \in\mathbb{R}$ represent the intrinsic growth rate (which are not necessarily positive),  $\kappa_u(x)>0$ and $\kappa_v(x)>0$ 
  represent the intensity of the competition, and  $\mu_u(x)>0$, $\mu_v(x)>0$ denote the mutation rates between the two populations.  
 We assume that these coefficients, including the diffusion coefficient $\sigma(x)$, are all $L$-periodic.  

The above problem is motivated partly by the study of SIS epidemiological models describing the propagation of pathogens that are subject to mutations, such as the following:
\begin{equation} \label{eq:SIS}
	\left\{\begin{aligned}\relax
		&\partial_t S_t = \partial_{x} (\sigma(x)\partial_x S) - \big(\beta_1(x) I_1+\beta_2(x) I_2\big)S + \gamma_1(x) I_1 + \gamma_2(x) I_2, && t>0, x\in\R,  \\
		&\partial_t I_1 = \partial_x(\sigma(x) \partial_{x}I_1) + \beta_1(x) S I_1 -\gamma_1(x) I_1 + \mu_2(x) I_2 - \mu_1(x) I_1, && t>0, x\in\R,   \\ 
		&\partial_t I_2 = \partial_x(\sigma(x) \partial_{x} I_2)+\beta_2(x) S I_2 - \gamma_2(x) I_2+\mu_1(x) I_1 - \mu_2(x) I_2,  && t>0, x\in\R  . 
	\end{aligned}\right.
\end{equation}
Here the $\beta_1,\beta_2$ denote the infection rates, $\gamma_1,\gamma_2$ the recovery rates, and $\mu_1,\mu_2$ stand for the mutation rates between the pathogens.

It is not difficult to show that the quantity $N(t, x):=S(t,x)+I_1(t,x)+I_2(t,x)$ satisfies a pure diffusion equation
$\partial_t N=\partial_x\left(\sigma(x)\partial_x N\right)$. Therefore, if we assume
that $ N(0, x)$ is constant in $x$, 
then $N(t, x)$ remains constant in $t$ and $x$. 
Thus we obtain that  $u=I_1$, $v=I_2$ satisfy \eqref{eq:main-sys} with 
$r_{*}(x) := N\beta_{i}(x)-\gamma_{i}(x)$ and $\kappa_{i}(x):=\beta_{i}(x)$, where $i=1,2$ and  $*=u,v$ .  
Hence the propagation dynamics of \eqref{eq:SIS} is equivalent to that of \eqref{eq:main-sys}.

System \eqref{eq:SIS} describes the propagation of a genetically unstable pathogen in a population of hosts which exhibit heterogeneity in space. This heterogeneity represents the spatially heterogeneous environment that affects the behavior of individuals depending on where they are. Spatial heterogeneity in the use of antibiotics, fungicides or insecticides affects the transmission of pathogens and pests and is explored as a way to minimize the risk of emergence of drug resistance \cite{Deb-Len-Gan-09}.    
Beaumont et al \cite{Bea-Bur-Duc-Zon-12} study a related model of propagation of salmonella in an industrial hen house. In their study the heterogeneity comes from the alignment of cages separated by free space that allow farmers to take care of the animals. Griette et al \cite{Griette-Alfaro-Raoul-Gandon} studied the propagation properties of a closely related model in the context of the evolution of drug resistance.   

The propagation speed of the solutions of reaction-diffusion equations is often linked to special solutions called \textit{traveling wave solutions}, that are  particular solutions that propagate at a prescribed speed. There exists a large literature on traveling wave solutions and the propagation dynamics of solutions to reaction-diffusion equations and systems, see \cite{Kol-Pet-Pis-1937, Fis-1937, Aro-Wei-75,Aro-Wei-78,Wei-82, Lui-89, Vol-Vol-Vol-94, Lia-Zha-07, Lia-Zha-10} among others. 
When the coefficients depend periodically on the spatial variable such as \eqref{eq:main-sys}, the traveling waves are often  called \textit{pulsating traveling waves}, see \cite{Shi-Kaw-Ter-86, Xin-00, Ber-Ham-02, Wei-02} among others.

Our system has a rather intriguing character in the sense that it is cooperative when $(u, v)$ is small while the competitive nature becomes dominant when $(u, v)$ is large. Therefore the standard comparison principle does not apply to the entire system. Such a system has been studied by Wang \cite{Wan-11}, Wang and Castillo-Chavez \cite{Wang-Castillo--Chavez-2012}, Griette and Raoul \cite{Gri-Rao-16}, Girardin \cite{Gir-18, Gir-18-MMMAS}, and Morris, B\"orger and Crooks \cite{Mor-Bor-Cro-19}, when the coefficients are homogeneous in space. 
However, in 
our case, the coefficients are spatially periodic. 
As far as scalar equations are concerend, there 
is a large literature on equations with periodic coefficients, notably \cite{Shi-Kaw-Ter-86, Xin-00, Ber-Ham-02, Wei-02, Lia-Zha-10}. 
As for systems, 
Alfaro and Griette \cite{Alf-Gri-18} constructed a traveling wave for a related system that travels at the expected minimal speed. %Apart from this last result, to the best of our knowledge, little  is known for systems of hybrid nature with spatially periodic coefficients.
The special case where the coefficients are spatially homogeneous has been treated in \cite{Wan-11, Wang-Castillo--Chavez-2012, Gri-Rao-16, Gir-18, Gir-18-MMMAS, Mor-Bor-Cro-19}, 

In our previous work \cite{Griette-Matano-2025}, we considered system \eqref{eq:main-sys} under the initial condition
\begin{equation}\label{main-initial}
u(0,x)=u_0(x),\ \ v(0,x)=v_0(x),\quad\ x\in\R ,
\end{equation}
where $u_0, v_0$ are bounded nonnegative functions on $\R $ whose supports are, in most cases, localized in space. 
Among other things, we proved that the solutions fronts propagate to the right and to the left asymptotically at constant speeds, that can be computed from a family of linearized eigenproblems around the leading edge. We also studied the behavior of the solution far behind the propagation front and proved the convergence to the unique steady state in the case where the coefficients are spatially homogeneous or rapidly oscillating. Finally, we studied the homogenization problem as $L\to 0$ and determined the effective coefficients of the homogenized limit. 

In the present paper, we focus on traveling waves and study their existence and qualitative properties.
We first prove the existence of pulsating traveling waves in the right (resp. left) direction for any speed $c\geq c^*_R$ (resp. $c\geq c^*_L$). 
We also show that the leading edge of any traveling wave decays exponentially 
as $x\to\infty$ (or quasi-exponentially in the critical case where $c=c^*_R$ or $c=c^*_L$). This result implies that all the pulsating traveling waves fulfill the ansatz that is
anticipated from formal linear analysis. This extends part of the results of Hamel \cite{Ham-08} to systems of equations. We note that our existence proof for the critical case $c=c^*_R$ or $c=c^*_L$ does not rely on the usual limiting argument $c\to c^*_R, c^*_L$. We instead give a direct proof of the existence for the critical case by constructing  super- and sub-solutions that are different from the non-critical case $c>c^*_R, c^*_L$.  

Next, we  consider the case where the coefficients are spatially homogeneous or rapidly oscillating, and prove that the profile of the traveling wave far behind its front converges to the unique periodic (or homogeneous) steady state. Finally, we present an example where $c^*_R \neq c^*_L$. This is done by a multi-scale singular limit analysis. We remark that, in the scalar KPP equation, the left and right critical wave speeds are always equal, even if the coefficients are not symmetric in $x$, which is a consequence of the Fredholm alternative. This last result highlights a marked difference between scalar equations and systems.

Our paper is organized as follows. In section \ref{s:main} we first recall key mathematical notions such as principal eigenvalues of various kinds, left and right spreading speeds, and state their basic properties. This will be done in subsection \ref{ss:eigenproblem}. Then, in subsections \ref{ss:traveling-waves} and \ref{ss:asymmetric}, we present our main results including the existence of critical and non-critical traveling waves (Theorem \ref{thm:TW1}), decay properties along the leading edge (Theorem \ref{thm:TW2}), convergence to the steady states in the spatially homogeneous and rapidly oscillating cases (Theorem \ref{thm:asymptotic-behavior}), and the construction of an example with asymmetric propagation speeds (Theorem \ref{thm:asymmetric-speeds}).
The proofs of  Theorems \ref{thm:TW1},  \ref{thm:TW2}, \ref{thm:asymptotic-behavior} and \ref{thm:asymmetric-speeds} will be given in sections \ref{s:existence}, \ref{s:decay}, \ref{sec:behind} and \ref{s:proof-asym-speed}, respectively.
%In section \ref{s:existence} we prove Theorem \ref{thm:TW1}; in section \ref{sec:behind} we prove Theorem \ref{thm:asymptotic-behavior}; in section \ref{s:decay} we prove Theorem \ref{thm:TW2}; and finally in section \ref{s:proof-asym-speed} we prove Theorem \ref{thm:asymmetric-speeds}. 

%%%%%%%%%%%%%%%%%%%%%
%%%%%%%%%%%%%%%%%%%%%
\section{Main results}
\label{s:main}
Throughout this article we make the following assumption on the coefficients of \eqref{eq:main-sys}.
\begin{assumption}[Cooperative-competitive system]\label{as:coop-comp}
	The coefficients $\sigma(x)>0$,   $\kappa_u(x)>0$, $\kappa_v(x)>0$, $\mu_v(x)> 0$, $\mu_u(x)> 0$, are $L$-periodic positive continuous functions and  $r_u(x)$, $r_v(x)$ are $L$-periodic continuous functions of arbitrary sign.  We assume moreover that $\sigma\in C^1(\R )$.
\end{assumption}
\begin{remark}
	In formulating our problem, we chose to use the same diffusion coefficient $\sigma(x)$ for $u$ and $v$, in order to make the connection between \eqref{eq:main-sys} and \eqref{eq:SIS} clearer. However, most of the results in the present paper remain to hold even if the diffusion coefficients for $u$ and $v$ are different;  in particular, the existence result for traveling waves (Theorem \ref{thm:TW1}) and the result on the exponential decay (Theorem \ref{thm:TW2}) can easily be extended to the case of unequal diffusion coefficients.
\end{remark}
To facilitate the statement of our results, we introduce the following notations: 
\begin{equation}\label{eq:maxmincoeff}
	\begin{gathered}
		r_{\max}:=\sup_{x\in\R}\max\big(r_u(x), r_v(x)\big),\ \
		\kappa_{\max}:=\sup_{x\in\R} \max\big(\kappa_u(x), \kappa_v(x)\big),\\ 
		\mu_{\max}:=\sup_{x\in\R}\max\big(\mu_u(x), \mu_v(x)\big), \ \ 
		\sigma_{\max}:=\sup_{x\in\mathbb{R}} \sigma(x), \\ 
		r_{\min}:=\inf_{x\in\R}\min\big(r_u(x), r_v(x)\big),\ \
		\kappa_{\min}:=\inf_{x\in\R} \min\big(\kappa_u(x), \kappa_v(x)\big),\\ 
		\mu_{\min}:=\inf_{x\in\R}\min\big(\mu_u(x), \mu_v(x)\big), \ \ 
		\sigma_{\min}:=\inf_{x\in\mathbb{R}} \sigma(x), \ \  
		\overline{K}:=\frac{r_{\max}}{\kappa_{\min}}.
	\end{gathered}
\end{equation}

Before presenting our main results in this section, we remark that nonnegative solutions of \eqref{eq:main-sys} are all bounded as $t\to+\infty$. 
%To state this basic estimate, we introduce the following notation:
%\begin{equation}\label{rmax-kappamin}
%r_{\max}:=\sup_{x\in\R}\max\big(r_u(x), r_v(x)\big),\ \ 
%\kappa_{\min}:=\inf_{x\in\R} \min\big(\kappa_u(x), \kappa_v(x)\big),\ \ 
%\end{equation}

\begin{prop}[Basic boundedness estimate \protect{\cite[Proposition 4]{Griette-Matano-2025}}]\label{prop:uniform-bound}
Let $(u(t,x),v(t,x))$ be a solution of \eqref{eq:main-sys} with nonnegative bounded initial data $(u_0(x),v_0(x))$. Then 
$u(t,x)\geq 0$, $v(t,x)\geq 0$ for all $t\geq0, x\in\R$, and
\begin{equation}\label{u+v<max(K,u0+v0)}
u(t,x)+v(t,x)\leq \max\big(\overline{K},\, \sup_{x\in\R}(u_0(x)+v_0(x))\big) \ \ \hbox{for all}\ \ t\geq0, x\in\R,
\end{equation}  
\begin{equation}\label{uniform-bound}
\limsup_{t\to+\infty}\sup_{x\in\R}\big(u(t,x)+v(t,x)\big)\leq \overline{K}.
\end{equation}
In particular, if $u_0(x)+v_0(x)\leq\overline{K}\,(x\in\R)$, then $u(t,x)+v(t,x)\leq\overline{K}\,(t\geq 0,\,x\in\R)$.
\end{prop}

\subsection{Principal eigenproblems  and spreading speed}\label{ss:eigenproblem}
In this subsection we recall some of the notions and basic results presented in our previous paper \cite{Griette-Matano-2025}. They will be used in our later arguments.

The linearized system associated with \eqref{eq:main-sys} is the following. 
\begin{equation}\label{eq:linearized}
	\begin{system}
		\relax &u_t=\big(\sigma(x) u_{x}\big)_x+r_u(x)u+\mu_v(x)v-\mu_u(x)u, & t>0, x\in\R, \\
		\relax &v_t=\big(\sigma(x) v_{x}\big)_x+r_v(x)v+\mu_u(x)u-\mu_v(x)v, &t>0, x\in\R.
	\end{system}
\end{equation}
Note that this is a cooperative system. 
We first define the notions of periodic principal, $\lambda$-periodic principal and Dirichlet principal eigenelements as follows. 
\begin{definition}[Periodic principal eigenpair]\label{def:periodic-principal-eigenpair}
	By a \textit{periodic principal eigenpair} associated with \eqref{eq:linearized} we mean any pair $\big(\lambda_1^{per}, (\varphi(x), \psi(x))\big)$ where $\lambda_1^{per}\in\R $, $\varphi(x)$ and $\psi(x) $ are positive $L$-periodic smooth functions that satisfy 
	\begin{equation}\label{eq:periodic-principal-eigen}
		\begin{system}
			\relax &  
			L^1[\varphi, \psi](x) :=\big(\sigma(x) \varphi_{x}\big)_x+r_u(x)\varphi+\mu_v(x)\psi-\mu_u(x)\varphi=\lambda_1^{per} \varphi, \\
			\relax & 
			L^2[\varphi, \psi](x) :=
			\big(\sigma(x) \psi_{x}\big)_x+r_v(x)\psi+\mu_u(x)\varphi-\mu_v(x)\psi=\lambda_1^{per}\psi .
		\end{system}
	\end{equation}
	We call $\lambda_1^{per}$  the \textit{periodic} principal eigenvalue and $(\varphi, \psi)$  a \textit{principal eigenvector}.
\end{definition}

{It follows from the Krein-Rutman Theorem that $\lambda_1^{per}$ is unique, and that $(\varphi, \psi)$ is unique up to multiplication by a positive scalar.}

\begin{definition}[$\lambda$-periodic principal eigenpair]\label{def:k(lambda)}
	For $\lambda>0$, by a \textit{$\lambda$-periodic principal eigenpair} associated with \eqref{eq:linearized} we mean any pair $\big(k(\lambda), (\varphi(x), \psi(x))\big)$ where $k(\lambda)\in\R $, $\varphi(x)$ and $\psi(x) $ are positive $L$-periodic smooth functions that satisfy 
	\begin{equation}\label{eq:lambda-periodic-principal-eigen}
		\begin{system}
			\relax & L^1_\lambda[\varphi, \psi](x) := e^{\lambda x} L^1[e^{-\lambda x}\varphi, e^{-\lambda x}\psi](x) =k(\lambda) \varphi, \\
			\relax & L^2_\lambda[\varphi, \psi](x) := e^{\lambda x} L^2[e^{-\lambda x}\varphi, e^{-\lambda x}\psi](x)= k(\lambda)\psi .
		\end{system}
	\end{equation}
	or, equivalently, 
	\begin{equation}\label{eq:lambda-periodic-principal-eigen2}\tag{\ref{eq:lambda-periodic-principal-eigen}$'$}
		\begin{system}
			\relax & \big(\sigma(x) \varphi_{x}\big)_x-2\lambda \sigma(x) \varphi_x+\big(\lambda^2\sigma(x)-\lambda\sigma_x(x)+r_u(x)\big)\varphi+\mu_v(x)\psi-\mu_u(x)\varphi=k(\lambda) \varphi, \\
			\relax & \big(\sigma(x) \psi_{x}\big)_x-2\lambda \sigma(x)\psi_x+\big(\lambda^2\sigma(x)-\lambda \sigma_x(x)+r_v(x)\big)\psi+\mu_u(x)\varphi-\mu_v(x)\psi = k(\lambda)\psi .
		\end{system}
	\end{equation}
	We call $k(\lambda)$ the \textit{$\lambda$-periodic principal eigenvalue} and $(\varphi, \psi)$ a \textit{$\lambda$-periodic principal eigenvector}. 
\end{definition}

{Again by the Krein-Rutman Theorem, for each $\lambda\in\R$, $k(\lambda)$ is unique and that $(\varphi, \psi)$ is unique up to multiplication by a positive scalar. Clearly we have $k(0)=\lambda_1^{per}$.}

\begin{definition}[Dirichlet principal eigenpair]\label{def:Dirichlet-principal-eigenpair}
	Let $R>0$ be given. By a \textit{Dirichlet principal eigenpair on $(-R, R)$} associated with \eqref{eq:linearized} we mean any pair $\big(\lambda_1^R, (\varphi(x), \psi(x))\big)$ where $\lambda_1^R\in\R $, $\varphi(x)$ and $\psi(x) $ are positive  smooth functions on $[-R, R]$ that satisfy
	\begin{subequations}\label{eq:Dirichlet-principal-eigen}
	\begin{equation}
		\begin{system}
			\relax &\big(\sigma(x) \varphi_{x}\big)_x+r_u(x)\varphi+\mu_v(x)\psi-\mu_u(x)\varphi=\lambda_1^{R} \varphi, \\
			\relax &\big(\sigma(x) \psi_{x}\big)_x+r_v(x)\psi+\mu_u(x)\varphi-\mu_v(x)\psi=\lambda_1^{R}\psi ,
		\end{system}
	\end{equation}
	and
		\begin{equation}
			\varphi(-R)= \psi(-R)=0 \text{ and } \varphi(R)= \psi(R)=0.
		\end{equation}
	\end{subequations}
	We call $\lambda_1^{R}$ the \textit{principal eigenvalue} and $(\varphi, \psi)$ a \textit{Dirichlet principal eigenvector}.
\end{definition}

{As before, the Krein-Rutman theorem ensures that $\lambda_1^R$ is unique and $(\varphi, \psi)$ is unique up to multiplication by a positive scalar.}

\begin{prop}[Properties of $k(\lambda)$  \protect{\cite[Proposition 1]{Griette-Matano-2025}}]\label{prop:k(lambda)}
	Let Assumption \ref{as:coop-comp} hold true.  Then:
	\begin{enumerate}[label={\rm(\roman*)}]
		\item \label{item:eigenpairk(lambda)}
		    For each $\lambda\in\R$, there exists a $\lambda$-periodic principal eigenpair $\big(k(\lambda),(\varphi, \psi)\big)$ with $\varphi(x)> 0$ and $\psi(x)>0$ for all $x\in\R $, which solves \eqref{eq:lambda-periodic-principal-eigen}, and $(\varphi, \psi)$ is unique up to the multiplication by a positive scalar.
		\item \label{item:minimaxk(lambda)}
			The following characterization of $k(\lambda)$ is valid:
			\begin{equation}\label{eq:minimax-k(lambda)}
				k(\lambda)=
				\underset{(\varphi, \psi)\in C^2_{per}(\R)^2}{\underset{{{\varphi}>{0}, \psi>0}}{\min}}\underset{x\in \R}{\sup}\, {\max}\left(\frac{L^1_\lambda[\varphi, \psi](x)}{\varphi(x)},\frac{L^2_\lambda[\varphi, \psi](x)}{\psi(x)}\right), 
			\end{equation}
			where $L^1_\lambda[\varphi, \psi](x)$, $L^2_\lambda[\varphi, \psi](x)$ are as defined in \eqref{eq:lambda-periodic-principal-eigen}. In addition,  the right-hand side has a unique minimizer up to multiplication by a positive scalar, which coincides with the principal eigenvector of the problem \eqref{eq:lambda-periodic-principal-eigen}.			
		\item \label{item:k(lambda)concave}
			The function $\lambda\mapsto k(\lambda) $ is analytic and strictly convex. Furthermore, recalling the notations in \eqref{eq:maxmincoeff},  
			the following inequalites hold:
			\begin{equation}\label{k-quadratic}
				\sigma_{\min} \lambda^2 + r_{\min}\leq k(\lambda)\leq \sigma_{\max} \lambda^2+r_{\max}
				\quad \hbox{for all}\ \  \lambda\in \R. 
			\end{equation}
	\end{enumerate}
\end{prop}

\begin{remark}\label{rem:251025a}
	By \eqref{k-quadratic}, we see that $k(\lambda)>0$ for all $\lambda\in\R$ if $r_{\min}>0$. However, as we do not assume sign conditions on the intrinsic growth rates $r_u, r_v$ at this stage, $k(\lambda)$ can possibly take negative values for some $\lambda$.
\end{remark}

{
\begin{prop}[Comparison of the principal eigenvalues \protect{\cite[Theorem 2.1]{Griette-Matano-2025}}]\label{prop:comparison}
Let Assumption \ref{as:coop-comp} hold true. Then the mapping $R\mapsto \lambda_1^R$ is strictly increasing, and $\lambda_1^R<k(\lambda)$ for all $R>0$ and $\lambda\in\R $. Furthermore,
\begin{equation}\label{lambdaR-k}
	\lim_{R\to+\infty}\lambda_1^R=\min_{\lambda\in\R}k(\lambda).
\end{equation}
\end{prop}

From \eqref{lambdaR-k} we see that
\[
\lim_{R\to+\infty}\lambda_1^R=\min_{\lambda\in\R}k(\lambda)
\leq k(0)=\lambda_1^{per},
\]
but the equality does not hold in general.
}

Under some symmetry conditions, $k(\lambda)$ is an even function.
\begin{proposition}[\protect{\cite[Proposition 3]{Griette-Matano-2025}}]\label{prop:symm-case}
	Suppose that Assumption \ref{as:coop-comp} holds true, and assume further that either: 
	\begin{enumerate}[label={\rm(\roman*)}]
	    \item \label{item:symmmatrix}
		$\mu_u(x)=\mu_v(x)$ for all $x\in\R $, or
	    \item \label{item:symmreflexion}
		all coefficients are even: $\sigma(x)=\sigma(-x)$, $r_u(x)=r_u(-x)$, $r_v(x)=r_v(-x)$, $\mu_u(x)=\mu_u(-x)$ and $r_v(x)=r_u(-x)$, for all $x\in\R $.
	\end{enumerate}
	Then the function $\lambda\mapsto k(\lambda)$ is even, \textit{i.e.} $k(\lambda)=k(-\lambda)$ for all $\lambda\in\R $. Consequently, we have
			\begin{equation}\label{k-even}
			    \lim_{R\to+\infty} \lambda_1^R= \min_{\lambda\in\R } k(\lambda) =k(0)
			    =\lambda_1^{per}.
			\end{equation}
\end{proposition}
We remark that $k(\lambda)$ is always an even function in the case of scalar KPP type equations {even if the coefficients are not even functions.} This is a consequence of the Fredholm alternative.

%Before presenting our main results in this section, we remark that nonnegative solutions of \eqref{eq:main-sys} are all bounded as $t\to+\infty$. 
%\begin{prop}[Basic boundedness estimate]\label{prop:uniform-bound}
%Let $(u(t,x),v(t,x))$ be a solution of \eqref{eq:main-sys} with nonnegative bounded initial data $(u_0(x),v_0(x))$. Then 
%$u(t,x)\geq 0$, $v(t,x)\geq 0$ for all $t\geq0, x\in\R$, and
%\begin{equation}\label{u+v<max(K,u0+v0)}
%u(t,x)+v(t,x)\leq \max\big(\overline{K},\, \sup_{x\in\R}(u_0(x)+v_0(x))\big) \ \ \hbox{for all}\ \ t\geq0, x\in\R,
%\end{equation}  
%\begin{equation}\label{uniform-bound}
%\limsup_{t\to+\infty}\sup_{x\in\R}\big(u(t,x)+v(t,x)\big)\leq \overline{K},
%\end{equation}
%where $\overline{K}:=\frac{r_{\max}}{\kappa_{\min}}$.
%In particular, if $u_0(x)+v_0(x)\leq\overline{K}\,(x\in\R)$, then $u(t,x)+v(t,x)\leq\overline{K}\,(t\geq 0,\,x\in\R)$.
%\end{prop}

{
Next we introduce the notion of front-like initial data.
\begin{definition}\label{def:front-like}
The pair of bounded nonnegative functions $(u_0,v_0)$ on $\R $ that appears in \eqref{main-initial} is called {\it right front-like} if there exists a real number $K$ such that
\[
\liminf_{x\to-\infty}\min(u_0(x), v_0(x))>0, \quad u_0(x)=v_0(x)= 0 \ \ \hbox{for all} \ \ x\geq K.
\]
It is called {\it left front-like} if there exists a real number $K$ such that
\[
u_0(x)=v_0(x)= 0 \ \ \hbox{for all} \ \ x\leq K, \quad \liminf_{x\to +\infty}\min(u_0(x), v_0(x))>0. 
\]
\end{definition}
}

\begin{thm}[Spreading speeds for front-like initial data \protect{\cite[Theorem 2.2]{Griette-Matano-2025}}]\label{thm:main-lindet}
	Let Assumption \ref{as:coop-comp} hold true and assume that $\lambda_1^{per}>0$. Then there exist real numbers $c^*_{R}$, $c^*_{L}$ and a positive number $\eta>0$ such that for any solution $(u,v)$ of \eqref{eq:main-sys}--\eqref{main-initial} whose initial data $(u_0,v_0)$ is right front-like, it holds that
\begin{equation}\label{right-spreading}
\begin{cases}
\displaystyle\ \liminf_{t\to\infty} \left[\inf_{x\leq ct}\min(u(t, x), v(t,x))\right] \geq \eta,
\ \ &\text{for all}\ \  c<c^*_{R},\\[9pt]
\displaystyle\ \limsup_{t\to\infty} \left[\sup_{x\geq ct}\max(u(t, x), v(t,x))\right] = 0, \ \ &\text{for all}\ \ c>c^*_{R},
\end{cases}
\end{equation}
while for any solution $(u,v)$ of \eqref{eq:main-sys}--\eqref{main-initial} whose initial data $(u_0,v_0)$ is left front-like, it holds that
\begin{equation}\label{left-spreading}
\begin{cases}
\displaystyle\ \liminf_{t\to\infty} \left[\inf_{x\geq -ct}\min(u(t, x), v(t,x))\right] \geq \eta,
\ \ &\text{for all}\ \  c<c^*_{L},\\[9pt]
\displaystyle\ \limsup_{t\to\infty} \left[\sup_{x\leq -ct}\max(u(t, x), v(t,x))\right] = 0, \ \ &\text{for all}\ \ c>c^*_{L}.
\end{cases}
\end{equation}
Furthermore, we have the following formula:
\begin{equation}\label{eq:speed}
		c^*_{R}=\inf_{\lambda>0} \frac{k(\lambda)}{\lambda}=\min_{\lambda>0} \frac{k(\lambda)}{\lambda}
		\quad\ \ 
		c^*_{L}=\inf_{\lambda<0} \frac{k(\lambda)}{-\lambda} =\min_{\lambda<0} \frac{k(\lambda)}{-\lambda}
		= \min_{\lambda>0} \dfrac{k(-\lambda)}{\lambda},
\end{equation}
where $k(\lambda)$ is the $\lambda$-principal periodic eigenvalue  defined in Definition~\ref{def:k(lambda)}.
\end{thm}

\begin{remark}
	The formula \eqref{eq:speed} can be regarded as the one-dimensional version of the G\"{a}rtner-Freidlin formula for scalar equations \cite{Gertner-Freidlin-1979, Berestycki-Hamel-Nadirashvili-2005}. 
{Note that the constant $\eta$ in \eqref{right-spreading} and \eqref{left-spreading} does not depend on the choice of $(u_0,v_0)$.}
\end{remark}

\begin{definition}[Right- and left spreading speeds]\label{def:spreading-speed}
The above quantities $c^*_{R}$ and $c^*_{L}$ are called the \textit{right spreading speed} and \textit{left spreading speed} of solutions to \eqref{eq:main-sys}, respectively.
\end{definition}

\begin{prop}[\protect{\cite[Proposition 5]{Griette-Matano-2025}}]\label{prop:c*-estimate}
Let $\sigma_{\min}$, $\sigma_{\max}$, $r_{\min}$, $r_{\max}$ be the constants that appear in \eqref{eq:maxmincoeff}. Then 
\[
c^*_{R}\leq 2\sqrt{\sigma_{\max} r_{\max}}, \quad c^*_{L}\leq 2\sqrt{\sigma_{\max} r_{\max}}.
\]
Furthermore, if $r_{\min}> 0$, then 
\[
c^*_{R}\geq 2\sqrt{\sigma_{\min} r_{\min}}, \quad c^*_{L}\geq 2\sqrt{\sigma_{\min} r_{\min}}.
\]
\end{prop}

\begin{remark}
	As mentioned in Remark \ref{rem:251025a}, $k(\lambda)$ can possibly take negative values if $r_{\min}<0$, in which case $c^*_R$ or $c^*_L$ becomes negative. A negative speed implies that the front is retreating. If we assume $\lambda(0) (=\lambda_1^{per})>0$ as in Theorem \ref{thm:main-lindet}, then at least one of $c^*_R, c^*_L$ is positive by the convexity of $\lambda\mapsto k(\lambda)$. 
\end{remark}

If we assume that both $c^*_R$ and $c^*_L$ are positive, then the following theorem holds.

\begin{thm}[Hair-trigger effect \protect{\cite[Theorem 2.3]{Griette-Matano-2025}}]\label{thm:hairtrigger}
	Let Assumption \ref{as:coop-comp} hold true. 
Then the following three conditions are equivalent: 
\[
{\rm (a)}\ \lambda_1^R>0 \ \ \hbox{for some} \ R>0,\quad\ 
{\rm (b)}\ \min_{\lambda\in\R }k(\lambda)>0, \quad\ 
{\rm (c)}\ c^*_{R}>0,\ c^*_{L}>0.
\]
If any of these conditions holds, there exists a number $\eta>0$ depending only on the coefficients of system \eqref{eq:main-sys} such that for any nonnegative bounded initial data $(u_0, v_0)$ satisfying $(u_0(x), v_0(x))\not\equiv (0, 0)$, the solution $\big(u(t, x), v(t, x)\big)$ of \eqref{eq:main-sys}--\eqref{main-initial} has the following property: 
	\begin{equation}\label{eq:hairtrigger-below}
		\liminf_{t\to+\infty}u(t, x)\geq \eta \ \ \text{and}\ \  \liminf_{t\to+\infty} v(t, x)\geq \eta \ \ \text{for all}\ \ x\in\R .
	\end{equation}
Furthermore, if, in addition, $u_0$ and $v_0$ are compactly supported, then the right front and the left front of $(u,v)$ propagate at the speed $c^*_{R}$ and $c^*_{L}$, respectively. More precisely,
\begin{subequations}{\label{right-spreading2}}
	\begin{align}
\label{right2a}
		&\displaystyle\ \liminf_{t\to\infty} \left[\inf_{0\leq x\leq ct}\min(u(t, x), v(t,x))\right] \geq \eta,
\ \ \ \hbox{for all}\ \ 0<c<c^*_{R},\\
\label{right2b}
		&\displaystyle\ \limsup_{t\to\infty} \left[\sup_{x\geq ct}\max(u(t, x), v(t,x))\right] = 0, \ \ \ \ \hbox{for all}\ \ c>c^*_{R}.
	\end{align}
\end{subequations}
\begin{subequations}{\label{left-spreading2}}
	\begin{align}
\label{left2a}
		&\displaystyle\ \liminf_{t\to\infty} \left[\inf_{-ct\leq x\leq 0}\min(u(t, x), v(t,x))\right] \geq \eta,
\ \ \ \hbox{for all}\ \ 0<c<c^*_{L},\\
\label{left2b}
		&\displaystyle\ \limsup_{t\to\infty} \left[\sup_{x\leq -ct}\max(u(t, x), v(t,x))\right] = 0,\ \ \ \ \hbox{for all}\ \ c>c^*_{L}.
	\end{align}
\end{subequations}
The assertions \eqref{right2a} and \eqref{left2a} hold for any nonnegative nontrivial solution of \eqref{eq:main-sys}.
\end{thm}

\begin{prop}[\protect{\cite[Proposition 6]{Griette-Matano-2025}}]\label{prop:right-left-speeds}
Let the assumption (i) or (ii) of Proposition~\ref{prop:symm-case} hold,
	along with the condition $\lambda_1^{per}>0$. Then $c^*_{R}=c^*_{L}>0$. In particular, if all the coefficients are spatially homogeneous, then $c^*_{R}=c^*_{L}>0$. 
\end{prop}

%%%%%%%%%%%%%%
\subsection{Traveling waves}\label{ss:traveling-waves}

Now we are ready to state the main results of the present paper. 
Throughout the rest of this section, we assume, as in Theorem~\ref{thm:main-lindet}, that $\lambda_1^{per}>0$. This condition implies that $(u,v)=(0,0)$ is unstable under positive $L$-periodic perturbations. Without this condition, any solution of \eqref{eq:main-sys} with small nonnegative initial data decays to $0$ as $t\to\infty$ uniformly on $\R$.

First we recall the notion of traveling wave solutions under a spatially periodic environment. We distinguish the ``right traveling waves'' and ``left traveling waves''.

\begin{defn}[Traveling wave solutions]\label{def:TW}
	Let $c>0$ be given and $\big(u(t, x), v(t,x)\big)$ be an entire solution to \eqref{eq:main-sys}, {\it i.e.} a solution that is defined for all $t\in\R$ and $x\in\R$. We say that $\big(u(t, x), v(t, x)\big)$ is a \textit{right traveling wave} propagating at speed $c$ if it satisfies 
	\begin{equation}\label{eq:TW-propagating}
		u\left(t+\frac{L}{c}, x\right) = u(t, x-L), \; v\left(t+\frac{L}{c}, x\right) = v(t, x-L), \ \ \text{for all}\ \ (t,x)\in\R^2, 
	\end{equation}
	as well as the following asymptotics at $x=\pm\infty$:
	\begin{subequations}\label{eq:TW-limits}
	\begin{align}
	    \lim_{x\to+\infty}u(t, x)=0,\; \lim_{x\to+\infty}v(t, x)=0,\ \  \text{for all}\ \  t\in\R, \label{eq:TW-limits+inf} \\ 
		\liminf_{x\to-\infty}u(t, x)>0,\; \liminf_{x\to-\infty}v(t, x)>0,\ \  \text{for all}\ \  t\in\R. \label{eq:TW-limits-inf}
	\end{align}
	\end{subequations}
The {\it left traveling wave} is defined the same way by reversing the direction of the $x$-axis, namely
	\begin{equation}\label{eq:TW-propagating2}
		u\left(t+\frac{L}{c}, x\right) = u(t, x+L), \; v\left(t+\frac{L}{c}, x\right) = v(t, x+L), \ \ \text{for all}\ \ (t,x)\in\R^2, 
	\end{equation}
with reversed asymptotics at $x=\pm\infty$.
\end{defn}

Next we define the notion of right and left stationary wave solutions. Stationary waves, which we may also call ``stationary fronts'', are the natural extension of traveling waves to the case $c=0$. 
%%although they are stationary solutions of \eqref{eq:main-sys}, we keep this terminology for consistency with traveling waves. 

\begin{defn}[Stationary wave solutions]
	Let  $\big(u(t, x), v(t,x)\big)=\big(u(x), v(x)\big)$ be a stationary solution to \eqref{eq:main-sys}, {\it i.e.} a solution that is constant in time. We say that $\big(u(x), v(x)\big)$ is a \textit{right stationary wave} (or a \textit{right stationary front}) if it satisfies 
	 the  asymptotics \eqref{eq:TW-limits} at $x=\pm\infty$.
A {\it left stationary wave} (or a \textit{left stationary front}) is defined the same way by reversing the direction of the $x$-axis, namely, it satisfies the asymptotics
	\begin{subequations}\label{eq:left-stationary-limits}
	\begin{gather}
	    \lim_{x\to-\infty}u(x)=0,\; \lim_{x\to-\infty}v(x)=0,  \label{eq:stationary-limits+inf} \\ 
		\liminf_{x\to+\infty}u(x)>0,\; \liminf_{x\to+\infty}v(x)>0,\ \  \text{for all}\ \  t\in\R. 
	\end{gather}
	\end{subequations}
\end{defn}

Our first main result is concerned with the existence of traveling waves. Here, the symbols $c^*_{R}$ and $c^*_{L}$ denote, respectively, the right and left spreading speeds defined in Definition~\ref{def:spreading-speed}.

\begin{thm}[Existence of traveling waves]\label{thm:TW1}
	Let Assumption \ref{as:coop-comp} hold true and let $\lambda_1^{per}>0$.  Assume that  $c^*_{R}>0$. Then there exists a right traveling wave $\big(u(t, x), v(t, x)\big)$ for \eqref{eq:main-sys} with speed $c$ if and only if $c\geq c^*_{R}$. Similarly, if $c^*_{L}>0$, then there exists a left traveling wave $\big(u(t, x), v(t, x)\big)$ for \eqref{eq:main-sys} with speed $c$ if and only if $c\geq c^*_{L}$.
\end{thm}

\begin{rem}
    The above theorem, together with Theorem~\ref{thm:main-lindet}, shows that the right (resp. left) spreading speeds for solutions of \eqref{eq:main-sys} with front-like initial data coincide with the minimal speed of the right (resp. left) traveling waves provided that $c^*_{R}>0$ (resp. $c^*_{L}>0$). Note also that if  $c^*_{R}>0$ and $c^*_{L}>0$, then these values coincide with the right and left spreading speeds for solutions with compactly supported initial data. 
\end{rem}

\begin{thm}[Existence of stationary and traveling waves]\label{thm:TW1s}
	Let Assumption \ref{as:coop-comp} hold true and let $\lambda_1^{per}>0$. Assume that  $c^*_R=0$. Then there exists a right stationary wave for \eqref{eq:main-sys}. Furthermore, there exists a traveling wave $\big(u(t, x), v(t, x)\big)$ with speed $c$ if and only if $c> c^*_{R}=0$. Similarly, if $c^*_L=0$, then there exists a right stationary wave for \eqref{eq:main-sys}. Furthermore, there exists a traveling wave $\big(u(t, x), v(t, x)\big)$ with speed $c$ if and only if $c> c^*_{L}=0$. 
\end{thm}

\begin{rem}
	If $c^*_R=0$, the propagation of any compactly supported initial data is blocked in the right direction. This is easily seen by using the exponential upper barriers \eqref{eq:upperu} with $\lambda=\lambda^*$ and $c=c^*_R$, which do not move in space,  A similar statement holds true in the left direction if $c^*_L=0$.
\end{rem}

The proof of Theorems \ref{thm:TW1} and \ref{thm:TW1s} involves a fixed-point argument at its core. More precisely, for each speed $c\geq c^*_{R}$, we first restrict the system \eqref{eq:main-sys} on the moving interval $-M+ct\leq x <+\infty$ and prove the existence of a traveling wave for this restricted system by applying the Schauder fixed-point theorem, then we let $M\to\infty$ to obtain a traveling wave of the original system.

In order to apply the fixed-point theorem, we construct upper and lower barriers for our system \eqref{eq:main-sys} that have the same designated decay rate as $x\to\infty$. This can be done despite the absence of the comparison principle for the entire system. The upper barrier is easy to construct, as it is simply a solution of the linearized system. 
The construction of the lower barrier is much more involved. %Inspired by the paper of Hamel \cite{Ham-08}, our method relies on
In the supercritical case $c>c^*_R$ it uses the difference between two exponential functions. This method is inspired by the paper of Hamel \cite{Ham-08} for scalar equations. However, unlike the case of scalar equations, simply taking the maximum between $(0, 0)$ and a sub-solution does not necessarily provide a lower barrier.  Therefore, we need more subtle estimates in the case of systems, particularly when one of the components becomes negative while the other is still positive.
In the critical case $c=c^*_R, c^*_L$ our construction of the lower barrier  is different from that of \cite{Ham-08}.

As mentioned earlier, our existence proof for the critical speed $c=c_R^*, c_L^*$ does not depend on the limiting argument $c\to c_R^*, c_L^*$. Instead, we prove the existence for the critical case by constructing upper and lower barriers that are different from those for the super-critical case.

Figure~\ref{fig:TW} gives a typical image of traveling waves for different parameter values. In all these examples, the intrinsic growth rates satisfy $r_u<r_v$ so that the natural front speed of $v$ is faster than that of $u$, while their carrying capacity (in the absence of the other species and mutation) 
satisfy $r_u/\kappa_u>r_v/\kappa_v$ so that $u$ will eventually dominate $v$ in the long run.

\begin{figure}[H]
    \centering
    \includegraphics[bb = 0 0 118 131]{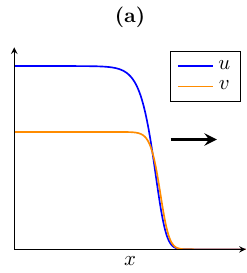} \hspace{1cm}
    \includegraphics[bb = 0 0 118 131]{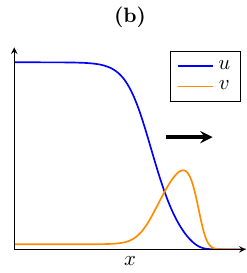}\hspace{1cm}
    \includegraphics[bb = 0 0 118 131]{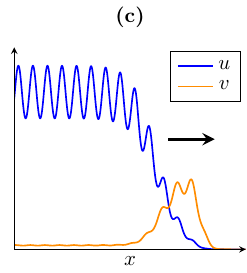}
    \caption{Profiles of traveling waves of \eqref{eq:main-sys} for different parameter values. {\bf (a)} Spatially homogeneous coefficients with large mutation rates $\mu_u$, $\mu_v$. In this case, the cooperative zone of system \eqref{eq:main-sys} is rather large, and the traveling wave lies entirely in this zone. As a result, both $u$ and $v$ have monotone profiles, just as in the case of scalar equations. {\bf (b)} Spatially homogeneous coefficients with small mutation rates $\mu_u$, $\mu_v$. In this case, a large part of the traveling wave profile lies outside the cooperative zone, and a hump appears on $v$. {\bf (c)} Spatially periodic case. 
The coefficients are the same as in (b), except $r_u(x)$ and $r_v(x)$, which have a cosine-like periodic fluctuation.}
    \label{fig:TW}
\end{figure}

The following theorem is concerned with the decay properties of \textit{any} traveling wave along the leading edge.  The proof of this result is rather involved and will be given in Section~\ref{s:decay}.

\begin{thm}[Decay rate of traveling waves]\label{thm:TW2}
	Let the assumptions of Theorem~\ref{thm:TW1} hold.
	\begin{enumerate}
		\item Let $(u, v)$ be  a right traveling wave solution for \eqref{eq:main-sys} with super-critical speed $c>c^*_R$, then there exists $\alpha>0$ such that the following statement holds. 
		    \begin{equation}\label{eq:cv-fronts-super}
				\lim_{x\to+\infty} \sup_{t\in [0, L/c]} \max\left(\left| u(t, x)e^{\lambda_c(x-ct)}-\alpha\varphi^{\lambda_c}(x)\right|,  \left|v(t, x)e^{\lambda_c(x-ct)}-\alpha\psi^{\lambda_c}(x)\right|\right) =0, 
			\end{equation}
			where $\lambda_c$ is the smallest positive root of the equation $k(\lambda)=\lambda c$ {and $(\varphi^{\lambda_c}, \psi^{\lambda_c})$ denotes the principal eigenvector of \eqref{eq:lambda-periodic-principal-eigen} with $\lambda=\lambda_c$.} 
		\item Let $(u, v)$ be  a  right traveling wave solution for \eqref{eq:main-sys} with critical speed $c=c^*_R$, then there exists $\alpha>0$ such that the following statement holds. 
		    \begin{equation}\label{eq:cv-fronts-critical}
				\lim_{x\to+\infty} \sup_{t\in [0, L/c]} \max\left(\left| u(t, x)\frac{e^{\lambda^*(x-c^*_Rt)}}{x-c^*_Rt}-\alpha \varphi^{\lambda^*}(x)\right|,  \left|v(t, x)\frac{e^{\lambda^*(x-c^*_Rt)}}{x-c^*_Rt}-\alpha \psi^{\lambda^*}(x)\right|\right) =0, 
			\end{equation}
			where $\lambda^*$ is the unique positive root of the equation $k(\lambda)=\lambda c^*_R$ {and $(\varphi^{\lambda_c}, \psi^{\lambda_c})$ denotes the principal eigenvector of \eqref{eq:lambda-periodic-principal-eigen} with $\lambda=\lambda^*$}. 
	\end{enumerate}
	Similar statements hold for left traveling waves, if we replace $x-ct$ by $x+ct$ and define $\lambda_c$ (resp. $\lambda^*$) as the largest negative solution of $k(\lambda_c)=-\lambda_c c$ (resp. $k(\lambda^*)=-\lambda^* c^*_L$).
\end{thm}

\begin{remark}
The assertion 1 of Theorem~\ref{thm:TW2} remains true even if $c^*_R=0$, and the proof is exactly the same. The assertion 2 is also true even if  $c^*_R=0$. More precisely, any right stationary wave satisfies \eqref{eq:cv-fronts-critical} with $c^*_R=0$. The proof of this latter result is easier than that of Theorem~\ref{thm:TW2} and follows from the center manifold theory for ODEs. We omit the proof here. The same is true of left traveling waves and left stationary waves if $c^*_L=0$.
\end{remark}

The upper and lower barriers that are used in the proof of Theorem~\ref{thm:TW1} satisfy the same decay estimates 
as above. In other words, in Theorem \ref{thm:TW1}, we constructed traveling waves that satisfy these decay properties. Theorem \ref{thm:TW2} asserts that all the traveling waves indeed have these decay properties, therefore they are universal properties.
%on the constructed traveling waves. However, Theorem \ref{thm:TW2} gives a stronger property as is concerns \textit{any} traveling wave, not just the ones  constructed here. 

Next theorem is concerned with the uniqueness of traveling waves. In the case of scalar equations, the uniqueness has been proved for KPP-type equations in a much more general framework \cite{Hamel-Roques-2011}.
At the moment, we do not know if uniqueness generally holds in our system for each speed $c$. However, under extra conditions on the coefficients, we can prove the uniqueness.

\begin{thm}[Uniqueness of traveling waves]\label{thm:TW3}
Let Assumption \ref{as:coop-comp} hold true and let $\lambda_1^{per}>0$. Assume that
\begin{equation}\label{eq:r/kappa<mu}
\frac{\kappa_{\mathrm{max}}}{\kappa_{\mathrm{min}}}\leq \frac{\mu_{\mathrm{min}}}{r_{\mathrm{max}}} ,
\end{equation}
where $r_{\mathrm{max}}$, $\kappa_{\mathrm{min}}$,  $\kappa_{\mathrm{max}}$, $\mu_{\mathrm{min}}$ are as in \eqref{eq:maxmincoeff}. Then, for each $c> c^*_R$ (resp. $c> c^*_L$), the right (resp. left) traveling wave with speed $c$ is unique up to a shift in time, and $t\mapsto \big(u(t, x), v(t,x)\big)$ is increasing. 
The same holds for $c=c^*_R$ (resp. $c=c^*_L$) if $c^*_R>0$ (resp. $c^*_L>0$).
\end{thm}

We end this subsection with a result on the behavior behind the front. 
The profile of traveling waves of \eqref{eq:main-sys} far behind the front is not well understood except that they are bounded from above and below by positive constants. The difficulty comes from the fact that \eqref{eq:main-sys} is not entirely a cooperative system, therefore the comparison principle holds only partially. However, if the coefficients are spatially homogeneous or rapidly oscillating, we can show that the profile of any traveling wave converges to the unique positive steady state as $x\to-\infty$.

We begin with the spatially homogeneous case. In this case, \eqref{eq:main-sys} is written in the form
\begin{equation}\label{eq:main-sys-homo}
	\begin{system}
		\relax &u_t=\sigma u_{xx}+\big(r_u-\kappa_u (u+v)\big)u+\mu_v v-\mu_u u, \\
		\relax &v_t=\sigma v_{xx}+\big(r_v-\kappa_v (u+v)\big)v+\mu_u u-\mu_v v,
	\end{system}
\end{equation}
and the linearized system \eqref{eq:linearized} is written in the following form:
\[
%%\begin{equation}\label{eq:linearized}
	\begin{system}
		\relax &u_t-\sigma u_{xx}=(r_u-\mu_u)u+\mu_v v, \\
		\relax &v_t-\sigma v_{xx}=(r_v-\mu_v)v+\mu_u u, 
	\end{system}
	\quad t>0, x\in\R.
%%\end{equation}
\]
The coefficient matrix of the right-hand side of the above system is given by
\begin{equation}\label{A}
A:=
\begin{pmatrix} r_u-\mu_u & \mu_v\\ \mu_u & r_v-\mu_v \end{pmatrix}.
\end{equation}
Let $\lambda_A$ denote the principal eigenvalue of $A$, that is, the larger of the two eigenvalues of $A$. Then it is shown in \cite[(2.29)]{Griette-Matano-2025} that 
\[
k(\lambda)=\sigma \lambda^2+\lambda_A.
\]
Consequently, $\lambda_A=\min_{\lambda\in\R}k(\lambda)=k(0)=\lambda_1^{per}$. In particular, if $\lambda_A>0$, then, by \eqref{eq:speed},
\begin{equation}\label{c*-homo}
c^*_R=c^*_L=2\sqrt{\sigma \lambda_A} >0.
\end{equation}
%%Under this assumption, the following theorem holds:

\begin{thm}[Behavior behind the front: the homogeneous case]\label{thm:asymptotic-behavior}
	Let  the coefficients of \eqref{eq:main-sys} are all constants and let Assumption \ref{as:coop-comp} hold true. Assume that $\lambda_A>0$, where $\lambda_A$ denotes the principal eigenvalue of the matrix $A$ in \eqref{A}. Then, there exists a unique positive steady state $(u^*, v^*)$, which is spatially homogeneous. Furthermore, any right traveling wave $\big(u(t, x), v(t, x)\big) =\big(U(x-ct), V(x-ct)\big)$ of \eqref{eq:main-sys} with speed $c\geq c^*_R$ converges to $(u^*, v^*)$ behind the front: 
	\begin{equation}\label{eq:convergence-homogeneous}
		\lim_{\xi\to -\infty} \big(U(\xi), V(\xi)\big) = \big(u^*, v^*\big). 
	\end{equation}
Similary results hold for left traveling waves.
\end{thm}

Next we consider the case where the coefficients are rapidly oscillating. For each $\ep>0$, let the coefficients of \eqref{eq:main-sys} be given in the following form:
\begin{align*}
		r_u^\ep(x) &= r^1_u\left(\frac{x}{\ep}\right), & r_v^\ep(x) &= r^1_v\left(\frac{x}{\ep}\right), & \kappa_u^\ep(x) &= \kappa^1_u\left(\frac{x}{\ep}\right), & \kappa_v^\ep(x) &= \kappa^1_v\left(\frac{x}{\ep}\right),\\
		\mu_u^\ep(x) &= \mu^1_u\left(\frac{x}{\ep}\right), & \mu_v^\ep(x) &= \mu^1_v\left(\frac{x}{\ep}\right), & \sigma(x) &= \sigma^1\left(\frac{x}{\ep}\right),
\end{align*}
where $r^1_u(x)$, $r^1_v(x)$, $\kappa^1_u(x)$, $\kappa^1_v(x)$, $\mu^1_u(x)$, $\mu^1_v(x)$, $\sigma^1(x)$ are 1-periodic functions on $\R$ satisfying Assumption \ref{as:coop-comp}, and we define
\begin{align*}
		\overline{r_u} &= \int_0^1 r^1_u(x)dx, & \overline{r_v} &= \int_0^1 r^1_v(x)dx, & \overline{\kappa_u} &= \int_0^1 \kappa^1_u(x)dx, & \overline{\kappa_v} &= \int_0^1\kappa^1_v(x)dx,\\
		\overline{\mu_u} &= \int_0^1 \mu^1_u(x)dx, & \overline{\mu_v} &= \int_0^1\mu^1_v(x)dx), & \bar{\sigma}^H &= \left(\int_0^1 \frac{1}{\sigma^1(x)}dx\right)^{-1}.
\end{align*}
Then the system \eqref{eq:main-sys} takes the form
\begin{equation}\label{eq:main-sys-ep}
	\begin{system}
		\relax &u_t=\big(\sigma^\ep(x) u_{x}\big)_x+\big(r_u^\ep(x)-\kappa_u^\ep(x)(u+v)\big)u+\mu_v^\ep(x)v-\mu_u(x)u,\\
		\relax &v_t=\big(\sigma^\ep(x) v_{x}\big)_x+\big(r_v^\ep(x)-\kappa_v^\ep(x)(u+v)\big)v+\mu_u^\ep(x)u-\mu_v^\ep(x)v,
	\end{system}
\end{equation}
and the formal homogenization limit of \eqref{eq:main-sys-ep} as $\ep\to 0$ is given by
\begin{equation}\label{eq:homogenized}
	\begin{system}
		\relax &u_t=\bar{\sigma}^H u_{xx}+\big(\overline{r_u}-\overline{\kappa_u}(u+v)\big)u+\overline{\mu_v}v-\overline{\mu_u}u,\\
		\relax &v_t=\bar{\sigma}^H v_{xx}+\big(\overline{r_v}-\overline{\kappa_v}(u+v)\big)v+\overline{\mu_u}u-\overline{\mu_v}v.
	\end{system}
\end{equation}
As in \eqref{A}, we define the coefficient matrix $\overline{A}$ for the linearized system of \eqref{eq:homogenized} by
\begin{equation}\label{A-bar}
\overline{A}:=
\begin{pmatrix} \overline{r_u}-\overline{\mu_u} & \overline{\mu_v}\\ \overline{\mu_u} & \overline{r_v}-\overline{\mu_v} \end{pmatrix}
\end{equation}
and let $\lambda_{\overline{A}}$ denote the principal eigenvalue of $\overline{A}$. If $\lambda_{\overline{A}}>0$, then, by \eqref{c*-homo}, the spreading speeds for the homogenized system \eqref{eq:homogenized} are given by
\begin{equation}\label{c*-homo2}
\overline{c^*_R}=\overline{c^*_L}=2\sqrt{\bar{\sigma}^H \lambda_{\overline{A}}} >0.
\end{equation}

Using Theorem 2.5 of our previous paper \cite{Griette-Matano-2025}, we can show the following:

\begin{thm}[Behavior behind the front: the rapidly oscillating case]\label{thm:asymptotic-behavior2}
	Let  Assumption \ref{as:coop-comp} hold true for the system \eqref{eq:main-sys-ep}, and assume that $\lambda_{\overline{A}}>0$. Let $c^*_{\ep,R}$, $c^*_{\ep,L}$ denote, respectively, the right and the left spreading speeds for \eqref{eq:main-sys-ep}. Then there exists $\bar{\ep}>0$ such that
	\begin{equation}\label{eq:c*-ep}
	c^*_{\ep,R}>0,\ c^*_{\ep,L}>0 \ \ (\hbox{for}\ \  0<\ep<\bar{\ep}), \quad\ 
	\lim_{\ep\to 0}c^*_{\ep,R}=\lim_{\ep\to 0}c^*_{\ep,L}=\overline{c^*_R}\  (=\overline{c^*_L}) =2\sqrt{\bar{\sigma}^H \lambda_{\overline{A}}}
	\end{equation}
	and that, for each $0<\ep<\bar{\ep}$, system \eqref{eq:main-sys-ep} has a unique positive bounded steady state $(u^*_\ep(x),v^*_\ep(x))$, which is $\ep$-periodic,  and any right traveling wave $\big(u(t, x), v(t, x)\big)$ of \eqref{eq:main-sys-ep} with speed $c\geq c^*_{\ep, R}$ converges to $(u^*_\ep(x),v^*_\ep(x))$ behind the front,  in the sense that
	\begin{equation}\label{eq:convergence-rapidosc}
		{\lim_{x\to -\infty}} \big(|u(t, x)-u^*_\ep(x)|+|v(t, x)-v^*_\ep(x)|\big) = 0, \quad \,^\forall t\in\mathbb{R}.
	\end{equation}
	Similar results hold for left traveling waves.
\end{thm}

%%%%%%%%%%%%%%
\subsection{Asymmetric propagation speed and a singular limit problem}\label{ss:asymmetric}

Here we give an example for which the right and left
speeds $c^*_R, c^*_L$ 
are different. 
As we mentioned before, such a situation never occurs in scalar KPP type equations, and also not in system \eqref{eq:main-sys} if the coefficients satisfy certain symmetry conditions (Proposition~\ref{prop:right-left-speeds}). 

We have shown  such an example in a preprint~\cite{Griette-Matano-2021}, which, however, heavily relies on the fact that the diffusion coefficients of the two species are different, and cannot be adapted to the present situation. The example we present here is based on a totally different idea.
We construct the example through an analysis of a certain singular limit problem.

We consider the following system which is a special case of \eqref{eq:main-sys}:
\begin{equation}\label{eq:RD-asym}
\begin{cases}
    \ \partial_t u =\partial_{xx}u + u\big(1-\kappa(u+v)\big) - \mu_1(x)u +\mu_2(x) v\\[2pt]
    \ \partial_t v =\partial_{xx}v + v\big(r-\kappa(u +v)\big) + \mu_1(x)u -\mu_2(x) v,
\end{cases}
\end{equation}
where $\kappa>0, r>0$ are positive constants, and $\mu_1(x), {\mu_2}(x)$ are nonnegative $L$-periodic functions that represent mutation rates and are given by:
\begin{equation}\label{eq:mu-12}
\mu_1(x)=
\begin{cases} \dfrac{\mu}{\ep\delta} & (0\leq x \leq \ep)\\[6pt]
0 & (\ep<x<L)
\end{cases}, \quad
\mu_2(x)=
\begin{cases} \dfrac{\mu}{\ep\delta} & (\delta-\ep\leq x \leq \delta)\\[6pt]
0 & (0\leq x<\delta-\ep,\,\delta<x<L),
\end{cases}
\end{equation}
where $\mu>0, \delta>0$ and $0<\ep<\frac{\delta}{2}$ are fixed constants. Strictly speaking, this parameter choice does not satisfy assumption \ref{as:coop-comp} because $\mu_u(x)$ and $\mu_v(x)$ take the value zero, however, we can recover assumption  \ref{as:coop-comp} by mollifying $\mu_u$ and $\mu_v$ with any positive mollifier, for instance the Gaussian kernel. 
We plot a schematic representation of $\mu_1$ and $\mu_2$ in a unit periodicity cell in Figure \ref{fig:mu-asym}.
\begin{figure}[H]
    \centering
	\includegraphics[bb = 0 0 315 117]{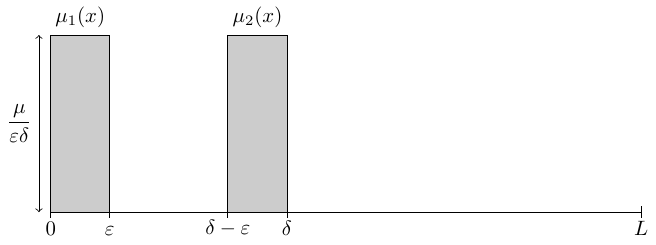}
	\caption{Schematic representation of the shape of the functions $\mu_1(x)$ and $\mu_2(x)$ over one period.}\label{fig:mu-asym}
\end{figure}

Our main result in this subsection is the following:

\begin{theorem}\label{thm:asymmetric-speeds}
	Assume $r>1$ and let $c^*_R(\delta, \ep)$ and $c^*_L(\delta, \ep)$ denote the right and the left spreading speeds of \eqref{eq:RD-asym}--\eqref{eq:mu-12}, which are given by \eqref{eq:speed}. Then, for all sufficiently small $\ep>0, \delta>0$ with $\ep=\mathcal{O}(\delta^2)$, the following inequality holds:
	\begin{equation}\label{eq:speed-asym}
		 c^*_R(\ep, \delta)<c^*_L(\ep, \delta) .
	\end{equation}
\end{theorem}
\begin{remark}
Note that, in the above theorem, the condition $\ep=O(\delta^2)$ does not mean that $\ep$ and $\delta^2$ are comparable. It simply says that $\ep/\delta^2$ remains bounded as $\ep,\delta\to 0$. In particular, \eqref{eq:speed-asym} holds if we fix a sufficiently small value of $\delta>0$ and let $\ep\to 0$.
\end{remark}

To verify the claim \eqref{eq:speed-asym}, we have performed numerical simulations of the evolution of solution fronts of \eqref{eq:RD-asym}--\eqref{eq:mu-12}. 
As it turned out that $\ep$ has to be very small to see an effective difference between the two speeds in \eqref{eq:speed-asym}, we computed the solution of the singular limit of the system \eqref{eq:RD-asym}--\eqref{eq:mu-12} as $\ep\to 0$ instead of computing the original system. 
The result is given in Figure~\ref{fig:numeric-asymm-speed} (upper right), which shows a noticeable difference between the right and left speeds. 

Let us explain this singular limit in more detail. We fix $\delta>0$ in \eqref{eq:mu-12} and let $\ep\to 0$. 
The term $-\mu_1(x)u$ in the first equation of \eqref{eq:RD-asym} represents a big absorption of mass {of the size $(\ep\delta)^{-1}\mu u$ over} the interval $[0,\ep]$ while the term $\mu_2(x)v$ represents a big generation of mass of the size $(\ep\delta)^{-1}\mu v$ over the interval $[\delta-\ep,\delta]$. In the second equation of \eqref{eq:RD-asym}, the roles of absorption and generation are reversed. As $\ep$ tends to $0$, the effect of absorption and that of generation of mass are expressed by a flux gap {of the size $\delta^{-1}\mu u$ or $\delta^{-1}\mu v$} at $x=mL$ and $x=mL+\delta\;(m\in{\mathbb Z})$.  
Thus, at least formally, the limit problem (as $\ep\to 0$) takes the following form:
\begin{subequations}\label{eq:RD-limit}
\begin{equation}\label{eq:RD-limit-eqn}
\begin{cases}
	\ \partial_t u =\partial_{xx}u + u(1-\kappa(u+v))  & x\in \bigcup_{m\in\mathbb{Z}}\left(I_m\cup J_m\right),\;t>0\\[2pt]
	\ \partial_t v =\partial_{xx}v + v(r-\kappa(u+v))  & x\in \bigcup_{m\in\mathbb{Z}}\left(I_m\cup J_m\right),\;t>0,
\end{cases}
\end{equation}
where
\[
I_m=(mL,mL+\delta),\quad J_m=(mL+\delta,(m+1)L),
\]
along with the following boundary conditions for all $m\in \mathbb{Z},\;t>0$, which represent the flux gap:
\begin{equation}\label{eq:RD-limit-BC-u}
\begin{cases}
\ u(mL+0,t)=u(mL-0,t),\\ 
\ u_x(mL+0,t)= u_x(mL-0,t)+\dfrac{\mu}{\delta}\hspace{1pt}u(mL,t),\\
\ u(mL+\delta+0,t)=u(mL+\delta-0,t),\\
\ u_x(mL+\delta+0,t)=u_x(mL+\delta-0,t)-\dfrac{\mu}{\delta}\hspace{1pt}v(mL+\delta,t),
\end{cases}
\end{equation}
\begin{equation}\label{eq:RD-limit-BC-v}
\begin{cases}
\ v(mL+0,t)=v(mL-0,t),\\ 
\ v_x(mL+0,t)= v_x(mL-0,t)-\dfrac{\mu}{\delta}\hspace{1pt}u(mL,t),\\
\ v(mL+\delta+0,t)=v(mL+\delta-0,t),\\
\ v_x(mL+\delta+0,t)=v_x(mL+\delta-0,t)+\dfrac{\mu}{\delta}\hspace{1pt}v(mL+\delta,t).
\end{cases}
\end{equation}
\end{subequations}

\begin{figure}[H]
\centering
	\includegraphics[bb = 0 0 169 130]{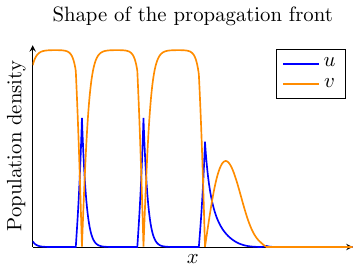}\hspace{1cm} \includegraphics[bb = 0 0 198 134]{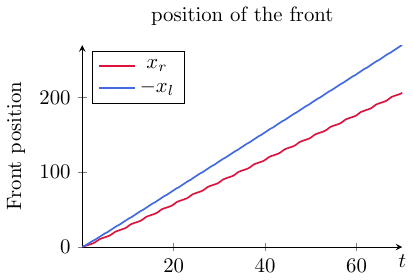} \\
	\includegraphics[bb = 0 0 389.568 137.535]{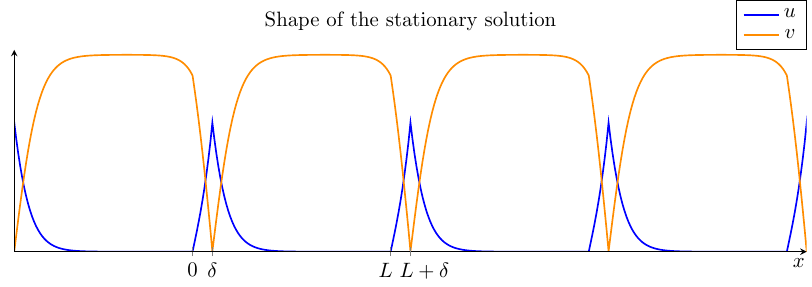}
	\caption{Numerical computation of \eqref{eq:RD-limit} starting from the initial data $u_0(x) = \mathbb{1}_{-0.1\leq x\leq 0.1}$ and $v_0(x)\equiv 0$ for $r=4.0$, $\mu=1000.0$, $\delta = 1.0$, $\kappa=1.0$. 
	\textbf{Upper left figure:} the right-bound front at time $t=100.0$. The graph shows flux gaps at $x=mL$ and $x=mL+\delta$. \textbf{Upper right figure:} the position of the {left-bound} front ($x_l$) and the {right-bound} front ($x_r$) as a function of time, with {$x_l=\inf\{x: u(x)+v(x)\geq 0.1\}$} and $x_r=\sup\{x: u(x)+v(x)\geq 0.1\}$. 
	\textbf{Bottom figure:}
the solution converges to a positive stationary solution far behind the left and the right fronts. The horizontal scale is more stretched than the upper left figure. The positions of flux gaps show clear asymmetry.}
	\label{fig:numeric-asymm-speed}
\end{figure}

The proof of Theorem \ref{thm:asymmetric-speeds} will be given in Section \ref{s:proof-asym-speed}. In view of \eqref{eq:speed}, it suffices to show that $k(\lambda)<k(-\lambda)$ for a certain range of $\lambda>0$. This is a consequence of the asymmetry in the position of flux gaps. We will first establish this inequality for the limit problem as $\ep\to 0$, then show that the same inequality holds for sufficiently small $\ep>0$.

%%%%%%%%%%%%%%%%%%%%%%
%%%%%%%%%%%%%%%%%%%%%%
\section{Existence  and uniqueness of traveling waves}
\label{s:existence}

%%%%%%%%%%%%%%%%%%%%
\subsection{Comparison with a lower barrier}

As already stated, the existence of traveling waves is proved by constructing upper and lower barriers and applying the Schauder fixed point theorem along with a limiting argument. What is most important is to find suitable lower barriers. Since our system \eqref{eq:main-sys} is not entirely cooperative, the comparison principle does not hold for the entire system. Therefore, subsolutions of \eqref{eq:main-sys} may not act as lower barriers. To overcome this difficulty, we consider an auxiliary system \eqref{eq:auxiliary-below} below and construct lower barriers of  \eqref{eq:main-sys} using this system, as we did in our previous paper \cite{Griette-Matano-2025} to study spreading properties.  Note that this system has the same linearization as \eqref{eq:main-sys} along the leading edge. We introduce some key notations. Let 
\begin{equation}\label{K-beta}
    K:=\min\left(\inf_{x\in\R }\frac{\mu_v(x)}{\kappa_u(x)}, \inf_{x\in\R }\frac{\mu_u(x)}{\kappa_v(x)}\right) , \qquad 
    \beta:={2}\,\dfrac{\max\left(\sup_{x\in\R } r_u(x) , \sup_{x\in\R }r_v(x)\right)}{K}, 
\end{equation}
and let $\big(\tilde u(t, x), \tilde v(t, x)\big) $ denote a solution to the auxiliary system: 
	\begin{equation}\label{eq:auxiliary-below}
		\begin{system}
			\relax &\tilde u_t=\big(\sigma(x)\tilde u_{x}\big)_x+\big(r_u(x)-\kappa_u(x)(\tilde u+\tilde v)-\beta \tilde u\big)\tilde u+\mu_v(x)\tilde v-\mu_u(x)\tilde u, & t>0, x\in\R, \\
			\relax &\tilde v_t=\big(\sigma(x) \tilde v_{x}\big)_x+\big(r_v(x)-\kappa_v(x)(\tilde u+\tilde v)-\beta \tilde v\big)\tilde v+\mu_u(x)\tilde u-\mu_v(x)\tilde v, &t>0, x\in\R.
		\end{system}
	\end{equation}

\begin{lem}[Comparison with a lower barrier]\label{lem:comparison-below}                                                              Let Assumption \ref{as:coop-comp} hold true. Let $\tilde u_0(x) $ and $\tilde v_0(x)$ be continuous functions such that 
    \begin{equation}\label{eq:comparison-below-initcond}
		0\leq \tilde u_0(x)\leq\min\big(u_0(x), \frac{1}{2}K\big) \ \ \text{and}\ \  0\leq \tilde v_0(x)\leq \min\big(v_0(x), \frac{1}{2}K\big), 
	\end{equation}
	and let $\big(\tilde u(t, x), \tilde v(t, x)\big) $ be the solution of \eqref{eq:auxiliary-below} starting from $\tilde u(0, x)=\tilde u_0(x)$ and $\tilde v(0, x) = \tilde v_0(x)$. Then for all $t>0$ and $x\in\R $ we have
	\begin{equation}\label{tildeu<u}
		\tilde u(t, x) \leq u(t, x) \ \ \text{and}\ \ \tilde v (t, x) \leq v(t, x).
	\end{equation}
\end{lem}

\begin{proof} 
	We refer to our previous paper \cite[Lemma 9]{Griette-Matano-2025}.
\end{proof}

%%%%%%%%%%%%%%%%%%%%%
\subsection{Proof of existence}
\label{ss:traveling waves-proof}

The goal of this section is to prove Theorems \ref{thm:TW1} and \ref{thm:TW1s} on the existence of traveling waves  and stationary waves. 
We consider only the right traveling waves and stationary waves, as the left ones can be treated precisely the same way. We seek traveling waves that have a precise decay rate when $x\to+\infty$, that is anticipated by the linear analysis near the leading edge. Later, in Theorem \ref{thm:TW2}, we will see that these decay properties are universal for all traveling waves. As we mentioned earlier, the existence will be proved by a fixed-point argument. More precisely, for each speed $c\geq c^*_{R}$, we first restrict the system \eqref{eq:main-sys} on the moving interval $-M+ct\leq x <+\infty$ and prove the existence of a traveling wave for this restricted system by applying the Schauder fixed-point theorem, then we let $M\to\infty$ to obtain a traveling wave of the original system.

In order to apply the fixed-point theorem, we need to construct upper and lower barriers. The upper barrier is easy to construct, as it is simply a solution of the linearized system. 
On the other hand, the construction of the lower barrier is much more involved. Our construction of lower barriers for $c>c^*_{R}$ is partly adapted from \cite{Ham-08} on scalar equations, but the lower barrier for $c=c^*_{R}$ is different from \cite{Ham-08} and is new, even for scalar KPP equations. 

As a matter of fact, it is possible to prove the existence of traveling waves for the critical case $c=c^*_{R}$ by a limiting argument, namely by taking the limit as $c\to c^*_{R}$ of the traveling waves of speed $c>c^*_{R}$. 
Such an approach is more common in the study of scalar equations, but it does not give detailed information about the spatial decay rate of the critical traveling wave as $x\to+\infty$. 
Our approach, on the other hand, gives an optimal decay rate directly for the case $c=c^*_{R}$. 

Now, for each fixed $M\geq 0$ and $c\in\R $, let us consider the following system, which is obtained by restricting the system \eqref{eq:main-sys} on the interval $-M+ct\leq x <+\infty$:
\begin{equation}\label{eq:main-sys-Neumann}
	\begin{system}
		\relax &u_t=\big(\sigma(x) u_{x}\big)_x+\big(r_u(x)-\kappa_u(x)(u+v)\big)u+\mu_v(x)v-\mu_u(x)u, & t>0,\, x>-M+ct, \\
		\relax &v_t=\big(\sigma(x) v_{x}\big)_x+\big(r_v(x)-\kappa_v(x)(u+v)\big)v+\mu_u(x)u-\mu_v(x)v, &t>0,\, x>-M+ct,\\
		& u_x(t, -M+ct) = v_x(t, -M+ct) =0\quad\  (t>0).
	\end{system}
\end{equation}
The global existence and the uniqueness of solutions of system \eqref{eq:main-sys-Neumann} for nonnegative bounded continuous initial data $\big(u_0(x),v_0(x)\big)$ is classical. 
We next consider the following auxiliary system which is obtained by restricting \eqref{eq:auxiliary-below} on the interval $-M+ct\leq x <+\infty$:
\begin{equation}\label{eq:auxiliary-below-Neumann}
	\left\{\begin{aligned}\relax
		&\tilde u_t=\big(\sigma(x)\tilde u_{x}\big)_x+\big(r_u(x)-\kappa_u(x)(\tilde u+\tilde v)-\beta \tilde u\big)\tilde u+\mu_v(x)\tilde v-\mu_u(x)\tilde u, \\
		&\tilde v_t=\big(\sigma(x) \tilde v_{x}\big)_x+\big(r_v(x)-\kappa_v(x)(\tilde u+\tilde v)-\beta \tilde v\big)\tilde v+\mu_u(x)\tilde u-\mu_v(x)\tilde v, \\ 
		& \tilde u_x(t, -M+ct) = \tilde v_x(t, -M+ct) =0 \quad (t>0). 
	\end{aligned}\right.
\end{equation}

Then equivalents of Proposition~\ref{prop:uniform-bound} and Lemma~\ref{lem:comparison-below} hold for this system, with virtually the same proof except for a minor modification at the boundary $x=-M+ct$. 
We state below an equivalent of Lemma~\ref{lem:comparison-below} without proof.

\begin{lem}[Comparison with a lower barrier]\label{lem:comparison-below-Neumann}                                                              Let Assumption \ref{as:coop-comp} hold true and let $M\geq 0$ and $c\in\R $ be given, and let $\overline{K}$ be as in \eqref{eq:maxmincoeff}. 
Let $\big(u(t,x),v(t,x)\big)$ be a solution of \eqref{eq:main-sys-Neumann} whose initial data $\big(u_0(x),v_0(x)\big)$ is bounded, nonnegative and continuous, and let $\tilde u_0(x) $ and $\tilde v_0(x)$ be continuous functions such that 
	\begin{equation}
		0\leq \tilde u_0(x):=\min\big(u_0(x), \frac{1}{2}K\big) \text{ and } 0\leq \tilde v_0(x)\leq \min\big(v_0(x), \frac{1}{2}K\big).
	\end{equation}
Then the solution $\big(\tilde u(t, x), \tilde v(t, x)\big) $ of \eqref{eq:auxiliary-below-Neumann} with initial data $\big(\tilde u_0(x),\tilde v_0(x)\big)$ satisfies
	\begin{equation}
		\tilde u(t, x) \leq u(t, x) \ \ \text{and}\ \ \tilde v (t, x) \leq v(t, x)\ \ \hbox{for all}\ \ t>0,\ x\geq -M+ct.
	\end{equation}
\end{lem}

%Next we construct upper and lower barriers in the form of several lemmas.
The following Proposition gives an upper barrier for the supercritical case.

\begin{proposition}[Upper barrier]
\label{prop:upper-barrier}
	Let Assumption \ref{as:coop-comp} hold true,  $\overline{K}$ be as in \eqref{eq:maxmincoeff}
 and $M\geq 0$ be given. Let $\lambda>0$ and $c\in\R $ be such that $c\geq\frac{k(\lambda)}{\lambda}$. Define
	\begin{equation}\label{eq:upperu}
	    \overline{u}(t,x):= e^{-\lambda(x-ct)}\vect{\varphi}^\lambda(x) \quad \overline{v}(t,x):= e^{-\lambda(x-ct)}\vect{\psi}^\lambda(x),
	\end{equation}
	where $(\varphi^\lambda, \psi^\lambda)$ is the $\lambda$-periodic principal eigenvector of \eqref{eq:lambda-periodic-principal-eigen} that satisfies $\Vert (\vect{\varphi}^\lambda, \psi^\lambda)\Vert_{L^\infty}=1$. 

    If $M$ is sufficiently large, then for any solution $\big(u(t, x), v(t, x)\big)$ of \eqref{eq:main-sys-Neumann} satisfying $u(0, x)\leq \overline{u}(0, x)$, $v(0, x) \leq \overline{v}(0, x)$, $u(0, x) + v(0,x) \leq  \overline{K}$, it holds that 
    \begin{equation*}
	u(t, x)\leq \overline{u}(t, x), \ \ v(t, x) \leq \overline{v}(t, x), \ \  u(t, x)+v(t, x)\leq \overline{K}\ \ \hbox{for all}\ \ t\geq 0,\ x\geq -M+ct. 
    \end{equation*}
\end{proposition}

\begin{proof}
The fact that $u(t, x)+v(t, x)\leq \overline{K}$ is a consequence of (an equivalent of) Proposition~\ref{prop:uniform-bound}, so we prove the former two inequalities. 
    Direct computation show that
    \begin{equation}\label{eq:240125-1}
	\begin{aligned}\relax
	    &\overline{u}_t = \big(\sigma(x)\overline{u}_x\big)_x + \big(r_u(x)-\mu_u(x)+\lambda c-k(\lambda)\big)\overline{u} + \mu_v(x) \overline{v}, \\ 
	    &\overline{v}_t = \big(\sigma(x)\overline{v}_x\big)_x + \big(r_v(x)-\mu_v(x)+\lambda c-k(\lambda)\big)\overline{v} + \mu_u(x) \overline{u}.
	\end{aligned}
    \end{equation}
    Therefore $(\overline{u},\overline{v})$ is a solution of the linear cooperative system \eqref{eq:240125-1}. 
    Since $c\geq\frac{k(\lambda)}{\lambda}$, we have 
    \begin{equation*}
	\begin{aligned}\relax
	    {u}_t& = \big(\sigma(x){u}_x\big)_x + \big(r_u(x)-\mu_u(x)-\kappa_u(x)(u+v)\big){u}+ \mu_v(x) {v} \\ 
	    &\leq \big(\sigma(x){u}_x\big)_x + \big(r_u(x)-\mu_u(x)+\lambda c-k(\lambda)\big){u}+ \mu_v(x) {v}, \\
	    {v}_t &= \big(\sigma(x){v}_x\big)_x + \big(r_v(x)-\mu_v(x)-\kappa_v(x)(u+v)\big){v} + \mu_u(x) {u}\\
	    &\leq  \big(\sigma(x){v}_x\big)_x + \big(r_v(x)-\mu_v(x)+\lambda c-k(\lambda)\big){v} + \mu_u(x) {u}. 
	\end{aligned}
    \end{equation*}
    Hence $(u, v)$ is a subsolution to the system \eqref{eq:240125-1}. 
    Moreover, for $M$ sufficiently large, we have 
    \begin{equation*}
	\overline{u}(t, -M+ct) = e^{\lambda M} \varphi^\lambda(-M+ct) \geq \overline{K}\geq u(t, -M+ct)
    \end{equation*}
    and similarly $\overline{v}(t, -M+ct) \geq v(t, -M+ct). $ 
    Thus by the comparison principle for cooperative parabolic systems, the following inequalities hold, which prove the lemma:
    \begin{equation*}
	u(t, x)\leq \overline{u}(t, x),\ \ v(t, x)\leq \overline{v}(t, x) \quad \text{for all} \ \ t\geq 0,\  x\geq -M+ct.\qedhere
    \end{equation*}
\end{proof}

Before we construct a lower barrier, we prepare the following Lemma.
\begin{lemma}[A differential inequality for $c>c^*_R$]\label{lem:sub-solution}
    Let Assumption \ref{as:coop-comp} hold true and let $c>c^*_{R}$ be given. Let $\lambda>0$ be {the smallest positive solution of} ${\frac{k(\lambda)}{\lambda}=c}$. 
	Define 
	\begin{equation}\label{eq:ul-u}
	    \underline{u}(t, x):=e^{-\lambda(x-ct)}\varphi^\lambda(x)-\omega e^{-\nu(x-ct)}\varphi^\nu(x), \ \ \underline{v}(t, x):=e^{-\lambda(x-ct)}\psi^\lambda(x)-\omega e^{-\nu(x-ct)}\psi^\nu(x), 
	\end{equation}
	where the constant $\nu>0$ satisfies ${k(\nu)-\nu c<0}$ and $\lambda<\nu<2\lambda$. Then there exists $\omega^*>0$ such that, for all $\omega\geq \omega^*$, we have 
	\begin{subequations}\label{eq:sub-sol-beta}
	\begin{align}
	    \underline{u}_t&\leq \big(\sigma(x)\underline{u}_x\big)_x + \big(r_u(x)-\kappa_u(x)(\underline{u}+\underline{v})-\beta \underline{u}\big)\underline{u} + \mu_v(x) \underline{v}-\mu_u(x) \underline{u}, \\
	    \underline{v}_t&\leq \big(\sigma(x)\underline{v}_x\big)_x + \big(r_v(x)-\kappa_v(x)(\underline{u}+\underline{v})-\beta \underline{v}\big)\underline{v} + \mu_u(x) \underline{u}-\mu_v(x)\underline{v}, 
	\end{align}
	\end{subequations}
	whenever 
	\begin{equation}\label{eq:xi_omega}
	    x-ct> \xi_\omega:=\frac{1}{\nu-\lambda}\left(\ln(\omega) + \ln\left[\underset{{y\in\R}}{\inf}\min\left(\frac{\varphi^\nu(y)}{\varphi^\lambda(y)},\frac{\psi^\nu(y)}{\psi^\lambda(y)}\right)\right]\right),
	\end{equation}
	and $\underline{u}(t, x)<0$ and $\underline{v}(t, x)<0$ whenever $x-ct \leq \xi_\omega$.
\end{lemma}

\begin{proof}
    That $\underline{u}(t, x)<0$ and $\underline{v}(t, x)<0$ when $x-ct \leq \xi_\omega$ follows from direct computations with the explicit formula \eqref{eq:ul-u}. 
    Let us prove the differential inequality. We have:
    \begin{equation*}
	\begin{aligned}\relax 
	    &\underline{u}_t=\big(\sigma(x)\underline{u}_x\big)_x + \big(r_u(x)-\mu_u(x)\big)\underline{u} + \mu_v(x) \underline{v} - (-k(\nu)+\nu c)\omega e^{-\nu(x-ct)}\varphi^\nu(x),\\
	    &\underline{v}_t=\big(\sigma(x)\underline{v}_x\big)_x + \big(r_v(x)-\mu_v(x)\big)\underline{v} + \mu_u(x) \underline{u} - (-k(\nu)+\nu c)\omega e^{-\nu(x-ct)}\psi^\nu(x),
	\end{aligned}
    \end{equation*}
    where we recall that $-k(\nu)+\nu c>0$.  Observe that
    \begin{align*}
	\kappa_u(x)(\underline{u}+\underline{v})\underline{u}+\beta \underline{u}^2
	& \leq \big(\sup \kappa_u+\beta\big) \underline{u}^2 + \frac{\kappa_u(x)}{2} \big(\underline{u}^2 + \underline{v}^2\big) \leq \left(\frac{3}{2}\sup \kappa_u + \beta\right)\underline{u}^2 + \frac{1}{2}\sup \kappa_u \underline{v}^2 \\ 
	&\leq \big(2\sup \kappa_u+\beta\big)\max\big(\sup \varphi^\lambda, \sup\psi^\lambda\big)^2 e^{-2\lambda (x-ct)} \\
	&\leq \left[\big(2\sup \kappa_u+\beta\big)\frac{\max\big(\sup \varphi^\lambda, \sup\psi^\lambda\big)^2}{\inf \varphi^\nu} e^{-(2\lambda-\nu)(x-ct)}\right] \varphi^\nu(x)e^{-\nu(x-ct)} \\
	&\leq \frac{C_1}{\omega^{\frac{2\lambda-\nu}{\nu-\lambda}}}  \varphi^\nu(x)e^{-\nu(x-ct)}, 
    \end{align*}
    since $x-ct\geq \xi_\omega$, where the constant $C_1>0$ depends on $\sup\kappa_u, \beta, \lambda, \nu, \varphi^\lambda$ and $\varphi^\nu$. Taking $\omega\geq \left(\frac{C_1}{-k(\nu)+\nu c}\right)^{\frac{2\lambda-\nu}{\nu-\lambda}+1}= \left(\frac{C_1}{-k(\nu)+\nu c}\right)^{\frac{\nu-\lambda}{\lambda}}$ we obtain 
    \begin{equation*}
	\frac{C_1}{\omega^{\frac{2\lambda-\nu}{\nu-\lambda}}}\leq \omega (-k(\nu)+\nu c)
    \end{equation*}and therefore 
    \begin{equation*}
	\kappa_u(x)(\underline{u}+\underline{v})\underline{u}+\beta \underline{u}^2\leq \omega (-k(\nu)+\nu c)\varphi^\nu(x)e^{-\nu(x-ct)}.
    \end{equation*}
    Thus we finally get to 
    \begin{equation*}
	\underline{u}_t\leq \big(\sigma(x)\underline{u}_x\big)_x + \big(r_u(x)-\mu_u(x)-\kappa_u(x)(\underline{u}+\underline{v})-\beta \underline{u}\big)\underline{u} + \mu_v(x) \underline{v} .
    \end{equation*}
    Similarly, 
    there is a constant $C_2>0$ such that for any $\omega\geq \left(\frac{C_2}{-k(\nu)+\nu c}\right)^{\frac{\nu-\lambda}{\lambda}} $ we have 
    \begin{equation*}
	\kappa_v(x)(\underline{u}+\underline{v})\underline{v}+\beta \underline{v}^2\leq \omega (-k(\nu)+\nu c)\psi^\nu(x)e^{-\nu(x-ct)},
    \end{equation*}
    hence 
    \begin{equation*}
	\underline{v}_t\leq \big(\sigma(x)\underline{v}_x\big)_x + \big(r_v(x)-\mu_v(x)-\kappa_v(x)(\underline{u}+\underline{v})-\beta \underline{v}\big)\underline{v} + \mu_u(x) \underline{u} .
    \end{equation*}
    Therefore, \eqref{eq:sub-sol-beta} holds true for $x-ct\geq \xi_\omega$, which finishes the proof of Lemma \ref{lem:sub-solution}. 
\end{proof}

\begin{proposition}[Lower barrier for $c>c^*_{R}$]\label{prop:lower-barrier}
    Let Assumption \ref{as:coop-comp} hold true and let $M\geq 0$, $c>c^*_{R}$ be given. Let $\lambda>0$ be {the smallest positive solution of} ${\frac{k(\lambda)}{\lambda}=c}$ 
    and let $\big(\underline{u}(t,x), \underline{v}(t, x)\big)$ be the defined by \eqref{eq:ul-u},
	where the constant $\nu>0$ satisfies ${k(\nu)-\nu c<0}$ and $\lambda<\nu<2\lambda$. There exists $\omega^*>0$ such that, for all $\omega\geq \omega^*$, we have 
	$\underline{u}(t, -M+ct)<0$ and $\underline{v}(t, -M+ct)<0$ for all $t\geq 0$, and 
	that, for any 
	$n\in\big\{0, \ldots, \left\lfloor\frac{M}{L}\right\rfloor\big\}$ 
	and any nonnegative solution $\big(u(t,x), v(t, x)\big)$ of \eqref{eq:main-sys-Neumann} satisfying
	\begin{equation*}
	    u_0(x)\geq \underline{u}(0, x+nL), \ \  v_0(x)\geq \underline{v}(0,x+nL)\quad 
	    \hbox{for all}\ \ x\geq -M,
	\end{equation*} 
	we have 
	\begin{equation}\label{eq:ulu-super-u}
	    u(t,x)\geq \underline{u}(t, x+nL),\ \ v(t, x)\geq \underline{v}(t, x+nL) \quad \text{for all} \ \ t\geq 0,\  x\geq -M+ct.
	\end{equation}
\end{proposition}

The following proof is adapted from \cite{Ham-08} on scalar KPP type equations.

\begin{proof}
    We first prove the lemma for the special case $n=0$, then discuss the general case later.

    Let us first remark that $\underline{u}(t+\frac{L}{c}, x)=\underline u(t, x-L)$ and $\underline{v}(t+\frac{L}{c}, x)=\underline{v}(t, x-L)$. It follows from direct computations that $\sup_{t, x} \underline{u}$ and $\sup_{t, x}\underline{v}$ are both finite and become arbitrarily small when $\omega\to+\infty$. 
    By Lemma  \ref{lem:comparison-below-Neumann}, it suffices to check that the vector function $(\underline{u}, \underline{v}) $ is a lower barrier to the system \eqref{eq:auxiliary-below-Neumann} if $\omega>0$ is sufficiently large.     We assume that 
    \begin{equation*}
	\omega>\exp\left(-\ln\left[\underset{{y\in\R}}{\inf}\min\left(\frac{\varphi^\nu(y)}{\varphi^\lambda(y)},\frac{\psi^\nu(y)}{\psi^\lambda(y)}\right)\right]\right) = \left(\underset{{y\in\R}}{\inf}\min\left(\frac{\varphi^\nu(y)}{\varphi^\lambda(y)},\frac{\psi^\nu(y)}{\psi^\lambda(y)}\right)\right)^{-1}
    \end{equation*}
    so that $\xi_{\omega}\geq 0$ and in particular $\underline{u}(t, x)<0$ and $\underline{v}(t, x)<0$ at the boundary $x=ct-M$.    The differential inequality for $x-ct>\xi_\omega$ follows from Lemma \ref{lem:sub-solution} and the lemma is proved for $n=0$. 
    
    Precisely the same argument holds for $(\underline{u}(t, x+nL), \underline{v}(t, x+nL)\big)$ since these functions satisfy the same differential inequalities as the case $n=0$ and since $\underline{u}(t, x+nL)<0$ and $\underline{v}(t, x+nL)<0$ at the boundary $x=-M+ct$ for all $n=1, \ldots, \left\lfloor\frac{M}{L}\right\rfloor$. This latter claim holds because $\underline{u}(t, x)<0$ and $\underline{v}(t, x)<0$ whenever $x-ct<0$ by our choice of $\omega$. This completes the proof of Lemma \ref{prop:lower-barrier}.
\end{proof}

To construct upper and lower barriers for the critical speed $c=c^*_R$, we need a different approach, since the decay rate of the traveling wave as $x\to+\infty$ is not purely exponential. 
Let $\lambda^*>0$ be the unique positive  solution of ${\frac{k(\lambda)}{\lambda}=c^*_{R}}$, and define
\begin{equation}\label{eq:critical-aux}
	U(t, x):=-\frac{\partial}{\partial \lambda}\big(e^{-\lambda(x-c^*_{R}t)}\varphi^\lambda(x)\big)\big|_{\lambda=\lambda^*}, \quad 
	    V(t, x):=-\frac{\partial}{\partial \lambda}\big(e^{-\lambda(x-c^*_{R}t)}\varphi^\lambda(x)\big)\big|_{\lambda=\lambda^*}, 
	\end{equation}
	and, for positive constants $\alpha>0$,  $\omega>0$ and $\lambda^*<\nu<2\lambda^*$, 
	\begin{equation}\label{eq:critical-super}
	\begin{split}
	    &\overline{u}^*(t, x) := U(t, x)+\alpha e^{-\lambda^*(x-c^*_{R}t)}\varphi^{\lambda^*}(x),\\
	    &\overline{v}^*(t, x) := V(t, x)+\alpha e^{-\lambda^* (x-c^*_{R}t)}\psi^{\lambda^*}(x),
	\end{split}
	\end{equation}
	\begin{equation}\label{eq:critical-sub}
	\begin{split}
	    &\underline{u}^*(t, x) := U(t, x)+e^{-\nu(x-c^*_{R}t)}\varphi^{\nu}(x)-\omega e^{-\lambda^*(x-c^*_{R}t)}\varphi^{\lambda^*}(x),\\ 
	    &\underline{v}^*(t, x) := V(t, x)+e^{-\nu(x-c^*_{R}t)}\psi^{\nu}(x)-\omega e^{-\lambda^*(x-c^*_{R}t)}\psi^{\lambda^*}(x). 
	\end{split}
\end{equation}
We will show that these functions play the role of an upper and a lower barrier for the critical case.
\begin{lemma}[A differential inequality for $c=c^*_R$]\label{lem:critical-diff-ineq}
	Let Assumption \ref{as:coop-comp} hold true and  $\lambda^*>0$ be {the unique positive solution of} ${\frac{k(\lambda)}{\lambda}=c^*_R}$. Let $\nu\in (\lambda^*, 2\lambda^*)$ be given  and $\alpha>0$ be fixed.
	\begin{enumerate}
		\item The vector function $\big(\overline{u}^*, \overline{v}^*\big)$ defined by \eqref{eq:critical-super} satisfies   
			\begin{subequations}\label{eq:critical-solution-cooperative}
				\begin{align}
					\overline{u}^*_t &= \big(\sigma(x)\overline{u}^*_x\big)_x + r_u(x)\overline{u}^*+\mu_v(x)\overline{v}^* - \mu_u(x) \overline{u}^*, \\
					\overline{v}^*_t &= \big(\sigma(x)\overline{v}^*_x\big)_x + r_v(x)\overline{v}^*+\mu_u(x)\overline{u}^* - \mu_v(x)\overline{v}^*.
				\end{align}
			\end{subequations}

		\item Then there exists $\omega^*>0$ such that, for all $\omega\geq \omega^*$, the vector function $\big(\underline{u}^*, \underline{v}^*\big)$ defined by \eqref{eq:critical-sub} satisfies the differential inequality
			\begin{subequations}\label{eq:critical-sub-ineq}
				\begin{align}
					\underline{u}^*_t&\leq \big(\sigma(x)\underline{u}^*_x\big)_x + \big(r_u(x)-\kappa_u(x)(\underline{u}^*+\underline{v}^*)-\beta \underline{u}^*\big)\underline{u}^* + \mu_v(x) \underline{v}^*-\mu_u(x) \underline{u}^*, \\
					\underline{v}^*_t&\leq \big(\sigma(x)\underline{v}^*_x\big)_x + \big(r_v(x)-\kappa_v(x)(\underline{u}^*+\underline{v}^*)-\beta \underline{v}^*\big)\underline{v}^* + \mu_u(x) \underline{u}^*-\mu_v(x)\underline{v}^*, 
				\end{align}
			\end{subequations}
	whenever 
	\begin{equation}\label{eq:xi_omega^*}
		x-c^*_Rt> \xi_\omega^*:=\omega+\inf_{y\in\mathbb{R}}\left[\frac{\partial\varphi^{\lambda}|_{\lambda^*}(y)}{\varphi^{\lambda^*}(y)}-\frac{\varphi^{\nu}(y)}{\varphi^{\lambda^*}(y)}\right]
	\end{equation}
	and $\underline{u}^*(t, x)<0$ and $\underline{v}^*(t, x)<0$ whenever $ 0\leq x-c^*_Rt \leq  \xi_\omega^*$.
	\end{enumerate}
\end{lemma}
\begin{proof}
	To obtain equality \eqref{eq:critical-solution-cooperative}, we first
    differentiate \eqref{eq:240125-1} with respect to $\lambda$ and evaluate at $\lambda=\lambda^*$,  which yields 
    \begin{equation}\label{eq:240125-3}
	\begin{aligned}\relax
	    {U}_t& = \big(\sigma(x){U}_x\big)_x + \big(r_u(x)-\mu_u(x)+c^*_{R}-k'(\lambda^*)\big){U} + \mu_v(x) {V} \\ 
	    &= \big(\sigma(x){U}_x\big)_x + \big(r_u(x)-\mu_u(x)\big){U} + \mu_v(x) {V}, \\
	    {V}_t  &= \big(\sigma(x){V}_x\big)_x + \big(r_v(x)-\mu_v(x)+ c^*_{R}-k'(\lambda^*)\big){V} + \mu_u(x) {U} \\
	    &=\big(\sigma(x){V}_x\big)_x + \big(r_v(x)-\mu_v(x)\big){V} + \mu_u(x) {U}.
	\end{aligned}
    \end{equation}
    Indeed $\lambda^*$ is the minimizer of $\lambda\mapsto\frac{k(\lambda)}{\lambda}$, and therefore  $c^*_{R}=k'(\lambda^*)$. 
	Thus, $(U,V)$ is a solution of the cooperative system \eqref{eq:critical-solution-cooperative}, and clearly $e^{\lambda^*(x-c^*_Rt)}\big(\varphi^{\lambda^*}(x), \psi^{\lambda^*}(x)\big)$ is also a solution of \eqref{eq:critical-solution-cooperative},  so $(\overline{u}^*,\overline{v}^*)$ solves \eqref{eq:critical-solution-cooperative} as well. 
	\medskip

	We turn to \eqref{eq:critical-sub-ineq}.
    Direct computations show
    \begin{equation*}
	\begin{aligned}\relax 
	    &\underline{u}^*_t=\big(\sigma(x)\underline{u}^*_x\big)_x + \big(r_u(x)-\mu_u(x)\big)\underline{u}^* + \mu_v(x) \underline{v}^* - (k(\nu)-\nu c^*_{R}) e^{-\nu(x-c^*_{R}t)}\varphi^\nu(x),\\
	    &\underline{v}^*_t=\big(\sigma(x)\underline{v}^*_x\big)_x + \big(r_v(x)-\mu_v(x)\big)\underline{v}^* + \mu_u(x) \underline{u}^* - (k(\nu)-\nu c^*_{R}) e^{-\nu(x-c^*_{R}t)}\psi^\nu(x),
	\end{aligned}
    \end{equation*}
    where we recall that $k(\nu)-\nu c^*_{R}>0$ since $\nu>\lambda^*$.

Now we show that, for $x-c^*_{R}t$ sufficiently large, the term $e^{-\nu(x-c^*_{R}t)}$ dominates the quadratic nonlinearities $(\underline{u}^*)^2$, $(\underline{v}^*)^2$ and   $\underline{u}^*\underline{v}^*$ which behave like $(x-c^*_{R}t)^2e^{-2\lambda^*(x-c^*_{R}t)}$. 
    Once this is shown, we see that $\big(\underline{u}^*,\underline{v}^*\big)$ is a 
    subsolution to \eqref{eq:auxiliary-below-Neumann} 
    for carefully chosen parameters.

    We remark that the leading term in $U(t, x)$ near $x=+\infty$ is $(x-c^*_{R}t)e^{-\lambda^*(x-c^*_{R}t)}$. 
    Direct computations show
    \begin{equation*}
	\begin{aligned}\relax 
	    &\underline{u}^*_t=\big(\sigma(x)\underline{u}^*_x\big)_x + \big(r_u(x)-\mu_u(x)\big)\underline{u}^* + \mu_v(x) \underline{v}^* - (k(\nu)-\nu c^*_{R}) e^{-\nu(x-c^*_{R}t)}\varphi^\nu(x),\\
	    &\underline{v}^*_t=\big(\sigma(x)\underline{v}^*_x\big)_x + \big(r_v(x)-\mu_v(x)\big)\underline{v}^* + \mu_u(x) \underline{u}^* - (k(\nu)-\nu c^*_{R}) e^{-\nu(x-c^*_{R}t)}\psi^\nu(x),
	\end{aligned}
    \end{equation*}
    where we recall that $k(\nu)-\nu c^*_{R}>0$.

    Computing the differential in $\lambda$,  $U(t, x)$ can be written as
	\begin{equation}
		U(t, x) = \left(1-\frac{1}{x-c^*_Rt}\frac{\partial_\lambda \varphi^{\lambda}|_{\lambda^*} (x)}{\varphi^{\lambda^*}(x)}\right) (x-c^*_R t)e^{-\lambda^*(x-c^*_R t)}\varphi^{\lambda^*}(x), 
	\end{equation}
	and similarly 
	\begin{equation}
		V(t, x) = \left(1-\frac{1}{x-c^*_Rt}\frac{\partial_\lambda \psi^{\lambda}|_{\lambda^*} (x)}{\psi^{\lambda^*}(x)}\right) (x-c^*_R t)e^{-\lambda^*(x-c^*_R t)}\psi^{\lambda^*}(x), 
	\end{equation}
	whenever $x-c^*_Rt>0$. Thus for $x-c^*_Rt\geq \xi_0:=\max\left(2\underset{x\in\mathbb{R}}{\sup}\frac{\partial_\lambda \varphi^{\lambda}|_{\lambda^*} (x)}{\varphi^{\lambda^*}(x)},-\underset{x\in\mathbb{R}}{\inf} \frac{\partial_\lambda \varphi^{\lambda}|_{\lambda^*} (x)}{\varphi^{\lambda^*}(x)}\right)$
	we have
    \begin{equation*}
	\frac{1}{2}(x-c^*_{R}t)e^{-\lambda^*(x-c^*_{R}t)}\varphi^{\lambda^*}(x)\leq U(t, x)\leq 2(x-c^*_{R}t)e^{-\lambda^*(x-c^*_{R}t)}\varphi^{\lambda^*}(x),
    \end{equation*}
    and similarly for $x-c^*_Rt\geq \xi_1:=\max\left(2\underset{x\in\mathbb{R}}{\sup}\frac{\partial_\lambda \psi^{\lambda}|_{\lambda^*} (x)}{\psi^{\lambda^*}(x)},-\underset{x\in\mathbb{R}}{\inf} \frac{\partial_\lambda \psi^{\lambda}|_{\lambda^*} (x)}{\psi^{\lambda^*}(x)}\right)$ we obtain
    \begin{equation*}
	\frac{1}{2}(x-c^*_{R}t)e^{-\lambda^*(x-c^*_{R}t)}\psi^{\lambda^*}(x)\leq V(t, x)\leq 2(x-c^*_{R}t)e^{-\lambda^*(x-c^*_{R}t)}\psi^{\lambda^*}(x).
    \end{equation*}
    Note that we have 
    \begin{equation}\label{u*v*-boundary}
	    \underline{u}^*(t, x)<0, \quad \underline{v}^*(t,x)<0 \quad \ \text{ whenever } \ \ 0\leq x-c^*_Rt\leq \xi_\omega^*,
    \end{equation}
    where $\xi_\omega^*$ is defined by \eqref{eq:xi_omega^*}, and that $\xi_\omega^*\geq \max(\xi_0, \xi_1)$ whenever 
    \begin{equation*}
	    \omega\geq \omega_0:=\max(\xi_0, \xi_1) + \sup_{x\in\mathbb{R}}\max\left(\frac{\varphi^{\nu}(x)-\partial_\lambda\varphi^{\lambda}|{\lambda^*}(x)}{\varphi^{\lambda^*}(x)},
	    \frac{\psi^{\nu}(x)-\partial_\lambda\psi^{\lambda}|{\lambda^*}(x)}{\psi^{\lambda^*}(x)}\right). 
    \end{equation*}
    By developing the square of \eqref{eq:critical-sub}, we obtain for $\xi\geq\max(\xi_0, \xi_1)$:
    \begin{align*}
	\underline{u}^*(t, x)^2 &\leq 3U(t, x)^2+3 e^{-2\nu(x-c^*_{R}t)}\varphi^\nu(x)^2 + 3\omega^2e^{-2\lambda^*(x-c^*_{R}t)}\varphi^{\lambda^*}(x)^2  \\ 
	    &\leq \big(12\varphi^{\lambda^*}(x)^2(x-c^*_R t)^2 + 3\omega^2\varphi^{\lambda^*}(x)^2  + 3\varphi^\nu(x)^2\big) e^{-2\lambda^*(x-c^*_{R}t)}, 
    \end{align*}
       and similarly 
    \begin{equation*}
	\underline{v}^*(t, x)^2 \leq 
	    \big(12\psi^{\lambda^*}(x)^2(x-c^*_R t)^2 + 3\omega^2\psi^{\lambda^*}(x)^2  + 3\psi^\nu(x)^2\big) e^{-2\lambda^*(x-c^*_{R}t)} .
    \end{equation*}
    Thus 
    \begin{align}
	\nonumber
	    \kappa_u(x)\big(\underline{u}^*+\underline{v}^*)\underline{u}^*&+\beta (\underline{u}^*)^2 
	\leq  \left(\frac{3}{2}\sup \kappa_u + \beta \right)(\underline{u}^*)^2 +\frac{1}{2}\sup \kappa_u(\underline{v}^*)^2 \\
	\nonumber
	    &\leq \big(C_0 (x-c^*_{R}t)^2+C_1 \omega^2+C_2\big) e^{-2\lambda^*(x-c^*_{R}t)}\\
	    \nonumber
	    &\leq \big(\tilde{C}_0 (x-c^*_{R}t)^2e^{-2(\lambda^*-\nu)(x-c^*_{R}t)}+\tilde{C}_1 \omega^2e^{-2(\lambda^*-\nu)(x-c^*_{R}t)}+\tilde{C}_2e^{-2(\lambda^*-\nu)(x-c^*_{R}t)}\big) \\
	    &\quad \times \big(k(\nu)-\nu c^*_{R}\big)e^{-\nu(x-c^*_{R}t)}\varphi^\nu(x),\label{eq:250923a}
    \end{align}
where the constants $\tilde{C}_0$, $\tilde{C}_1$ and $\tilde{C}_2$ are independent of $\omega$. In the right-hand side of the last inequality \eqref{eq:250923a}, the three terms can be controlled when $\omega$ is large:\\
\noindent$\bullet\quad  \tilde{C}_0 (x-c^*_{R}t)^2e^{-2(\lambda^*-\nu)(x-c^*_{R}t)}$: clearly $\xi^2 e^{-2(\lambda^*-\nu)\xi}\to 0 $ when $\xi\to \infty$, so there exists $\xi_3$ with the property that $\xi^2 e^{-2(\lambda^*-\nu)\xi}\leq \frac{1}{3\tilde{C}_0}$ for any $\xi\geq \xi_3$. Set $\omega_1$ such that $\xi_\omega^*\geq \xi_3$ for any $\omega\geq \omega_1$; then when $\omega\geq \omega_1$ we have
\begin{equation}\label{eq:250923b}
	\tilde{C}_0 (x-c^*_{R}t)^2e^{-2(\lambda^*-\nu)(x-c^*_{R}t)} \leq \frac{1}{3}, \qquad \,^\forall x-c^*_Rt \geq \xi_\omega^*. 
\end{equation}
\noindent$\bullet\quad  \tilde{C}_1 \omega^2e^{-2(\lambda^*-\nu)(x-c^*_{R}t)}$: Because of $\xi_\omega^*\geq \omega - C_1'$, we have $\tilde{C}_1 \omega^2e^{-2(\lambda^*-\nu)(x-c^*_{R}t)}\leq \tilde{C}_1 e^{-2(\lambda^*-\nu)(\omega-C_1')}$. Therefore for any $\omega\geq \omega_2:= C_1'+\frac{1}{2\lambda-\nu}\ln(3\tilde{C}_1)$ we have
\begin{equation}\label{eq:250923c}
	\tilde{C}_1 \omega^2e^{-2(\lambda^*-\nu)(x-c^*_{R}t)}\leq \frac{1}{3}, \qquad \,^\forall x-c^*_Rt \geq \xi_\omega^*. 
\end{equation}
\noindent$\bullet\quad \tilde{C}_2e^{-2(\lambda^*-\nu)(x-c^*_{R}t)}$: since, again,  $\xi_\omega^*\geq \omega - C_1'$, we have $\tilde{C}_2e^{-2(\lambda^*-\nu)(x-c^*_{R}t)}\leq \tilde{C}_2 e^{-2(\lambda^*-\nu)(\omega-C_1')}$. Therefore for any $\omega\geq \omega_3:= C_1'+\frac{1}{2\lambda-\nu}\ln(3\tilde{C}_2)$ we have
\begin{equation}\label{eq:250923d}
	\tilde{C}_2 e^{-2(\lambda^*-\nu)(x-c^*_{R}t)}\leq \frac{1}{3}, \qquad \,^\forall x-c^*_Rt \geq \xi_\omega^*. 
\end{equation}

Gathering \eqref{eq:250923a}, \eqref{eq:250923b}, \eqref{eq:250923c} and \eqref{eq:250923d},  we find that for any $\omega\geq \max(\omega_0, \omega_1, \omega_2, \omega_3)$ we have
    \begin{equation*}
	\kappa_u(x)\big(\underline{u}^*+\underline{v}^*)\underline{u}^*+\beta (\underline{u}^*)^2 \leq \big(k(\nu)-\nu c^*_{R}\big)e^{-\nu(x-c^*_{R}t)}\varphi^\nu(x) \ \ \text{for all}\  x-c^*_{R} t\geq \xi_\omega^*.
    \end{equation*}
    By a similar argument, we find that, if $\omega\geq \omega_4$, we have 
    \begin{equation*}
	\kappa_v(x)\big(\underline{u}^*+\underline{v}^*)\underline{v}^*+\beta (\underline{v}^*)^2 \leq \big(k(\nu)-\nu c^*_{R}\big)e^{-\nu(x-c^*_{R}t)}\psi^\nu(x) \ \ \text{for all}\  x-c^*_{R} t\geq \xi_\omega^*.
    \end{equation*}
We have checked that \eqref{eq:critical-sub-ineq} holds for $\omega\geq\omega^*:=\max(\omega_0, \omega_1, \omega_2, \omega_3, \omega_4)$. This finishes the proof of Lemma \ref{lem:critical-diff-ineq}.
\end{proof}

\begin{proposition}[Lower and upper barrier for $c=c^*_{R}$]\label{prop:critical}
    Let Assumption \ref{as:coop-comp} hold, let $\overline{K}$ be as in \eqref{eq:maxmincoeff} and let $M\geq 0$ be given. 
    Then there exist $\alpha^*>0$, $\omega>0$ and $\xi_0>0$ such that
    \[
    \underline{u}^*(t, \xi_0+ct)<0\ \ \hbox{and}\ \ \underline{v}^*(t, \xi_0+ct)<0,\quad \hbox{for all} \ \ t\geq 0,
    \] 
    and that, for any $\alpha>\alpha^*$, any $n\in\big\{0, \ldots, \left\lfloor\frac{M}{L}\right\rfloor\big\}$ and any solution $\big(u(t,x), v(t, x)\big)$ of \eqref{eq:main-sys-Neumann} satisfying
	\begin{subequations}\label{eq:critical-initial}
	    \begin{equation}\label{eq:critical-initial1}
		u_0(x)\geq 0,\ \  v_0(x)\geq 0\ \ \text{and} \ \ u_0(x)+v_0(x)\leq \overline{K}, \ \ \ {}^\forall x \geq -M, 
	    \end{equation}
	    \begin{equation}\label{eq:critical-initial2}
		u_0(x)\leq \overline{u}^*(0, x), \ \  v_0(x)\leq \overline{v}^*(0, x), \ \ \ {}^\forall x\geq \xi_0,
	    \end{equation}
	    \begin{equation}\label{eq:critical-initial3}
	    u_0(x)\geq\underline{u}^*(0, x+nL),\quad v_0(x)\geq\underline{v}^*(0, x+nL),   \ \ \ {}^\forall x\geq \xi_0 -nL,
	    \end{equation}
	\end{subequations} 
	it holds that 
	\begin{subequations}\label{eq:critical-sub-super}
	    \begin{equation}\label{eq:critical-sub-super-xgeq-M}
		u(t, x)\geq 0,\ \ v(t, x)\geq 0\ \ \text{and} \ \ u(t, x)+v(t, x)\leq \overline{K}, \ \   {}^\forall t\geq 0 \text{ and } {}^\forall x\geq -M+c^*_R t, 
	    \end{equation}
	    \begin{equation}\label{eq:critical-sub-super-super}
	    u(t, x)\leq \overline{u}^*(t, x), \ \  v(t, x)\leq \overline{v}^*(t, x), 
		  ~  {}^\forall t\geq 0 \text{ and } {}^\forall x\geq \xi_0+c^*_R t, 
	    \end{equation}
	    \begin{equation}\label{eq:critical-sub-super-sub}
		u(t,x)\geq \underline{u}^*(t, x+nL),\ \  v(t,x)\geq\underline{v}^*(t, x+nL), 
		  ~  {}^\forall t\geq 0 \text{ and } {}^\forall x\geq \xi_0-nL+c^*_R t. 
	    \end{equation}
	\end{subequations}
\end{proposition}
	
\begin{proof}
   The claim \eqref{eq:critical-sub-super-xgeq-M} is a direct consequence of (a variant of) Proposition~\ref{prop:uniform-bound} and \eqref{eq:critical-initial1}. We next  
	consider \eqref{eq:critical-sub-super-super}. It follows from  Lemma \ref{lem:critical-diff-ineq} that $(\overline{u}^*,\overline{v}^*)$
     is a solution of the cooperative system \eqref{eq:critical-solution-cooperative}. On the other hand, $(u,v)$ is a subsolution of \eqref{eq:240125-1}. 
    Furthermore, for $\alpha$ sufficiently large we have  
    \begin{equation*}
	\overline{u}^*(t, \xi_0+c^*_{R}t) =\big((\xi_0+\alpha)\varphi^{\lambda^*}(\xi_0+c^*_{R}t)-\partial_\lambda \varphi^{\lambda^*}(\xi_0+c^*_{R}t)\big)e^{-\lambda^* \xi_0} \geq \overline{K} \geq u(t, \xi_0+c^*_Rt),  
    \end{equation*}
    and similarly $\overline{v}^*(t, \xi_0+c^*_{R}t)\geq \overline{K}\geq v(t, \xi_0+c^*_Rt)$.
    This, together with the inequality \eqref{eq:critical-initial2} and the comparison principle for the linear cooperative system \eqref{eq:240125-1}, proves \eqref{eq:critical-sub-super-super}.
   \medskip

    Next we prove \eqref{eq:critical-sub-super-sub} for $n=0$.
    By Lemma  \ref{lem:comparison-below-Neumann}, 
    in order to show that 
    $(\underline{u}, \underline{v}) $ is a 
    lower barrier for $(u, v)$, 
    it suffices to verify that $(\underline{u}, \underline{v}) $ is a subsolution of \eqref{eq:auxiliary-below-Neumann} in the range $x\geq \xi_0+c^*_R t$ and that $(u,v)\geq (\underline{u}, \underline{v})$ at $x=\xi_0+c^*_R t$ 
    for sufficiently large $\omega>0$ and $\xi_0$. This is exactly the conclusion of Lemma \ref{lem:critical-diff-ineq}, with $\omega>\omega^*$ and  $\xi_0=\xi_\omega^*$ defined by \eqref{eq:xi_omega^*}.
    The Lemma is proved for $n=0$.

    Precisely the same argument holds for $(\underline{u}^*(t, x+nL), \underline{v}^*(t, x+nL)\big)$ 
    with $\xi_0$ replaced by $\xi_0-nL$,
    since it satisfies the same differential inequalities as the case $n=0$ and since $\underline{u}^*(t, x+nL)<0$ and $\underline{v}^*(t, x+nL)<0$ at the boundary $x=\xi_0-nL+c^*_Rt$ for all $n=1, \ldots, \left\lfloor\frac{M}{L}\right\rfloor$. This latter claim holds because, by our choice of $\omega=\omega(\xi_0)$, we have
    $\underline{u}(t, x)<0$ and $\underline{v}(t, x)<0$ whenever $x-c^*_Rt=\xi_0$; moreover $\xi_0>0$, therefore $\xi_0-nL\geq -nL \geq -M$. This completes the proof of Lemma \ref{prop:critical}. 
    \end{proof}

\begin{proof}[Proof of Theorem \ref{thm:TW1}]
	Let $c^*_R> 0$ be the right spreading speed defined in \eqref{eq:speed}.
    The fact that there exists no traveling wave for $c<c^*_{R}$ is a direct consequence of the spreading property (Theorem~\ref{thm:main-lindet}). The existence of a traveling wave for $c\geq c^*_{R}$ will be shown by using the Schauder fixed-point Theorem and a limiting argument.
    
	We begin with the case $c> c^*_{R}$. Fix $c> c^*_{R}$, and let $\lambda$ be the smallest positive solution to ${\lambda c=k(\lambda)}$. We let $\overline{u}(t,x) $ and $\overline{v}(t,x) $ be the functions defined in Lemma \ref{prop:upper-barrier} by \eqref{eq:upperu} and $\underline{u}(t,x)$ and  $\underline{v}(t,x)$ be the function defined in Lemma \ref{prop:lower-barrier} by \eqref{eq:ul-u}.

	For $M>0$, let $N_M:=\lfloor\frac{M}{L}\rfloor$, 
	i.e., the largest integer not exceeding $\frac{M}{L}$, 
	and let $E_M$ be the (convex) set of all vector functions  $(u_0, v_0)\in BUC\big([-M, +\infty)\big)^2$ satisfying the following constraints:
	\begin{equation*}
	    \left.
	    \begin{aligned}
		u_0(x) &\geq 0, & \max_{n\in\{0, 1, \ldots, N_M\}}\underline{u}(0, x+nL)\leq u_0(x)&\leq \overline{u}(0,x), \\
		v_0(x) &\geq 0, & \max_{n\in\{0, 1, \ldots, N_M\}}\underline{v}(0, x+nL)\leq v_0(x)&\leq \overline{v}(0,x), \\
		& & u_0(x)+v_0(x)& \leq \overline{K}, 
	    \end{aligned}\ \ \ \right\}
	    \ \ \ {}^\forall x\geq -M,
	\end{equation*} 
	where $\overline{K}$ is as in \eqref{eq:maxmincoeff}.
	Note that $E_M$ is not empty. Indeed, since $\overline{u}(0,x)>\underline{u}(0, x)$ and since $\overline{u}(0,x+nL)$ is monotone decreasing with respect to $n\in\mathbb{N}$, we clearly have $\overline{u}(x,0)>\underline{u}(x+nL)$ for all $n\geq 0$, and the same holds between $\overline{v}$ and $\underline{v}$. 
	Next we define an 
	operator $Q^M:E_M\to BUC\big([-M, +\infty)\big)^2$ by 
	\begin{equation}\label{QM}
	Q^M(u_0, v_0)(x)=\left(u\left(\frac{L}{c}, x+L\right),  v\left(\frac{L}{c}, x+L\right)\right),
	\end{equation}
	where $\big(u(t, x),v(t, x)\big)$ is the solution of the system \eqref{eq:main-sys-Neumann} starting from the initial data $(u_0, v_0)$.

	Then, it follows from Lemmas \ref{prop:upper-barrier}  and  \ref{prop:lower-barrier}  that  $Q^M(E_M)\subset E_M$. Moreover 
	the  operator 
	$Q^M$ is compact.
	Thus, the Schauder fixed-point Theorem implies the existence of a fixed-point $\big(u^M, v^M\big)\in E_M$ such that $Q^M(u^M, v^M)=(u^M, v^M)$. 
	Again by 
	the parabolic regularity, there exists a sequence $M_n\to+\infty$ such that $u^{M_n} $ converges locally uniformly to a solution $\vect{u}^\infty$ to $ Q^\infty(u^\infty, v^\infty)=(u^\infty, v^\infty)$, which belongs to the set $E_\infty\subset BUC(\mathbb{R})^2$ defined by the constraints
	\begin{equation*}
	    \left.
	    \begin{aligned}
		u_0(x) &\geq 0, & \sup_{n\in\mathbb{N}}\underline{u}(0, x+nL)\leq u_0(x)&\leq \overline{u}(0,x), \\
		v_0(x) &\geq 0, & \sup_{n\in\mathbb{N}}\underline{v}(0, x+nL)\leq v_0(x)&\leq \overline{v}(0,x), \\
		& & u_0(x)+v_0(x)& \leq \overline{K}, 
	    \end{aligned}\ \ \ \right\}
	    \ \ \ {}^\forall x\in\mathbb{R}.
	\end{equation*}
	Note first that the equality $Q^\infty(u^\infty, v^\infty)=(u^\infty, v^\infty)$ is equivalent to \eqref{eq:TW-propagating}. Next, 
	since $\overline{u}(t,x)\to 0$ 
	and $\overline{v}(t,x)\to 0$ 
	as $x\to+\infty$, we have $u^\infty(t, x)\to 0 $,  
	$v^\infty(t, x)\to 0 $ as $x\to+\infty$,
	which proves \eqref{eq:TW-limits+inf}. The condition 
	\eqref{eq:TW-limits-inf} is satisfied since $\sup_{n\in\mathbb{N}}\underline{u}(t, x+nL)$ and $\sup_{n\in\mathbb{N}}\underline{v}(t, x+nL)$ are uniformly bounded below 
	by positive constants as 
	$x\to-\infty$. 
	Thus 
	$(u^\infty, v^\infty)$ is  the expected traveling wave. 
	\smallskip
	
	Next we consider the case $c=c^*_{R}>0$. In this case, we can basically repeat the same procedure, replacing $E_M$ by the set $E_M^*\subset BUC\big([-M, +\infty)\big)^2$ defined by the constraints
	\[
	\begin{array}{ll}
		u_0(x) \geq 0, \ \  v_0(x)\geq 0, \ \  u_0(x)+v_0(x)\leq \overline{K},  \quad & {}^\forall x \geq -M; \\[6pt]
	     u_0(x)\leq \overline{u}^*(0,x), \quad\  v_0(x) \leq \overline{u}^*(0,x), \quad & {}^\forall x\geq \xi_0; \\[6pt]
	      \underline{u}^*_M(x)\leq u_0(x),\quad\ \  \underline{v}^*_M(x)\leq v_0(x),\ \ & {}^\forall x\geq \xi_0-N_ML,
	     	     \end{array}
	     \]
	     where 
	     \[
	     \underline{u}^*_M(x):=\max\limits_{n\in\{0, 1, \ldots, N_M\}}\underline{u}^*(0, x+nL), \quad 
	     \underline{v}^*_M(x):=\max\limits_{n\in\{0, 1, \ldots, N_M\}}\underline{v}^*(0, x+nL).
	     \]
	     Note that $E^*_M\ne\varnothing$ since $\overline{u}^*(0,x)>\underline{u}^*(0,x)$ and  
	     $\overline{u}^*(0,x+nL)$ is monotone decreasing with respect to $n\in\mathbb{N}$.
	     The operator $Q^M:E_M^*\to E_M^*$ is defined by \eqref{QM}. Thus $Q^M$ possesses a fixed point in $E^*_M$ and its limit $(u^\infty,v^\infty)$ as 
	$M\to \infty$ belongs to the set $E_\infty^*$ defined by the following constraints:
	\[
	\begin{array}{ll}
		u_0(x) \geq 0, \ \  v_0(x)\geq 0, \ \  u_0(x)+v_0(x)\leq \overline{K},  \quad & {}^\forall x \in\R; \\[6pt]
	     u_0(x)\leq \overline{u}^*(0,x), \quad\  v_0(x) \leq \overline{u}^*(0,x), \quad & {}^\forall x\geq \xi_0; \\[6pt]
	      \displaystyle 
	      \sup_{n\in\mathbb{N}}\underline{u}^*(0, x+nL)\leq u_0(x),\ \  
		\sup_{n\in\mathbb{N}}\underline{v}^*(0, x+nL)\leq v_0(x), \ \  & {}^\forall x\in\R.
	     	     \end{array}
	     \]
As above, one easily sees that the vector function $(u^\infty,v^\infty)$ satisfies \eqref{eq:TW-propagating} and \eqref{eq:TW-limits}, therefore it is the desired traveling wave. 
This completes the proof of Theorem \ref{thm:TW1}.
\end{proof}

\begin{proof}[Proof of Theorem \ref{thm:TW1s}]
The proof of the existence of right traveling waves for any super-critical speed $c>c^*_R=0$ is precisely the same as the proof of Theorem~\ref{thm:TW1} for the case  $c>c^*_R$, so we omit it. We now focus on the existence of a right stationary wave.

Let us explain how to adapt the argument in the proof of Theorem~\ref{thm:TW1} for $c=c^*_R>0$ to the case where $c^*_R=0$. For $\tau>0$,  we consider the solution map $Q^M_\tau$ defined by 
\begin{equation}\label{QM-tau}
	Q^M_\tau(u_0, v_0)(x)=\left(u\left(\tau, x\right),  v\left(\tau, x\right)\right),
\end{equation}
where $\big(u(t, x),v(t, x)\big)$ is the solution of the system \eqref{eq:main-sys-Neumann} starting from the initial data $(u_0, v_0)$. Let $E^*_M$ be the set defined above. 
Since the subsolution $\big(\underline{u}^*, \underline{v}^*\big)$ and the supersolution $\big(\overline{u}^*, \overline{v}^*\big)$ are stationary, we have $Q^M_\tau(E^*_M)\subset E^*_M$ for any $\tau>0$. As a consequence, there exists a fixed point $\big(u^M_{0, \tau}(x), v^M_{0, \tau}(x)\big)$, satisfying $Q^M_\tau\big(u^M_{0,\tau}, v^M_{0, \tau}\big)(x)=\big(u^M_{0,\tau}(x), v^M_{0, \tau}(x)\big)$. The corresponding solution of \eqref{eq:main-sys}, which we denote by $\big(u^M_\tau(t, x), v^M_\tau(t,x)\big)$, is therefore $\tau$-periodic in time. We consider the family of such solutions obtained by setting $\tau=2^{-k}=:\tau_k\,(k\in\mathbb{N})$. Since $\big(u^M_{0, \tau_k}(x), v^M_{0, \tau_k}(x)\big) =\big(u^M_{\tau_k}(1, x), v^M_{\tau_k}(1, x)\big)  $, by the standard estimates on parabolic operators, the family $\{\big(u^M_{0, \tau_k}, v^M_{0, \tau_k}\big)\::\: k\in\mathbb{N}\}$ is relatively compact with respect to the uniform topology; therefore, up to extracting a subsequence, it converges uniformly to a limit $\big(u^M_{0, 0}(x), v^M_{0, 0}(x)\big)$ as $k\to+\infty$. Clearly $\big(u^M_{0, 0}(x), v^M_{0, 0}(x)\big)$ is time-periodic of period $2^{-k}$ for any $k\in\mathbb{N}$, therefore it is stationary. 
By taking the limit $M\to-\infty$ with the usual diagonal extraction procedure, we obtain a stationary wave for \eqref{eq:main-sys} that belongs to $E^*_\infty$. This completes the proof of Theorem \ref{thm:TW1s}
\end{proof}

%%%%%%%%%%%%%%%%%%%
\subsection{Proof of uniqueness}

Here we prove Theorem \ref{thm:TW3} on the uniqueness of traveling waves under the condition \eqref{eq:r/kappa<mu}.

\begin{proof}[Proof of Theorem \ref{thm:TW3}]
We just focus on right traveling waves, and deal with the case $c\geq c^*_R>0$. The case $c>c^*_R=0$ follows from precisely the same arguments.
	By Proposition \ref{prop:uniform-bound}, the set 
	\begin{equation}
		\mathcal{T}:=\left\{\big(u(x), v(x)\big)\in BUC(\mathbb{R})^2\: :\: u(x)\geq 0, v(x)\geq 0 \text{ and } u(x)+v(x)\leq \frac{r_{\textrm{max}}}{\kappa_{\textrm{min}}} \right\}
	\end{equation}
	is positively invariant under the semiflow generated by \eqref{eq:main-sys}.
	Furthermore, any nonnegative solution will eventually be attracted by $\mathcal{T}$, uniformly in $x\in\mathbb{R}$ as $t\to+\infty$. Consequently, every traveling wave solution is contained in $\mathcal{T}$.

	We  remark that, under the assumptions of the theorem, the system \eqref{eq:main-sys} is cooperative in $\mathcal{T}$. Indeed, by \eqref{eq:r/kappa<mu}, we have $u+v\leq \frac{\mu_{\min}}{\kappa_{\max}}$ in $\mathcal{T}$, hence
	\begin{equation*}
		\frac{\partial f_u}{\partial v}(x, u, v) = \mu_{v}(x)-\kappa_{u}(x) u \geq 0, \qquad  \frac{\partial f_v}{\partial u}(x, u, v) = \mu_{u}(x)-\kappa_{v}(x) v \geq 0,
	\end{equation*}
	where $f_{u}(x, u, v)$ and $f_{v}(x,u,v)$ denote, respectively, the reaction terms in the first and the second line of \eqref{eq:main-sys}. In particular, the classical comparison principle for cooperative parabolic systems holds true: given a subsolution $\big(\underline{u}(t, x), \underline{v}(t, x)\big)$ and a supersolution $\big(\overline{u}(t, x), \overline{v}(t, x)\big)$ satisfying 
	 $0\leq \underline{u}(T, x)\leq \overline{u}(T, x)$,  $0\leq \underline{v}(T, x)\leq \overline{v}(T, x)$ for some $T\in\mathbb{R}$ and for all $x\in\mathbb{R}$, and $\big(\overline{u}(t, \cdot), \overline{v}(t,\cdot)\big)\in\mathcal{T}$ for all $t\geq T$, 
	we have
	\begin{equation}
		\underline{u}(t, x)\leq \overline{u}(t, x),\text{ and }\underline{v}(t, x)\leq \overline{v}(t, x)\qquad  \,^\forall t>T \text{ and }\,^\forall x\in\mathbb{R}.
	\end{equation}
%%	Here, $\big(\underline{u}(t, x), \underline{v}(t, \cdot)\big)$ is a nonnegative subsolution of \eqref{eq:main-sys}, therefore satisfies the upper estimate from Proposition \ref{prop:uniform-bound}. Hence $\big(\underline{u}(t, \cdot), \underline{v}(t, \cdot)\big)\in \mathcal{T}$ for all $t\geq T$. 
	Furthermore if $\underline{u}(T, x)\not \equiv \overline{u}(T, x)$ or $\underline{v}(T, x)\not \equiv \overline{v}(T, x)$, then the strict inequalities $\underline{u}(t, x)<\overline{u}(t, x)$ and $\underline{v}(t, x)<\overline{v}(t, x)$ holds for any $t>T$ and $x\in\mathbb{R}$.
	By this remark and Theorem \ref{thm:TW2}, the proof then follows from a classical sliding argument, which we recall here.

	Let $\big(u^1(t, x), v^1(t,x)\big)$ and $\big(u^2(t, x), v^2(t,x)\big) $ be two traveling waves with speed $c\geq c^*_R$. 
	Since \eqref{eq:main-sys} is cooperative in $\mathcal{T}$ and the nonlinearity satisfies KPP-type conditions, 
	it is classical that there exists a unique positive stationary solution $\big(u^*(x), v^*(x)\big)$ of \eqref{eq:main-sys}. 
	Furthermore, {although our definition of traveling waves, Definition~\ref{def:TW}, only requires that the traveling wave be bounded below by a positive constant far behind the front, it is not difficult to show that any traveling wave in this broader sense converge to the positive steady state as $x\to -\infty$, if the system is entirely cooperative. This follows by the shifting argument and the comparison principle. Thus we have}
	\begin{equation}
		\lim_{x\to-\infty}\big(|u^i(t, x)-u^*(x)|+|v^i(t, x)-v^*(x)|\big) =0, \quad {}^\forall t\in\R,\ i=1,2.
	\end{equation}
	See for example, \cite[Theorem 2.1]{Wei-02} or \cite[Proposition 3.32]{Lia-Zha-10}.
	We remark that $\big(u^{1, \gamma}(t, x), v^{1, \gamma}(t, x)\big):=\gamma\big(u^1(t, x), v^{1}(t, x)\big)$ is a subsolution of \eqref{eq:main-sys} for all $\gamma\in(0, 1)$.  
	Fix $t\in\mathbb{R}$ arbitrarily and $x_0\ll -1$. 
	For $x\leq x_0$, we have $\big(u^2(t, x), v^2(t, x)\big)\geq \frac{1+\gamma}{2}\big(u^*(x), v^*(x)\big)$, and therefore  there exists $\epsilon>0$ such that 
	\begin{equation}\label{eq:260616a}
		\big(u^2(t, x), v^2(t, x)\big)\geq \big(u^{1, \gamma}(t', x)+\epsilon, v^{1, \gamma}(t', x)+\epsilon\big), \qquad ^{\forall} t'\in\mathbb{R} \text{ and }^\forall x\leq x_0.
	\end{equation}
	By Theorem \ref{thm:TW2} we know moreover that there exist constants $\alpha_1, \alpha_2>0$ such that 
	\begin{multline}
		\big(u^1(t, x), v^1(t, x)\big) \sim \alpha_1 e_{\lambda}(x-ct)\big(\varphi^\lambda(x), \psi^{\lambda}(x)\big), \text{ and } \\
		\big(u^2(t, x), v^2(t, x)\big) \sim \alpha_2 e_\lambda(x-ct)\big(\varphi^\lambda(x), \psi^{\lambda}(x)\big),
		 \text{ as } x\to+\infty,
	\end{multline}
	where $e_\lambda(\xi)=e^{-\lambda \xi}$ if $c>c^*_R$ and $e_\lambda(\xi)=\xi e^{-\lambda^*\xi}$ if $c=c_R$.
	As a consequence, there exists $t_0$ such that 
	\begin{equation*}
		\big(u^2(t, x), v^2(t, x)\big)\geq \big(u^{1, \gamma}(t-t_0, x), v^{1, \gamma}(t-t_0, x)\big), \qquad ^\forall x\in\mathbb{R}.
	\end{equation*}
	Define 
	\begin{equation*}
		t_*:=\inf\{t'\in\mathbb{R}\: :\: \big(u^2(t, x), v^2(t, x)\big)\geq \big(u^{1, \gamma}(t-t', x), v^{1, \gamma}(t-t', x)\big), \, ^\forall x\in\mathbb{R}\}.
	\end{equation*}
	Then $t_*\leq t_0<+\infty$. By the definition of $t_*$ we have $u^2(t, x)\geq u^{1, \gamma}(t-t_*, x)$ and $v^2(t, x)\geq v^{1, \gamma}(t-t_*, x)$ for all $x\in\mathbb{R}$. Moreover, $u^2$ and $u^{1, \gamma}$ are not identical because they have different behaviors as $x\to-\infty$; therefore by the strong comparison principle we have the strict inequality 
	\begin{equation*}
		u^2(t, x)=u^2(t+L/c, x+L) > u^{1, \gamma}(t-t_*+L/c, x+L)=u^{1, \gamma}(t-t_*, x), \qquad ^\forall x\in\mathbb{R},
	\end{equation*}
	where we have used the pulsating relation \eqref{eq:TW-propagating}; 
	and similarly we conclude that $  v^2(t, x)>v^{1, \gamma}(t-t_*, x) $ for all $x\in\mathbb{R}$. In view of this,  \eqref{eq:260616a} and the definition of $t_*$, there exists a sequence $x_n\to +\infty$ such that either 
	$\lim_{n\to+\infty}\frac{u^{1, \gamma}(t-t_*, x_n)}{u^2(t,x_n)}=1$, or 
	$\lim_{n\to+\infty}\frac{v^{1, \gamma}(t-t_*, x_n)}{v^2(t,x_n)}=1$. 
	In both cases, we obtain 
	\begin{equation*}
		1=\lim_{n\to+\infty} \frac{\gamma\alpha_1 e_\lambda\big(x_n-c(t-t_*)\big)}{\alpha_2 e_\lambda(x_n-ct)} = \gamma \frac{\alpha_1}{\alpha_2} e^{\lambda c t_*},
	\end{equation*}
	so that 
	\begin{equation*}
		t_* = \frac{1}{\lambda c} \ln\left(\frac{\alpha_2}{\gamma\alpha_1}\right).
	\end{equation*}
	By taking the limit $\gamma\to 1$ and up to redefining $t_*$, we have shown the property that 
	\begin{equation}\label{eq:260616b}
		\big(u^1(t-t', x),v^1(t-t', x)\big) \leq \big(u^2(t, x), v^2(t, x)\big), \qquad^\forall t'\geq t_*:= \frac{1}{\lambda c}\ln\left(\frac{\alpha_2}{\alpha_1}\right).
	\end{equation}
	Exchanging the roles of $\big(u^1(t, x), v^1(t, x)\big) $ and $ \big(u^2(t, x), v^2(t, x)\big)$ and arguing similarly, we obtain
	\begin{equation}\label{eq:260616c}
		\big(u^1(t, x),v^1(t, x)\big) \geq \big(u^2(t-t', x), v^2(t-t', x)\big), \qquad^\forall t'\geq t'_*:= \frac{1}{\lambda c}\ln\left(\frac{\alpha_1}{\alpha_2}\right)=-t_*.
	\end{equation}
	This proves that $\big(u^1(t-t_*, x), v^1(t-t_*, x) \big)\equiv \big(u^2(t, x), v^2(t, x)\big)$, for all $x\in\mathbb{R}$, and consequently this equality holds for any $t\in\mathbb{R}$. This proves uniqueness up to a time shift. 

	Now, we apply \eqref{eq:260616b} to the pair $\big(u^1(t, x), v^1(t, x)\big)= \big(u^2(t, x), v^2(t, x)\big)$. In this case $\alpha_1=\alpha_2$, therefore we have $t_*=0$. Hence \eqref{eq:260616b} shows that $\big(u^1(t, x), v^1(t, x)\big)$ is time-increasing. 
	This finishes the proof or Theorem \ref{thm:TW3}.
\end{proof}

\bigskip
%%%%%%%%%%%%%%%%%%%%%%%%%%%%%%%%%%%%%%%%
%%%%%%%%%%%%%%%%%%%%%%%%%%%%%%%%%%%%%%%%

\section{Exponential decay of traveling waves}\label{s:decay}

In this section we focus on the exponential decay of traveling waves along the leading edge, as stated in Theorem \ref{thm:TW2}. 
Throughout this section we 
fix $c>0$. 
Our argument is largely inspired by the scalar case studied by Hamel \cite{Ham-08}. 
As in \cite{Ham-08}, we make
 a change of variable to obtain compactness in the ``space variable'', by introducing the following functions:
\begin{equation}
    U(t,x) = u\left(\frac{t+x}{c}, x\right); \ V(t, x)=v\left(\frac{t+x}{c}, x\right).
\end{equation}
It is easily seen from the pulsating relation \eqref{eq:TW-propagating} that the new functions $U$ and $V$ are $L$-periodic with respect to the $x$ variable and satisfy
\begin{equation}\label{eq:TW-newvar}
    \begin{split}
	%-\sigma(x)\big(c^{-2}U_{tt}-2c^{-1}U_{tx}+U_{xx}\big)-\sigma_x(x)\big(-c^{-1}U_t+U_x\big)+c^{-1}U_t &= f_u(x, U, V),  \\
	%-\sigma(x)\big(c^{-2}V_{tt}-2c^{-1}V_{tx}+V_{xx}\big)-\sigma_x(x)\big(-c^{-1}V_t+V_x\big)+c^{-1}V_t &= f_v(x, U, V).  
	cU_t-\sigma(x)\big(U_{tt}-2U_{tx}+U_{xx}\big)+\sigma_x(x)\big(U_t-U_x\big)&= f_u(x, U, V),  \\
	cV_t-\sigma(x)\big(V_{tt}-2V_{tx}+V_{xx}\big)+\sigma_x(x)\big(V_t-V_x\big) &= f_v(x, U, V).  
    \end{split}
\end{equation}
Moreover, given $U(t, x)$ and $V(t, x)$, we can recover the original functions by the change of variables $u(t, x) = U\left(ct-x, x\right)$ and $v(t, x) = V\left(ct-x, x\right)$. 
With these functions $U, V$, the statements \eqref{eq:cv-fronts-super}, \eqref{eq:cv-fronts-critical} of Theorem \ref{thm:TW2} are expressed in the following form:
\begin{equation}\label{eq:decay-U,V}
\lim_{t\to-\infty}\frac{U(t, x)}{e(t)\varphi^\lambda(x)}=
\lim_{t\to-\infty}\frac{V(t, x)}{e(t)\psi^\lambda(x)}=\alpha \quad\hbox{uniformly in}\ \ x\in\R,
\end{equation}
where $e(t)= e^{\lambda_c t}$ and $\lambda=\lambda_c$ if $c>c^*_R$, while $e(t)= -t e^{\lambda^*t}$ and $\lambda=\lambda^*$ if $c=c^*_R$.

Let us give a brief summary of the structure of the proof.
Although  the system \eqref{eq:TW-newvar} is neither parabolic nor uniformly elliptic, the strong maximum principle still holds for \eqref{eq:TW-newvar}, because the original equation is uniformly parabolic. Therefore, we can establish the sweeping principle (Proposition \ref{prop:sweeping-method}), which applies to a family of  subsolutions  $\big(\underline{U}_\mu, \underline{V}_\mu\big)$ and supersolutions $\big(\overline{U}_\mu, \overline{V}_\mu\big)$, to show that $\big(\underline{U}_0, \underline{V}_0\big)\leq \big(\overline{U}_0, \overline{V}_0\big)$ implies $\big(\underline{U}_1, \underline{V}_1\big)\leq \big(\overline{U}_1, \overline{V}_1\big)$.
 Then, we use this result to derive an analogue of the sliding method (Proposition \ref{prop:comparison-rescaled}). 
 These two Propositions, 
 along with the Harnack inequality derived in Lemma \ref{lem:Harnack}, constitute our main tools to establish the rate of decay along the leading edges of traveling waves.

Then, we start analyzing the exponential behavior of the leading edges of traveling waves. Proposition~\ref{prop:lower-bound-tail} is devoted to a rough exponential lower bound, while Proposition~\ref{prop:upper-bound-tail} establishes a rough exponential upper bound. The proofs are quite similar and rely on the same idea. For the lower bound, suppose by contradiction that there exists a sequence of points $(t_n, x_n)$ with $t_n\to -\infty$ such that $\big(U(t_n), V(t_n)\big)$ converge to 0 faster than expected, i.e.
\begin{equation*}
    \min\big(U(t_n, x_n), V(t_n, x_n)\big) \leq C_n e^{\lambda t_n},
\end{equation*}
where  $\lambda$ is the smallest positive solution of the equation $\lambda c=k(\lambda)$ (here we explain the argument for super-critical speeds, 
which is simpler). By using a supersolution of the type $ \big(\overline{U}(t, x), \overline{V}(t, x)\big):= e^{\lambda t}\big(\varphi^\lambda(x), \psi^\lambda(x)\big)$ and the Harnack inequality, we can use the sliding method on the interval $[t_n, 0] $ to show that 
\begin{equation*}
    \min\big(U(0, x_n), V(0, x_n)\big) \leq \tilde{C}_n ,
\end{equation*}
where $\tilde{C}_n$ converges to $0$ as $n\to+\infty$. Taking the limit, this is obviously a contradiction, 
since $\big(U(0, x), V(0, x)\big)$ are uniformly positive because of the periodicity.
We proceed similarly for the upper bound. The technical details are slightly more involved for the critical case $c=c^*_R$. To summarize, we use a supersolution to derive a rough lower bound, and a subsolution to derive a rough upper bound, which may sound slightly unusual. As we shall see below, for a more refined lower bound, we use a family of subsolutions.

In Lemma \ref{lem:exponential-decay-minimum}, we prove the stronger property that 
\[
    \inf_{x\in\mathbb{R}}\min\left(\frac{U(t, x)}{e^{\lambda t }\varphi^\lambda(x)}, \frac{V(t, x)}{e^{\lambda t}\psi^\lambda(x)}\right)
\]
converges to a positive limit, 
say $\alpha$, as $t\to -\infty$.
This is done by a family of subsolutions and the sweeping method (Proposition \ref{prop:sweeping-method}). To apply the sweeping method, we have to show that the starting member of the family of subsolutions lies below the traveling wave, which is guaranteed by the rough lower bound in Proposition \ref{prop:lower-bound-tail}. Again, the method is the same for the critical speed, but the computations are more involved. 
Note that, by shifting the time, we may assume without loss of generality that $\alpha=1$ in the above limit.

Finally we prove \eqref{eq:decay-U,V} with $\alpha=1$ by showing that
\[
    \sup_{x\in\mathbb{R}}\max\left(\frac{U(t, x)}{e^{\lambda t }\varphi^\lambda(x)}, \frac{V(t, x)}{e^{\lambda t}\psi^\lambda(x)}\right)\to 1\quad\hbox{as}\ \ t\to-\infty.
\]
This is done by a contradiction argument  and the strong maximum principle.

The above strategy is, for the most part,  similar to \cite{Ham-08}. However, there are  some important differences.  
Among other things, since the traveling wave does not necessarily lie in the cooperative region, we do not know if the traveling wave is increasing in $t$ entirely, which is in marked contrast with the scalar case. 
Another difficulty, which is also seen in the construction of the traveling wave in section \ref{s:existence}, is that the two components of the subsolutions $\big(\underline{U}(t, x), \underline{V}(t, x)\big)$ do not necessarily become negative at the same position, therefore, we have to pay attention to the behavior of $\big(\underline{U}(t, x), \underline{V}(t, x)\big)$ even where they are negative.  

We now proceed to the proof. 
We start with  a comparison principle. 
\begin{proposition}[The sweeping method]\label{prop:sweeping-method}
	Let 
	$\sigma $ be a $C^1$, positive and $L$-periodic function, 
	and $\big(\overline{U}_\mu(t, x), \overline{V}_\mu(t, x)\big)$ and $\big(\underline{U}_\mu(t, x), \underline{V}_\mu(t, x)\big)$ be continuous with respect to $\mu\in [0,1]$,  $L$-periodic in $x$ and satisfy respectively
    \begin{equation}\label{eq:newvar-super}
	\begin{split}
	&c\overline{U}_t-\sigma(x)\big(\overline{U}_{tt}-2\overline{U}_{tx}+\overline{U}_{xx}\big)+\sigma_x(x)\big(\overline{U}_t-\overline{U}_x\big) \geq g_u(x, \overline{U}, \overline{V}),  \\
	&c\overline{V}_t-\sigma(x)\big(\overline{V}_{tt}-2\overline{V}_{tx}+\overline{V}_{xx}\big)+\sigma_x(x)\big(\overline{V}_t-\overline{V}_x\big) \geq g_v(x, \overline{U}, \overline{V}),  
	\end{split}
    \end{equation}
    and
    \begin{equation}\label{eq:newvar-sub}
	\begin{split}
	&c\underline{U}_t-\sigma(x)\big(\underline{U}_{tt}-2\underline{U}_{tx}+\underline{U}_{xx}\big)+\sigma_x(x)\big(\underline{U}_t-\underline{U}_x\big) \leq g_u(x, \underline{U}, \underline{V}),  \\
	&c\underline{V}_t-\sigma(x)\big(\underline{V}_{tt}-2\underline{V}_{tx}+\underline{V}_{xx}\big)+\sigma_x(x)\big(\underline{V}_t-\underline{V}_x\big) \leq g_v(x, \underline{U}, \underline{V}),  
	\end{split}
    \end{equation}
	in the classical sense, for all $\mu\in[0,1]$, $t\in[T_1, T_2]$ and $x\in\mathbb{R}$, where $g=(g_u, g_v)$ is a cooperative nonlinearity that is also locally Lipschitz  in $(U, V)$  and $L$-periodic in $x$. Suppose that the following three  statements hold true.  
	    \begin{subequations}\label{eq:250928a}
	    \begin{equation}\label{eq:250928aa}
		    \underline{U}_\mu(T_1, x)<\overline{U}_\mu(T_1,x),\quad  \underline{V}_\mu(T_1, x)<\overline{V}_\mu(T_1,x), \,^\forall x\in\mathbb{R}\text{ and }\,^\forall\mu\in [0,1), 
	    \end{equation}
		\begin{equation}\label{eq:250928ab}
		    \underline{U}_\mu(T_2,x)<\overline{U}_\mu(T_2, x),\quad  \underline{V}_\mu(T_2,x)<\overline{V}_\mu(T_2, x), \,^\forall x\in \mathbb{R}\text{ and }\,^\forall\mu\in [0,1),
	    \end{equation}
	    \begin{equation}\label{eq:250928ac}
		    \underline{U}_0(t, x)\leq\overline{U}_0(t,x),\quad  \underline{V}_0(t, x)\leq\overline{V}_0(t,x), \,^\forall(t, x)\in [T_1, T_2]\times\mathbb{R}. 
	    \end{equation}
	    \end{subequations}
    Then, 
    \begin{equation}\label{eq:250928b}
	\underline{U}_1(t, x)\leq \overline{U}_1(t, x)  \text{ and } \underline{V}_1(t, x)\leq \overline{V}_1(t, x), \text{ for all }(t, x)\in[T_1, T_2]\times \mathbb{R}.
    \end{equation}

%	If moreover $g=(g_u, g_v)$ is strongly coupled in the sense that $v\mapsto g_u(x, u, v)$ and $u\mapsto g_v(x,u,v)$ are strictly increasing 
%	and satisfies $g_u(x, 0, 0)\geq 0$,  $g_v(x, 0, 0)\geq 0$,
	If $\overline{U}_\mu>0$ and $\overline{V}_\mu> 0 $ for all $\mu\in[0, 1)$,  
	then the inequalities \eqref{eq:newvar-sub} need to be satisfied   only on the set of $(\mu, t, x)$ where either $\underline{U}_{\mu}(t, x)> 0$ or $\underline{V}_\mu(t, x)> 0$.
\end{proposition}
\begin{proof}
	Let us define
    \begin{equation}\label{eq:250929a}
	\mu^*:=\sup\{\mu\geq 0\: :\: \underline{U}_\mu(t, x)\leq \overline{U}_\mu(t, x) \text{ and } \underline{V}_\mu(t, x)\leq \overline{V}_\mu(t, x) , \,^\forall (t,x)\in [T_1, T_2]\}. 
    \end{equation}
	Our goal is to show that $\mu^*=1$, which would establish \eqref{eq:250928b}. Assume by contradiction that $\mu^*<1$. Then there exists a sequence of points $(\mu_n, t_n, x_n)\in (\mu^*, 1]\times [T_1, T_2]\times \mathbb{R}$ with $\mu_n\to \mu^*$ as $n\to+\infty$, such that  
    \begin{equation}
	    \text{ either } \underline{U}_{\mu_n}(t_n, x_n)>\overline{U}_{\mu_n}(t_n, x_n) \text{ or } \underline{V}_{\mu_n}(t_n, x_n)>\overline{V}_{\mu_n}(t_n, x_n) , \qquad \,^\forall n\in\mathbb{N}.
    \end{equation}
	We will assume the former, namely that $\underline{U}_{\mu_n}(t_n, x_n)>\overline{U}_{\mu_n}(t_n, x_n)$; the argument is similar if $\underline{V}_{\mu_n}(t_n, x_n)>\overline{V}_{\mu_n}(t_n, x_n)$. By using the compactness of  $(t, x)\in [T_1, T_2]\times[0,L]$, we may assume that $t_n\to t^*\in[T_1, T_2]$ and $x_n\to x^*\in[0,L]$. Using \eqref{eq:250929a} we deduce that 
    \begin{equation}\label{eq:250929b}
	    \underline{U}_{\mu^*}(t^*, x^*)=\overline{U}_{\mu^*}(t^*, x^*) \text{ and } \underline{V}_{\mu^*}(t^*, x^*)\leq\overline{V}_{\mu^*}(t^*, x^*).
    \end{equation}
    Next, we deduce from our assumption \eqref{eq:250928aa} that $t^*>T_1$ and we deduce from \eqref{eq:250928ab} that $t^*<T_2$. 
    \medskip

	We define the functions $\overline{u}({t}, x):=\overline{U}_{\mu^*}(ct-x, x)$, $\overline{v}({t}, x):=\overline{V}_{\mu^*}(ct-x, x)$, $\underline{u}({t}, x):=\underline{U}_{\mu^*}(ct-x, x)$, $\underline{v}({t}, x):=\underline{V}_{\mu^*}(ct-x, x)$ in the strip $\frac{T_1+{x}}{c}\leq {t}\leq \frac{T_2+{x}}{c}$ (and $x\in\mathbb{R}$). Then $(\overline{u}, \overline{v})$ and $(\underline{u}, \underline{v})$ satisfy 
    \begin{equation}
	\begin{aligned}
	    \overline{u}_t - \big(\sigma(x)\overline{u}_x\big)_x & \geq g_u(x, \overline{u}, \overline{v}), \\
	    \overline{v}_t - \big(\sigma(x)\overline{v}_x\big)_x &\geq g_v(x, \overline{u}, \overline{v}), 
	\end{aligned}
    \end{equation}
and 
    \begin{equation}
	\begin{aligned}
	    \underline{u}_t - \big(\sigma(x)\underline{u}_x\big)_x & \leq g_u(x, \underline{u}, \underline{v}), \\
	    \underline{v}_t - \big(\sigma(x)\underline{v}_x\big)_x &\leq g_v(x, \underline{u}, \underline{v}). 
	\end{aligned}
    \end{equation}
	It follows from \eqref{eq:250929b} that $\overline{u}(\tilde{t}^*+\tau^*/c, x^*) = \underline{u}(\tilde{t}^*, x^*)$ and $\overline{v}(\tilde{t}^*+\tau^*/c, x^*) \leq  \underline{v}(\tilde{t}^*,x^*)$, where $\tilde{t}^*:=\frac{t^*+x^*}{c}$; and the point $(\tilde{t}^*, x^*)$ is  interior to the strip $\frac{T_1+{x}}{c}< {t}< \frac{T_2-\tau^*+{x}}{c}$. It follows from the strong maximum principle for cooperative parabolic systems (see e.g. \cite[Chap. 3 Theorem 13 p.190]{Pro-Wei-1984}) that
	\begin{equation}\label{eq:251002a}
	\overline{u}(t/c, x)\equiv\underline{u}(t, x), \text{ for all }(t, x) \text{ such that } \frac{T_1+{x}}{c}< {t}< \frac{T_2+{x}}{c}.
    \end{equation}
	Coming back to $\overline{U}_{\mu^*}$ and $\underline{U}_{\mu^*}$ and using the periodicity in $x$, 
	one gets $\overline{U}_{\mu^*}(t, \cdot)\equiv \underline{U}_{\mu^*}(t, \cdot)$ for $t\in [T_1, t^*]$, which contradicts \eqref{eq:250928aa}. We have shown that $\mu^*=1$, which proves \eqref{eq:250928b}. 

	Moreover, if we assume that $\overline{U}_\mu>0$ and $\overline{V}_\mu>0$ for all $\mu\in[0,1)$, then it is clear that the same argument as above applies, if we simply assume that $(\underline{U}_\mu, \underline{V}_\mu)$ satisfies the inequalities \eqref{eq:newvar-sub} in the region where either  $\underline{U}_{\mu}(t, x)> 0$ or $\underline{V}_\mu(t, x)> 0$. This proves the last part of Proposition \ref{prop:sweeping-method}. The proof is complete. 
\end{proof}
A direct consequence of Proposition \ref{prop:sweeping-method} is the following result, known as the ``sliding method''.
\begin{proposition}[The sliding method]\label{prop:comparison-rescaled}
    Let Assumption \ref{as:coop-comp} hold true, and $\big(\overline{U}(t, x), \overline{V}(t, x)\big)$ and $\big(\underline{U}(t, x), \underline{V}(t, x)\big)$ be $L$-periodic in $x$ and satisfy  
	respectively \eqref{eq:newvar-super} and \eqref{eq:newvar-sub}  
	in the classical sense, for all $t\in[T_1, T_2]$ and $x\in\mathbb{R}$, where $g=(g_u, g_v)$ is cooperative nonlinearity that is also locally Lipschitz  in $(U, V)$  and $L$-periodic in $x$. Suppose that: 
    \begin{enumerate}\label{prop:comparison-monotone}
	\item Either 
	    \begin{subequations}\label{eq:250708a}
	    \begin{equation}\label{eq:250708aa}
		\underline{U}(T_1, x)<\overline{U}(t,x),\quad  \underline{V}(T_1, x)<\overline{V}(t,x), \,^\forall (t, x)\in (T_1, T_2]\times\mathbb{R}, 
	    \end{equation}
		and 
		\begin{equation}\label{eq:250708ab}
		    \overline{U}(T_2, x)>\underline{U}(t,x),\quad  \overline{V}(T_2, x)>\underline{V}(t,x), \,^\forall (t, x)\in [T_1,T_2)\times\mathbb{R};
	    \end{equation}
	    \end{subequations}
	\item Or, 
	    \begin{subequations}\label{eq:250708b}
	    \begin{equation}
		\underline{U}(T_2, x)<\overline{U}(t,x),\quad  \underline{V}(T_2, x)<\overline{V}(t,x), \,^\forall(t, x)\in [T_1, T_2)\times\mathbb{R}, 
	    \end{equation}
		 and 
		    \begin{equation}\label{eq:251002b}
			    \overline{U}(T_1, x)>\underline{U}(t,x), \quad \overline{V}(T_1, x)>\underline{V}(t,x), \,^\forall (t, x)\in (T_1, T_2]\times\mathbb{R};
	    \end{equation}
	    \end{subequations}
    \end{enumerate}
    Then 
    \begin{equation}\label{eq:250708c}
	\overline{U}(t, x) \geq \underline{U}(t, x) \text{ and }\overline{V}(t, x)\geq \underline{V}(t, x), \text{ for all }(t, x)\in[T_1, T_2]\times \mathbb{R}.
    \end{equation}
\end{proposition}
\begin{proof}
	To prove assertion $1$,  we apply the sweeping method (Proposition  \ref{prop:sweeping-method}) to the family of sub- and supersolutions defined by $\big(\underline{U}_{\mu}(t, x), \underline{V}_\mu(t, x)\big)=\big(\underline{U}(t, x), \underline{V}(t, x)\big)$ and
	$\big(\overline{U}_{\mu}(t, x), \overline{V}_\mu(t, x)\big)=\big(\overline{U}(t+(1-\mu)(T_2-T_1), x), \overline{V}(t+(1-\mu)(T_2-T_1), x)\big)$.

%	The proof of point $1$ is similar to that of Proposition \ref{prop:sweeping-method}, by using the family $\big(\underline{U}_{\mu}(t, x), \underline{V}_\mu(t, x)\big)=\big(\underline{U}(t, x), \underline{V}(t, x)\big)$ and
	%$\big(\overline{U}_{\mu}(t, x), \overline{V}_\mu(t, x)\big)=\big(\overline{U}(t+(1-\mu)(T_2-T_1), x), \overline{V}(t+(1-\mu)(T_2-T_1), x)\big)$.
	%The only difference is that $\big(\overline{U}_{\mu}(t, x), \overline{V}_\mu(t, x)\big)$ is not defined in the whole interval $t\in[T_1,T_2]$, but only for $t\in[T_1, T_2-(1-\mu)(T_2-T_1)]$. This is not a problem, because our assumption \eqref{eq:250708a} enforces that $\underline{U}_\mu(t, x)<\overline{U}_{\mu}$ and $\underline{V}_\mu(t, x)<\overline{V}_{\mu}(t, x)$ for any $(t, x)\in \{T_1, T_2-(1-\mu)(T_2-T_1)\}\times\mathbb{R}$. We omit the details. 

The proof of assertion $2$ is similar to that of assertion $1$, by using the family $\big(\underline{U}_{\mu}(t, x), \underline{V}_\mu(t, x)\big)=\big(\underline{U}(t+(1-\mu)(T_2-T_1), x), \underline{V}(t+(1-\mu)(T_2-T_1), x)\big)$ and
	$\big(\overline{U}_{\mu}(t, x), \overline{V}_\mu(t, x)\big)=\big(\overline{U}(t, x), \overline{V}(t, x)\big)$.
\end{proof}

\begin{lemma}[Harnack-type estimate]\label{lem:Harnack}
    Let Assumption \ref{as:coop-comp} hold true and let $c\geq c^*_R$ be fixed.  There exists a constant $H>0$ such that for any classical solution $\big(U(t, x), V(t, x)\big)$ of \eqref{eq:TW-newvar} with $(t,x)\in\mathbb{R}\times\mathbb{R}$, one has
    \begin{equation}\label{eq:est-Harnack}
	\min\big(U(t, x), V(t, x)\big) \geq H^{-1}\max\big(U(t, y), V(t, y)\big)
    \end{equation}
    for any $t\in\mathbb{R}$ and all $(x, y)\in\mathbb{R}\times\mathbb{R}$.
\end{lemma}
\begin{proof}
	We come back to the original variables by setting $u(t, x)=U(ct-x, x)$ and $v(t, x) = V(ct-x, x)$; then $(u,v)$ is an entire solution of the parabolic system \eqref{eq:main-sys}. In order to obtain \eqref{eq:est-Harnack}, it is sufficient to establish that
    \begin{equation} \label{eq:250708f}
	\underset{x\in [-2L, L]}{\inf_{s\in [t/c, (t+L)/c]}} \min \big(u(s, x), v(s, x)\big) \geq H^{-1} \underset{x\in [-2L, L]}{\sup_{s\in [(t-2L)/c, (t-L)/c]}} \max\big(u(s, x), v(s, x)\big). 
    \end{equation}
	The Harnack inequality established by Földes and Polá\v{c}ik \cite[Theorem 3.9]{Fol-Pol-09} provides a constant $\kappa$ depending only on $\theta$ such that for any $\tau\in\mathbb{R}$,
    \begin{equation} 
	\underset{x\in [-2L, L]}{\inf_{s\in [\tau+7/2\theta, \tau+4\theta]}} \min \big(u(s, x), v(s, x)\big) \geq \kappa \underset{x\in [-2L, L]}{\sup_{s\in [\tau+\theta, \tau+2\theta]}} \max\big(u(s, x), v(s, x)\big). 
    \end{equation}
    Selecting $\theta=L/c$ and  $\tau=(t-3L)/c$ we obtain  \eqref{eq:250708f} with  $H=\frac{1}{\kappa} $  and therefore \eqref{eq:est-Harnack}.
\end{proof}
We can now obtain a lower  bound of the leading edge of the traveling wave.

\begin{proposition}[Lower rough bound of the leading edge]\label{prop:lower-bound-tail}
Let Assumption \ref{as:coop-comp} hold true and let $c\geq c^*_R$ be fixed. Let $\big(U(t, x), V(t, x)\big) $ be a classical, nonnegative and nontrivial solution of \eqref{eq:TW-newvar}. Then, 
    \begin{enumerate}
	\item if $c>c^*_R$, there is a constant $C>0$ and $T\in\mathbb{R}$ such that 
	    \begin{equation}\label{eq:250725b}
		U(t, x) \geq Ce^{\lambda_c t} \varphi^{\lambda_c}(x) \text{ and } V(t, x)\geq C e^{\lambda_c t}\psi^{\lambda_c}(x) ,  \qquad \,^\forall t\leq T, \,^\forall x\in \mathbb{R}, 
	    \end{equation}
	    where $\lambda_c>0$ is the smallest positive solution of $\lambda_c c = k(\lambda_c)$;
	\item 
	    if $c=c^*_R$, there is a constant $C>0$  and $T\in\mathbb{R}$ such that
	    \begin{equation}\label{eq:250725c}
		U(t, x) \geq C|t|e^{\lambda^* t} \varphi^{\lambda^*}(x) \text{ and } V(t, x)\geq C|t| e^{\lambda^* t}\psi^{\lambda^*}(x) ,  \qquad \,^\forall t\leq T, \,^\forall x\in \mathbb{R}, 
	    \end{equation}
	    where $\lambda^*>0$ is the smallest positive solution of $\lambda^* c = k(\lambda^*)$.
    \end{enumerate}
\end{proposition}

\begin{proof}
    Let us prove \eqref{eq:250725b} first. Assume by contradiction that \eqref{eq:250725b} does not hold, then there exist sequences $t_n\to -\infty$, $x_n\in\mathbb{R}$ and $C_n\to 0$ such that 
	    \begin{equation*}
		U(t_n, x_n) \leq C_ne^{\lambda_c t_n} \varphi^{\lambda_c}(x_n) \text{ or } V(t_n, x_n)\leq C_n e^{\lambda_c t_n}\psi^{\lambda_c}(x_n).
	    \end{equation*}
	    By the Harnack inequality \eqref{eq:est-Harnack} in Lemma \ref{lem:Harnack} we deduce
	    \begin{equation}
		U(t_n, x)< \widetilde{C}_n e^{\lambda_c t_n}\varphi^{\lambda_c }(x) \text{ and } V(t_n, x)< \widetilde{C}_n e^{\lambda_c t_n}\psi^{\lambda_c }(x) , \qquad \,^\forall x\in\mathbb{R}, 
	    \end{equation}
	    where $\widetilde{C}_n:=2 H \frac{\max(\varphi^{\lambda_c}, \psi^{\lambda_c})}{\min(\varphi^{\lambda_c}, \psi^{\lambda_c})}C_n$. It follows from elementary computations that the vector function 
	    \begin{equation*}
		\big(\overline{U}(t, x), \overline{V}(t, x)\big):= \widetilde{C}_n e^{\lambda_c t}\big(\varphi^{\lambda_c}(x), \psi^{\lambda_c}(x)\big)
	    \end{equation*}
	    satisfies \eqref{eq:newvar-super} with $g_u(x, u, v)= u\big(r_u(x)-\mu_u(x)\big) + \mu_v(x) v$ and $g_v(x, u, v)= v\big(r_v(x)-\mu_v(x)\big) + \mu_u(x) u$ (there is even equality), and $\big(U(t, x), V(t, x)\big)$ satisfies \eqref{eq:newvar-sub} with the same $g$, which is cooperative. Moreover, $\big(U(t, x), V(t, x)\big)$ is globally bounded by a constant $K>0$ and there exists $T'\geq 0$ such that $\inf_{x\in\mathbb{R}}\min\big(\overline{U}(T', x), \overline{V}(T', x)\big)>K$. Thus we may apply assertion $1$ in Proposition \ref{prop:comparison-rescaled} to find 
	    \begin{equation*}
		U(t, x)\leq \widetilde{C}_n e^{\lambda_c t}\varphi^{\lambda_c }(x) \text{ and } V(t, x)\leq \widetilde{C}_n e^{\lambda_c t}\psi^{\lambda_c }(x) , \qquad \,^\forall t\in [t_n, T'], \,^\forall x\in\mathbb{R}.
	    \end{equation*}
	    Since $t_n\to -\infty$ and $T'\geq 0$, we obtain in particular 
	    \begin{equation}
		U(0, x)\leq \widetilde{C}_n \varphi^{\lambda_c }(x) \text{ and } V(t, x)\leq \widetilde{C}_n \psi^{\lambda_c }(x) , \qquad \,^\forall t\in [t_n, T'], \,^\forall x\in\mathbb{R}.
	    \end{equation}
	    Letting $n\to \infty$ and recalling $\widetilde{C}_n\to 0$, we obtain
	    \begin{equation*}
		U(0, x) = 0 \text{ and } V(0, x)=0, \qquad \,^\forall x\in\mathbb{R}.
	    \end{equation*}
	    This is obviously a contradiction, which proves \eqref{eq:250725b}.
	    \medskip

	    Next we prove \eqref{eq:250725c}. We follow the same steps. Assume by contradiction that \eqref{eq:250725c} does not hold, then there exist sequences $t_n\to -\infty$, $x_n\in\mathbb{R}$ and $C_n\to 0$ such that 
	    \begin{equation*}
		U(t_n, x_n) \leq C_n|t_n|e^{\lambda^* t_n} \varphi^{\lambda^*}(x_n) \text{ or } V(t_n, x_n)\leq C_n |t_n| e^{\lambda^* t_n}\psi^{\lambda^*}(x_n).
	    \end{equation*}
	    Up to increasing $C_n$, we will assume that $C_nt_n\to -\infty$ while keeping $C_n\to 0$ (for instance, we define $\tilde{C}_n:=\max(C_n, 1/\sqrt{-t_n})$ and drop the tilde).
	    Next, recall the upper barrier  $\big(\overline{u}^*(t, x), \overline{v}^*(t, x)\big)$ defined in \eqref{eq:critical-super}. 
	    Setting $\big(\overline{U}^*(t, x), \overline{V}^*(t, x)\big):= \big(\overline{u}^*(\frac{t+x}{c}, x), \overline{v}^*(\frac{t+x}{c}, x)\big)$ this upper barrier in the new variables 
	    satisfies the following:
	    \begin{equation}
	    \begin{split}
		    & \overline{U}^*(t, x) = (-t+\alpha) e^{\lambda^* t}\varphi^{\lambda^*}(x) - e^{\lambda^* t} \partial_\lambda \varphi^{\lambda^*}(x),\\
		    & \overline{V}^*(t, x) = (-t+\alpha) e^{\lambda^* t}\psi^{\lambda^*}(x) - e^{\lambda^* t} \partial_\lambda \psi^{\lambda^*}(x).
		 \end{split}
	    \end{equation}
		     We assume that $\alpha>0$ is sufficiently large, so that (1) $\inf \min\big(\overline{U}^*(-1/\lambda^*, x), \overline{V}^*(-1/\lambda^*,x)\big)> K$ (the value $t=-1/\lambda_c$ is chosen as the maximum of the function $t\mapsto |t|e^{\lambda^*t}$ for $t\in (-\infty, 0]$); and (2) $\alpha\geq \sup_x \max\left(\frac{\partial_\lambda \varphi^{\lambda^*}(x)}{\varphi^{\lambda^*}(x)},\frac{\partial_\lambda \psi^{\lambda^*}(x)}{\psi^{\lambda^*}(x)}\right)$, which implies that $t\mapsto\big(\overline{U}^*(t, x), \overline{V}^*(t, x)\big)$ is decreasing on $(-\infty, -1/\lambda^*]$ and  strictly positive on this interval.

	    Recalling \eqref{eq:240125-3}, the vector function $\big(\overline{U}^*(t, x), \overline{V}^*(t, x)\big)$ satisfies \eqref{eq:newvar-super} with $g_u(x, u, v)= u\big(r_u(x)-\mu_u(x)\big) + \mu_v(x) v$ and $g_v(x, u, v)= v\big(r_v(x)-\mu_v(x)\big) + \mu_u(x) u$ (there is even equality), as defined in the proof of \eqref{eq:250725b}; and  the same holds for the shifted functions  $\big(\overline{U}^*(t-T, x), \overline{V}^*(t-T, x)\big)$ for arbitrary $T\in\mathbb{R}$.

	We investigate the maximal value of $T>0$ such that $\big(\overline{U}^*(t-T, x), \overline{V}^*(t-T, x)\big)\geq \big({U}(t, x), {V}(t, x)\big)$ for all $t\in [t_n, T-1/\lambda^*]$. Clearly, from assertion $1$ in Proposition \ref{prop:comparison-rescaled}, it suffices to check that 
	    \begin{equation}\label{ineq:250919a}
\big(\overline{U}^*(t_n-T, x), \overline{V}^*(t_n-T, x)\big)\geq \big({U}(t_n, x), {V}(t_n, x)\big).
	    \end{equation}
	   For the $U$-component, a sufficient condition for the latter inequality is: 
	    \begin{align*}
		    & & \left[- (t_n-T)+\alpha\big)\varphi^{\lambda^*}(x)-\partial_\lambda \varphi^{\lambda^*}(x)\right] e^{\lambda^*(t_n-T)} &\geq C_n (-t_n)e^{\lambda^* t_n}\varphi^{\lambda^*}(x) \\
		    \Leftarrow & & \left[ -(t_n-T)e^{\lambda^*(t_n-T)} + C_nt_n e^{\lambda^* t_n}\right]\varphi^{\lambda^*}(x) & \geq 0 \\
		    \Leftrightarrow & &\lambda^*(t_n-T) e^{\lambda^*(t_n-T)} & \leq C_n \lambda^* t_n e^{\lambda^*t_n} \\
		    \Leftrightarrow & &\lambda^*(t_n-T) & \geq W_{-1}\big(C_n \lambda^* t_n e^{\lambda^*t_n}\big) \\
		    \Leftrightarrow & & T&\leq t_n - \frac{1}{\lambda^*}W_{-1}\big(C_n \lambda^* t_n e^{\lambda^*t_n}\big),
	    \end{align*}
	    where $W_{-1}(z)$ is the inverse of the function $z\mapsto ze^z$ for $z\leq -1$, also called the \textit{secondary branch of the Lambert-$W$ function}. Among other properties, $W_{-1}(z)$ is defined for $z\in[-e^{-1}, 0)$, is strictly decreasing,  $W_{-1}(-e^{-1})= -1$, and $W_{-1}(z)\to -\infty$ as $z\to 0$. 

	    We obtain a similar condition for the $V$-component. Hence \eqref{ineq:250919a} is automatically satisfied whenever 
	    \begin{equation}
		    T\leq T_n:= t_n - \frac{1}{\lambda^*}W_{-1}\big(C_n \lambda^* t_n e^{\lambda^*t_n}\big).
	    \end{equation}
	    It follows from the definition of $W_{-1}$ that $W_{-1}(z)=\ln(-z) - \ln\big(-W_{-1}(z)\big)$ for any $z\in (e^{-1}, 0)$. Therefore, by applying recursively this relation, we obtain
	    \begin{align*}
		    T_n&= t_n-\frac{1}{\lambda^*}W_{-1}\big(C_n \lambda^* t_n e^{\lambda^* t_n}\big) = t_n-\frac{1}{\lambda^*}\left[\ln\big(-C_n\lambda^* t_n e^{\lambda t_n}\big) - \ln\big(-W_{-1}\big(C_n \lambda^* t_n e^{\lambda^* t_n}\big)\big)\right] \\ 
		    &= t_n - \frac{1}{\lambda^*}\left[\lambda^* t_n + \ln\big(-C_n\lambda^* t_n\big) - \ln\big(-W_{-1}\big(C_n \lambda^* t_n e^{\lambda^* t_n}\big)\big)\right]\\
		    &=- \frac{1}{\lambda^*}\ln\big(-C_n\lambda^* t_n\big) +\frac{1}{\lambda^*} \ln\Big(-\ln\big(-C_n \lambda^* t_n e^{\lambda^* t_n}\big)  +\ln\big(-W_{-1}\big(C_n \lambda^* t_n e^{\lambda^* t_n}\big)\big)\Big) \\
		    &=\frac{1}{\lambda^*}\ln\left[\frac{-1}{C_n \lambda^* t_n}\Big(-{\lambda^* t_n}- \ln\big(-C_n \lambda^* t_n \big)  +\ln\big(-W_{-1}\big(C_n \lambda^* t_n e^{\lambda^* t_n}\big)\big)\Big)\right] \\
		    &=\frac{1}{\lambda^*}\ln\left[\frac{1}{C_n} + \frac{1}{C_n\lambda^*t_n} \ln(-C_n\lambda^* t_n) - \frac{1}{C_n\lambda^*t_n}\ln\big(-W_{-1}\big(C_n \lambda^* t_n e^{\lambda^* t_n}\big)\big)\right]\\
		    &\geq\frac{1}{\lambda^*}\ln\left[\frac{1}{C_n} + \frac{1}{C_n\lambda^*t_n} \ln(-C_n\lambda^* t_n) \right] = \frac{1}{\lambda^*} \ln\left[\frac{1}{C_n}+o_n(1)\right]\xrightarrow[n\to+\infty]{}+\infty.
	    \end{align*}
	    The $o_n(1)$ term comes from our assumption that $C_nt_n\to -\infty$, which implies that $\frac{\ln(-C_n\lambda^* t_n)}{C_n\lambda^*t_n} =o_n(1)$ as $n\to+\infty$. 

	    Summarizing, we have proved that 
	    \begin{equation*}
		    \big(\overline{U}^*(t-T_n, x), \overline{V}^*(t-T_n, x)\big)\geq \big({U}(t, x), {V}(t, x)\big), \qquad \,^\forall t\in [t_n, T_n-1/\lambda^*], x\in\mathbb{R},
	    \end{equation*}
	    thus taking the limit $n\to+\infty$, we obtain
	    \begin{equation*}
		    (0, 0)\geq \big({U}(t, x), {V}(t, x)\big), \qquad \,^\forall t\in \mathbb{R}, x\in\mathbb{R},
	    \end{equation*}
	    which is obviously a contradiction.  This proves assertion $2$ of Proposition \ref{prop:lower-bound-tail}. Since both assertions have been proved, the proof is finished. 
\end{proof}

\begin{proposition}[Upper rough bound of the leading edge]\label{prop:upper-bound-tail}
Let Assumption \ref{as:coop-comp} hold true and let $c\geq c^*_R$ be fixed. Let $\big(U(t, x), V(t, x)\big) $ be a classical, nonnegative and nontrivial solution of \eqref{eq:TW-newvar}, that vanishes when $t\to+\infty$. Then, 
    \begin{enumerate}
	\item if $c>c^*_R$, there is a constant $C>0$ and $T\in\mathbb{R}$ such that 
	    \begin{equation}\label{eq:250920b}
		U(t, x) \leq Ce^{\lambda_c t} \varphi^{\lambda_c}(x) \text{ and } V(t, x)\leq C e^{\lambda_c t}\psi^{\lambda_c}(x) ,  \qquad \,^\forall t\leq T, \,^\forall x\in \mathbb{R}, 
	    \end{equation}
	    where $\lambda_c>0$ is the smallest positive solution of $\lambda_c c = k(\lambda_c)$;
	\item 
	    if $c=c^*_R$, there is a constant $C>0$  and $T\in\mathbb{R}$ such that
	    \begin{equation}\label{eq:250920c}
		U(t, x) \leq C|t|e^{\lambda^* t} \varphi^{\lambda^*}(x) \text{ and } V(t, x)\leq C|t| e^{\lambda^* t}\psi^{\lambda^*}(x) ,  \qquad \,^\forall t\leq T, \,^\forall x\in \mathbb{R}, 
	    \end{equation}
	    where $\lambda^*>0$ is the smallest positive solution of $\lambda^* c = k(\lambda^*)$.
    \end{enumerate}
\end{proposition}
\begin{proof}
	Let us prove \eqref{eq:250920b} first. Assume by contradiction that \eqref{eq:250920b} does not hold, then there exist sequences $t_n\to -\infty$, $x_n\in\mathbb{R}$ and $C_n\to +\infty$ such that 
	    \begin{equation*}
		U(t_n, x_n) \geq C_ne^{\lambda_c t_n} \varphi^{\lambda_c}(x_n) \text{ or } V(t_n, x_n)\geq C_n e^{\lambda_c t_n}\psi^{\lambda_c}(x_n).
	    \end{equation*}
	    By the Harnack inequality \eqref{eq:est-Harnack} in Lemma \ref{lem:Harnack} we deduce
	    \begin{equation}
		U(t_n, x)> \widetilde{C}_n e^{\lambda_c t_n}\varphi^{\lambda_c }(x) \text{ and } V(t_n, x)> \widetilde{C}_n e^{\lambda_c t_n}\psi^{\lambda_c }(x) , \qquad \,^\forall x\in\mathbb{R}, 
	    \end{equation}
	    where $\widetilde{C}_n:=\frac{1}{2H} \frac{\min(\varphi^{\lambda_c}, \psi^{\lambda_c})}{\max(\varphi^{\lambda_c}, \psi^{\lambda_c})}C_n$. 
	    Let $\big(\underline{u}_{\omega}(t, x), \underline{v}_{\omega}(t,x)\big)$ be defined by \eqref{eq:ul-u}  with $\omega\geq \omega^*$ so that the differential inequality \eqref{eq:sub-sol-beta} is satisfied by applying Lemma \ref{lem:sub-solution} ($\beta$ is chosen sufficiently large so that the system \eqref{eq:auxiliary-below} is cooperative). Define
	    \begin{equation*}
		    \big(\underline{U}_{\omega}(t, x), \underline{V}_{\omega}(t, x)\big):= \left(\underline{u}_{\omega}\left(\frac{t+x}{c}, x\right), \underline{v}_{\omega}\left(\frac{t+x}{c}, x\right)\right).
	    \end{equation*}
	    In other words, 
	    \begin{equation*}
		    \underline{U}_{\omega}(t, x) = e^{\lambda_c t} \varphi^{\lambda_c}(x) - \omega e^{\nu t} \varphi^{\nu}(x) 
		\text{ and } 
		    \underline{V}_{\omega}(t, x) = e^{\lambda_c t} \psi^{\lambda_c}(x) - \omega e^{\nu t} \psi^{\nu}(x) ,
	    \end{equation*}
	    for some $\nu\in(\lambda_c, 2\lambda_c)$. It follows from Lemma \ref{lem:sub-solution} that $\underline{U}_{\omega}(t, x)<0$ and $\underline{V}_{\omega}(t, x)<0$ whenever $t\geq -\xi_\omega$, with $\xi_\omega$ defined in \eqref{eq:xi_omega}. Changing the variable in \eqref{eq:sub-sol-beta}, we obtain that $(\underline{U}_{\omega}, \underline{V}_{\omega})$ satisfies the differential inequality \eqref{eq:newvar-sub} for any $t\leq -\xi_\omega$, with
	    \begin{subequations}\label{eq:250926a}
	    \begin{align}
		g_u(x, u,v) &= u(r_u(x)-\kappa_u(x)(u+v)-\beta u)+\mu_v(x)v - \mu_u(x) u, \\ 
		g_v(x, u,v) &= v(r_v(x)-\kappa_v(x)(u+v)-\beta v)+\mu_u(x)u - \mu_v(x) v,
	    \end{align}
	    \end{subequations}
	    which is cooperative and strongly coupled (in the sense of Proposition \ref{prop:sweeping-method}) on $[0, \frac{\sup\max(r_u, r_v)}{\beta}]^2$ when $\beta\geq \frac{\sup\max(r_u, r_v, \kappa_u, \kappa_v)}{\inf\min (\mu_u, \mu_v)}$. 

	    We investigate the minimal value of $T$ such that $\big(\underline{U}_{\omega}(t-T, x), \underline{V}_{\omega}(t-T, x)\big)\leq \big(U(t, x), V(t, x)\big)$ for any $t\in [t_n, -\xi_\omega+T]$. Suppose 
	    \begin{equation*}
		T\geq T_n:=-\frac{1}{\lambda_c}\ln\big(\tilde{C}_n\big),
	    \end{equation*}
	    then we have $\tilde{C}_n e^{\lambda t_n}\varphi^{\lambda_c}(x)\geq e^{\lambda_c(t_n-T)}\varphi^{\lambda_c}(x)$ and similarly $\tilde{C}_n e^{\lambda t_n}\psi^{\lambda_c}(x)\geq e^{\lambda_c(t_n-T)}\psi^{\lambda_c}(x).$ 
	    Note that $T_n\to -\infty$ as $n\to +\infty$. In order to apply Proposition \ref{prop:sweeping-method}, we use $\omega$ as the continuous family parameter. Clearly $\big(\underline{U}_\omega(t-T, x), \underline{V}_{\omega}(t-T, x)\big)\leq (0, 0)$ for all $(t, x)\in [t_n, -\xi_\omega+T]\times \mathbb{R}$, uniformly as $\omega\to +\infty$, which shows that the assumption \eqref{eq:250928ac} holds; for any $\omega\geq \omega^*$ the assumptions \eqref{eq:250928aa} and \eqref{eq:250928ab} hold with $T_1=t_n$ and $T_2=-\xi_{\omega^*}+T$, and we have already checked that the inequality system \eqref{eq:newvar-sub} holds for any $(\omega, t, x)$ at which either $\underline{U}_\omega(t, x)>0$ or $\underline{V}_{\omega}(t, x)>0$.

	    Applying Proposition \ref{prop:sweeping-method}, we have proved that 
	    \begin{equation*} 
		\big(\underline{U}(t-T, x), \underline{V}(t-T, x)\big)\leq \big(U(t, x), V(t, x)\big), \qquad \text{ for any } T\geq T_n \text{ and } t\in [t_n, -\xi_\omega+T].
	    \end{equation*}
	    Taking the limit $n\to+\infty$, we obtain 
	    \begin{equation*}
		\big(\underline{U}(t-T, x), \underline{V}(t-T, x)\big)\leq \big(U(t, x), V(t, x)\big), \qquad \text{ for any } T\in \mathbb{R} \text{ and }t\in (-\infty, -\xi_\omega+T], 
	    \end{equation*}
which implies
    \begin{equation*}
	U(t, x) \geq \sup_{s\in\mathbb{R}} \underline{U}(s, x) >0 \text{ and } V(t, x) \geq \sup_{s\in\mathbb{R}} \underline{V}(s, x) >0. 
    \end{equation*}
    Clearly this contradicts the limit $U(t, x)\to 0 $ and $V(t, x)\to 0$ as $t\to-\infty$. This proves assertion 1 of Proposition \ref{prop:upper-bound-tail}.
    \medskip

    Next we turn to \eqref{eq:250920c}. Assume by contradiction that \eqref{eq:250920c} does not hold, then there exist sequences $t_n\to -\infty$, $x_n\in\mathbb{R}$ and $C_n\to +\infty$ such that 
	    \begin{equation*}
		U(t_n, x_n) \geq C_n|t_n|e^{\lambda_c t_n} \varphi^{\lambda_c}(x_n) \text{ or } V(t_n, x_n)\geq C_n |t_n| e^{\lambda_c t_n}\psi^{\lambda_c}(x_n).
	    \end{equation*}
	    By the Harnack inequality \eqref{eq:est-Harnack} in Lemma \ref{lem:Harnack} we deduce
	    \begin{equation}
		    U(t_n, x)> \widetilde{C}_n |t_n|e^{\lambda_c t_n}\varphi^{\lambda_c }(x) \text{ and } V(t_n, x)> \widetilde{C}_n |t_n|e^{\lambda_c t_n}\psi^{\lambda_c }(x) , \qquad \,^\forall x\in\mathbb{R}, 
	    \end{equation}
	    where $\widetilde{C}_n:=\frac{1}{2H} \frac{\min(\varphi^{\lambda_c}, \psi^{\lambda_c})}{\max(\varphi^{\lambda_c}, \psi^{\lambda_c})}C_n$. 
	    Let $\big(\underline{u}(t, x), \underline{v}(t,x)\big)$ be defined by \eqref{eq:critical-sub}  with $\omega\geq \omega^*$ so that the differential inequality \eqref{eq:critical-sub-ineq} is satisfied by applying Lemma \ref{lem:critical-diff-ineq} ($\beta$ is chosen sufficiently large so that the system \eqref{eq:auxiliary-below} is cooperative). As before we will use a family of subsolutions indexed by $\omega$, but let us omit the $\omega$ index for simplicity. Define
	    \begin{equation*}
		\big(\underline{U}(t, x), \underline{V}(t, x)\big):= \left(\underline{u}\left(\frac{t+x}{c}, x\right), \underline{v}\left(\frac{t+x}{c}, x\right)\right).
	    \end{equation*}
	    In other words, 
	    \begin{align*}
		    \underline{U}(t, x)& =\left(-t+\frac{\partial_\lambda \varphi^{\lambda}|_{\lambda^*}(x)}{\varphi^{\lambda^*}(x)} -\omega\right)e^{\lambda^* t} \varphi^{\lambda^*}(x) +  e^{\nu t} \varphi^{\nu}(x) 
		\text{ and } \\
		    \underline{V}(t, x) &=\left(-t+\frac{\partial_\lambda \psi^{\lambda}|_{\lambda^*}(x)}{\psi^{\lambda^*}(x)} -\omega\right)e^{\lambda_c t} \psi^{\lambda_c}(x) +  e^{\nu t} \psi^{\nu}(x) ,
	    \end{align*}
	    for $\nu\in(\lambda^*, 2\lambda^*)$. It follows from Lemma \ref{lem:critical-diff-ineq} that $\underline{U}(t, x)<0$ and $\underline{V}(t, x)<0$ whenever $t\geq -\xi_\omega$, with $\xi_\omega^*$ defined in \eqref{eq:xi_omega^*}. Changing the variable in \eqref{eq:critical-sub-ineq}, we obtain that $(\underline{U}, \underline{V})$ satisfies the differential inequality \eqref{eq:newvar-sub} for any $t\leq -\xi_\omega^*$, with $(g_u, g_v)$ defined by \eqref{eq:250926a}, 
	    which is cooperative on $[0, \frac{\sup\max(r_u, r_v)}{\beta}]^2$ when $\beta\geq \frac{\sup\max(r_u, r_v, \kappa_u, \kappa_v)}{\inf\min (\mu_u, \mu_v)}$. 

	We investigate the maximal value of $T>0$ such that $\big(\underline{U}(t-T, x), \underline{V}(t-T, x)\big)\leq \big(U(t, x), V(t, x)\big)$ for any $t\in [t_n, -\xi_\omega+T]$. Proceeding as in the case $c>c^*_R$ by applying Proposition \ref{prop:sweeping-method} with the family indexed by $\omega$,  it suffices to check that 
	    \begin{equation}\label{ineq:251008g}
\big(\underline{U}(t_n-T, x), \underline{V}(t_n-T, x)\big)\leq \big({U}(t_n, x), {V}(t_n, x)\big).
	    \end{equation}
	   For the $U$-component, a sufficient condition for the latter inequality is (we assume $t_n-T\leq 0$ in the computations because $\underline{U}(t, x)<0$ and $\underline{V}(t, x)<0$ when $t\geq 0$): 
	    \begin{align*}
		    & & \left(-(t_n-T)+\frac{\partial_\lambda \varphi^{\lambda}|_{\lambda^*}(x)}{\varphi^{\lambda^*}(x)} -\omega\right)e^{\lambda^* (t_n-T)} \varphi^{\lambda^*}(x) +  e^{\nu (t_n-T)} \varphi^{\nu}(x) &\leq C_n (-t_n)e^{\lambda^* t_n}\varphi^{\lambda^*}(x) \\
		    \Leftarrow & & \left[ -(t_n-T) + K_1 \right]e^{\lambda^* (t_n-T)} \varphi^{\lambda^*}(x) & \leq -C_nt_n e^{\lambda^* t_n} \varphi^{\lambda^*}(x)\\
		    \Leftrightarrow & &\lambda^*(t_n-T - K_1) e^{\lambda^*(t_n-T-K_1)} & \geq C_n \lambda^* t_n e^{\lambda^*(t_n-K_1)} \\
		    \Leftrightarrow & &\lambda^*(t_n-T-K_1) & \leq W_{-1}\big(\tilde{C}_n \lambda^* t_n e^{\lambda^*t_n}\big) \\
		    \Leftrightarrow & & t_n -K_1- \frac{1}{\lambda^*}W_{-1}\big(C_n \lambda^* t_n e^{\lambda^*t_n}\big) &\leq T,
	    \end{align*}
	    where $K_1$ is a sufficiently large constant and $\tilde{C}_n:= C_n e^{-\lambda^*K_1}$; we recall that $W_{-1}(z)$ is the inverse of the function $z\mapsto ze^z$ for $z\leq -1$, also called the \textit{secondary branch of the Lambert-$W$ function}. Among other properties, $W_{-1}(z)$ is defined for $z\in[-e^{-1}, 0)$, is strictly decreasing,  $W_{-1}(-e^{-1})= -1$, and $W_{-1}(z)\to -\infty$ as $z\to 0$. 

	    We obtain a similar condition for the $V$-component. Hence \eqref{ineq:251008g} is automatically satisfied whenever 
	    \begin{equation}
		    T\geq T_n:= t_n -K_1 - \frac{1}{\lambda^*}W_{-1}\big(\tilde{C}_n \lambda^* t_n e^{\lambda^*t_n}\big).
	    \end{equation}
	    It follows from the definition of $W_{-1}$ that $W_{-1}(z)=\ln(-z) - \ln\big(-W_{-1}(z)\big)$ for any $z\in (e^{-1}, 0)$. Therefore, 
	    \begin{align}
		    \nonumber T_n&= t_n-K_1-\frac{1}{\lambda^*}W_{-1}\big(\tilde{C}_n \lambda^* t_n e^{\lambda^* t_n}\big) \\
		    \nonumber &= t_n-K_1-\frac{1}{\lambda^*}\left[\ln\big(-\tilde{C}_n\lambda^* t_n e^{\lambda t_n}\big) - \ln\big(-W_{-1}\big(\tilde{C}_n \lambda^* t_n e^{\lambda^* t_n}\big)\big)\right] \\ 
		    \nonumber &= t_n-K_1 - \frac{1}{\lambda^*}\left[\lambda^* t_n + \ln\big(-\tilde{C}_n\lambda^* t_n\big) - \ln\big(-W_{-1}\big(\tilde{C}_n \lambda^* t_n e^{\lambda^* t_n}\big)\big)\right]\\
		    \nonumber &= -K_1- \frac{1}{\lambda^*}\ln\big(-\tilde{C}_n\lambda^* t_n\big) +\frac{1}{\lambda^*} \ln\Big(-\ln\big(-\tilde{C}_n \lambda^* t_n e^{\lambda^* t_n}\big)  +\ln\big(-W_{-1}\big(\tilde{C}_n \lambda^* t_n e^{\lambda^* t_n}\big)\big)\Big) \\
		    \nonumber &=-K_1+\frac{1}{\lambda^*}\ln\left[\frac{-1}{\tilde{C}_n \lambda^* t_n}\Big(-{\lambda^* t_n}- \ln\big(-\tilde{C}_n \lambda^* t_n \big)  +\ln\big(-W_{-1}\big(\tilde{C}_n \lambda^* t_n e^{\lambda^* t_n}\big)\big)\Big)\right] \\
		     \label{eq:251003a}&=-K_1+\frac{1}{\lambda^*}\ln\left[\frac{1}{\tilde{C}_n} - \frac{1}{\tilde{C}_n\lambda^*t_n} \ln\left( \frac{-W_{-1}\big(\tilde{C}_n \lambda^* t_n e^{\lambda^* t_n}\big)}{-\tilde{C}_n\lambda^* t_n}\right)\right].
%		    &=K_1+\frac{1}{\lambda^*}\ln\left[\frac{1}{\tilde{C}_n} + \frac{1}{\tilde{C}_n\lambda^*t_n} \ln(-\tilde{C}_n\lambda^* t_n) - \frac{1}{\tilde{C}_n\lambda^*t_n}\ln\big(-W_{-1}\big(\tilde{C}_n \lambda^* t_n e^{\lambda^* t_n}\big)\big)\right]\\
%		    &\geq\frac{1}{\lambda^*}\ln\left[\frac{1}{\tilde{C}_n} + \frac{1}{\tilde{C}_n\lambda^*t_n} \ln(-\tilde{C}_n\lambda^* t_n) \right] = \frac{1}{\lambda^*} \ln\left[\frac{1}{\tilde{C}_n}+o_n(1)\right]\xrightarrow[n\to+\infty]{}+\infty.
	    \end{align}
	    Let $0<\varepsilon< e^{-1}$ be given, it follows from elementary computations that $xe^x \geq -e^{(1-\varepsilon) x}$ whenever $x\leq x_\varepsilon:=\frac{1}{\varepsilon}W_{-1}(-\varepsilon)$; thus $W_{-1}(x)\geq \frac{1}{1-\varepsilon}\ln(-x)$ for $x\in (x_\varepsilon e^{x_\varepsilon}, 0)$. Hence,  
	    \begin{align*}
		    -W_{-1}\big(\tilde{C}_n \lambda^* t_n e^{\lambda^* t_n}\big) &\leq -\frac{1}{1-\varepsilon} \ln\big( -\tilde{C}_n \lambda^* t_n e^{\lambda^* t_n}\big) = -\frac{\lambda^* t_n}{1-\varepsilon}+\frac{-\ln\big(-\tilde{C}_n\lambda^*t_n\big)}{1-\varepsilon}.
	    \end{align*}
	    Coming back to \eqref{eq:251003a}, we conclude 
	    \begin{align*}
		    T_n&\leq -K_1+\frac{1}{\lambda^*}\ln\left[\frac{1}{\tilde{C}_n}+\frac{-1}{\tilde{C}_n\lambda^*t_n} \ln\left(\frac{1}{(1-\varepsilon)\tilde{C}_n} + \frac{-\ln\big(-\tilde{C}_n\lambda^*t_n\big)}{-(1-\varepsilon)\tilde{C}_n\lambda^*t_n}\right) \right] \\ 
		    &=-K_1+\frac{1}{\lambda^*}\ln\left[\frac{1}{\tilde{C}_n}+\frac{-1}{\lambda^*t_n} \frac{\ln\big((1-\varepsilon)\tilde{C}_n+o_n(1)\big)}{\tilde{C}_n}\right] = K_1+\frac{1}{\lambda^*}\ln\left[\frac{1}{\tilde{C}_n}+o_n(1)\right].
	    \end{align*}
	    This shows that $T_n\to -\infty$ as $n\to+\infty$. Hence, applying Proposition \ref{prop:sweeping-method} as in the previous step, we obtain 
	    \begin{equation*} 
		\big(\underline{U}(t-T, x), \underline{V}(t-T, x)\big)\leq \big(U(t, x), V(t, x)\big), \qquad \text{ for any } T\geq T_n \text{ and } t\in [t_n, -\xi_\omega+T].
	    \end{equation*}
	    Taking the limit $n\to+\infty$, we obtain 
	    \begin{equation*}
		\big(\underline{U}(t-T, x), \underline{V}(t-T, x)\big)\leq \big(U(t, x), V(t, x)\big), \qquad \text{ for any } T\in \mathbb{R} \text{ and }t\in (-\infty, -\xi_\omega+T], 
	    \end{equation*}
which implies
    \begin{equation*}
	U(t, x) \geq \sup_{s\in\mathbb{R}} \underline{U}(s, x) >0 \text{ and } V(t, x) \geq \sup_{s\in\mathbb{R}} \underline{V}(s, x) >0. 
    \end{equation*}
    Clearly this contradicts the limit $U(t, x)\to 0 $ and $V(t, x)\to 0$ as $t\to-\infty$. This proves assertion 2 and finishes the proof of Proposition \ref{prop:upper-bound-tail}.
\end{proof}
In the next Proposition, we adapt the proof of Proposition \ref{prop:upper-bound-tail} to show that the minimum of $\big(U(t, x), V(t, x)\big)$ behaves like a true exponential when $t\to -\infty$.
\begin{lemma}[Exponential decay of the minimum]\label{lem:exponential-decay-minimum}
Let Assumption \ref{as:coop-comp} hold true and let $c\geq c^*_R$ be fixed. Let $\big(U(t, x), V(t, x)\big) $ be a classical, nonnegative and nontrivial solution of \eqref{eq:TW-newvar}, that vanishes when $t\to+\infty$. Then, 
    \begin{enumerate}
	\item if $c>c^*_R$, then  the following limit exists and is positive:
	    \begin{equation}\label{eq:251005a}
		    \lim_{t\to-\infty} \inf_{x\in\mathbb{R}}\min\left(\dfrac{U(t, x)}{\varphi^{\lambda_c}(x)e^{\lambda_c t}}, 
		    \dfrac{V(t, x)}{\psi^{\lambda_c}(x)e^{\lambda_c t}}\right) >0,
	    \end{equation}
	    where $\lambda_c>0$ is the smallest positive solution of $\lambda_c c = k(\lambda_c)$;
	\item 
	    if $c=c^*_R$, then  the following limit exists and is positive:
		    \begin{equation}\label{eq:251005b}
		    \lim_{t\to-\infty} \inf_{x\in\mathbb{R}}\min\left(\dfrac{U(t, x)}{-t\varphi^{\lambda^*}(x)e^{\lambda^* t}}, 
		    \dfrac{V(t, x)}{-t\psi^{\lambda^*}(x)e^{\lambda^* t}}\right) >0,
	    \end{equation}
	    where $\lambda^*>0$ is the smallest positive solution of $\lambda^* c = k(\lambda^*)$.
    \end{enumerate}
\end{lemma}
\begin{proof}
	We first consider assertion 1. 
	From Proposition \ref{prop:upper-bound-tail} we may assume without loss of generality (up to a translation in time) that 
	\begin{equation}\label{eq:251007b}
		\liminf_{t\to-\infty}\inf_{x\in\mathbb{R}}\min\left(\dfrac{U(t, x)}{\varphi^{\lambda_c}(x)e^{\lambda_c t}}, 
		    \dfrac{V(t, x)}{\psi^{\lambda_c}(x)e^{\lambda_c t}}\right) = 1.
	\end{equation}
	Suppose by contradiction that there exist $\varepsilon>0$ and a sequence $s_n\to -\infty$ such that 	
	\begin{equation*}
		%\inf_{x\in\mathbb{R}}\min\left(\dfrac{U(t_n, x)}{\varphi^{\lambda_c}(x)e^{\lambda_c t_n}},
		%\dfrac{V(t_n, x)}{\psi^{\lambda_c}(x)e^{\lambda_c t_n}}\right) \xrightarrow[n\to+\infty]{}  1 \text{ and }
		\inf_{x\in\mathbb{R}}\min\left(\dfrac{U(s_n, x)}{\varphi^{\lambda_c}(x)e^{\lambda_c s_n}}, 
		    \dfrac{V(s_n, x)}{\psi^{\lambda_c}(x)e^{\lambda_c s_n}}\right) \geq 1 + \varepsilon.
	\end{equation*}

	    %Let $\big(\underline{u}_{\omega}(t, x), \underline{v}_{\omega}(t,x)\big)$ be defined by \eqref{eq:ul-u}  with $\omega\geq \omega^*$ so that the differential inequality \eqref{eq:sub-sol-beta} is satisfied by applying Lemma \ref{lem:sub-solution} ($\beta$ is chosen sufficiently large so that the system \eqref{eq:auxiliary-below} is cooperative). 
	    We define
	    \begin{equation*}
		    \underline{U}_{\omega}(t, x) = e^{\lambda_c t} \varphi^{\lambda_c}(x) - \omega e^{\nu t} \varphi^{\nu}(x) 
		\text{ and } 
		    \underline{V}_{\omega}(t, x) = e^{\lambda_c t} \psi^{\lambda_c}(x) - \omega e^{\nu t} \psi^{\nu}(x) ,
	    \end{equation*}
	    for some $\nu\in(\lambda_c, 2\lambda_c)$. 
	    As in the proof of Proposition \ref{prop:upper-bound-tail}, 
	    it follows from Lemma \ref{lem:sub-solution} that $\underline{U}_{\omega}(t, x)<0$ and $\underline{V}_{\omega}(t, x)<0$ whenever $t\geq -\xi_\omega$, with $\xi_\omega$ defined in \eqref{eq:xi_omega},
	    and we obtain that $(\underline{U}_{\omega}, \underline{V}_{\omega})$ satisfies the differential inequality \eqref{eq:newvar-sub} for any $t\leq -\xi_\omega$, with $(g_u, g_v)$ defined by \eqref{eq:250926a}.

	    We investigate the minimal value of $T$ such that $\big(\underline{U}_{\omega}(t-T, x), \underline{V}_{\omega}(t-T, x)\big)\leq \big(U(t, x), V(t, x)\big)$ for any $t\in [s_n, -\xi_\omega+T]$. Proceeding as in the proof of Proposition \ref{prop:upper-bound-tail} by applying Proposition \ref{prop:sweeping-method} (sweeping method) with the family indexed by $\omega$,  it suffices to check that 
	    \begin{equation}\label{ineq:251007a}
\big(\underline{U}(s_n-T, x), \underline{V}(s_n-T, x)\big)\leq \big({U}(s_n, x), {V}(s_n, x)\big).
	    \end{equation}
	    Clearly this is the case for any  $ T\geq T_\varepsilon:=-\frac{1}{\lambda_c}\ln\big(1+\varepsilon\big)$, which shows that 
	    \begin{equation*}
		    \big(\underline{U}(t-T_\varepsilon, x), \underline{V}(t-T_\varepsilon, x)\big)=\big((1+\varepsilon)e^{\lambda_c t}-\omega (1+\varepsilon)^{\frac{\nu}{\lambda}}e^{\nu t}\big)\big(\varphi^{\lambda_c}(x), \psi^{\lambda_c}(x)\big) \leq \big(U(t, x), V(t, x)\big), 
	    \end{equation*}
	    and in particular 
	    \begin{equation*}
		    \liminf_{t\to-\infty}\inf_{x\in\mathbb{R}}\min\left(\frac{U(t, x)}{e^{\lambda_c t}\varphi^{\lambda_c}(x)}, \frac{V(t, x)}{e^{\lambda_c t}\psi^{\lambda_c}(x)}\right) \geq 1+\varepsilon, 
	    \end{equation*}
	    which contradicts \eqref{eq:251007b}. The contradiction proves assertion 1.
	    \medskip

%%%%%%%%%%%%%%%%%%%%%%%%%%%%%%%%%%%%%%%%%%%%%%%%%%%%%%%%%%%%%%%%%%%%%%
Next we consider assertion 2. Without loss of generality we assume that 
	\begin{equation}\label{eq:251007c}
		\liminf_{t\to-\infty}\inf_{x\in\mathbb{R}}\min\left(\dfrac{U(t, x)}{-te^{\lambda^* t}\varphi^{\lambda^*}(x)}, 
		    \dfrac{V(t, x)}{-te^{\lambda^* t}\psi^{\lambda^*}(x)}\right) = 1,
	\end{equation}
	and we suppose by contradiction that 
	there is $\varepsilon>0$ and a sequence $s_n$ such that 	
	\begin{equation*}
		%\inf_{x\in\mathbb{R}}\min\left(\dfrac{U(t_n, x)}{\varphi^{\lambda^*}(x)e^{\lambda^* t_n}},
		%\dfrac{V(t_n, x)}{\psi^{\lambda^*}(x)e^{\lambda^* t_n}}\right) \xrightarrow[n\to+\infty]{}  1 \text{ and }
		\inf_{x\in\mathbb{R}}\min\left(\dfrac{U(s_n, x)}{-s_ne^{\lambda^* s_n}\varphi^{\lambda^*}(x)}, 
		    \dfrac{V(s_n, x)}{-s_ne^{\lambda^* s_n}\psi^{\lambda^*}(x)}\right) \geq 1 + \varepsilon.
	\end{equation*}
	We define
	    \begin{align*}
		    \underline{U}(t, x) &=\left(-t+\frac{\partial_\lambda \varphi^{\lambda}|_{\lambda^*}(x)}{\varphi^{\lambda^*}(x)} -\omega\right)e^{\lambda^* t} \varphi^{\lambda^*}(x) +  e^{\nu t} \varphi^{\nu}(x) 
		\text{ and } \\
		    \underline{V}(t, x) &=\left(-t+\frac{\partial_\lambda \psi^{\lambda}|_{\lambda^*}(x)}{\psi^{\lambda^*}(x)} -\omega\right)e^{\lambda^* t} \psi^{\lambda^*}(x) +  e^{\nu t} \psi^{\nu}(x) ,
	    \end{align*}
	    for $\nu\in(\lambda^*, 2\lambda^*)$. 
	    As in the proof of Proposition \ref{prop:upper-bound-tail}, 
	    it follows from Lemma \ref{lem:critical-diff-ineq} that $\underline{U}(t, x)<0$ and $\underline{V}(t, x)<0$ whenever $t\geq -\xi_\omega$, with $\xi_\omega^*$ defined in \eqref{eq:xi_omega^*}, 
	    and we obtain that $(\underline{U}, \underline{V})$ satisfies the differential inequality \eqref{eq:newvar-sub} for any $t\leq -\xi_\omega^*$, with $(g_u, g_v)$ defined by \eqref{eq:250926a}, 
	    which is cooperative on $[0, \frac{\sup\max(r_u, r_v)}{\beta}]^2$ when $\beta\geq \frac{\sup\max(r_u, r_v, \kappa_u, \kappa_v)}{\inf\min (\mu_u, \mu_v)}$. 

	We investigate the maximal value of $T>0$ such that $\big(\underline{U}(t-T, x), \underline{V}(t-T, x)\big)\leq \big(U(t, x), V(t, x)\big)$ for any $t\in [s_n, -\xi_\omega+T]$. Proceeding as in the proof of Proposition \ref{prop:upper-bound-tail} by applying Proposition \ref{prop:sweeping-method} with the family indexed by $\omega$,  it suffices to check that 
	    \begin{equation}\label{ineq:251007c}
\big(\underline{U}(s_n-T, x), \underline{V}(s_n-T, x)\big)\leq \big({U}(s_n, x), {V}(s_n, x)\big).
	    \end{equation}
	   For the $U$-component, a sufficient condition for the latter inequality is (we assume $s_n-T\leq 0$ in the computations because $\underline{U}(t, x)<0$ and $\underline{V}(t, x)<0$ when $t\geq 0$): 
	    \begin{align*}
		    & & \left(-(s_n-T)+\frac{\partial_\lambda \varphi^{\lambda}|_{\lambda^*}(x)}{\varphi^{\lambda^*}(x)} -\omega\right)e^{\lambda^* (s_n-T)} \varphi^{\lambda^*}(x) +  e^{\nu (s_n-T)} \varphi^{\nu}(x) &\leq (1+\varepsilon) (-s_n)e^{\lambda^* s_n}\varphi^{\lambda^*}(x) \\
		    \Leftarrow & & \left[ -(s_n-T)e^{\lambda^*(s_n-T)} + K_1 \right]e^{\lambda^* (s_n-T)} \varphi^{\lambda^*}(x) & \leq -(1+\varepsilon)s_n e^{\lambda^* s_n} \varphi^{\lambda^*}(x)\\
		    \Leftrightarrow & &\lambda^*(s_n-T - K_1) e^{\lambda^*(s_n-T-K_1)} & \geq (1+\varepsilon) \lambda^* s_n e^{\lambda^*(s_n-K_1)} \\
		    \Leftrightarrow & &W_{-1}\big((1+\varepsilon)e^{-\lambda^*K_1} \lambda^* s_n e^{\lambda^*s_n}\big) & \geq \lambda^*(s_n-T-K_1) \\
		    \Leftrightarrow & & s_n-K_1 - \frac{1}{\lambda^*}W_{-1}\big((1+\varepsilon)e^{-\lambda^*K_1}\lambda^* s_n e^{\lambda^*s_n}\big) &\leq T,
	    \end{align*}
	    where $K_1$ is a sufficiently large constant. 

	    We obtain a similar condition for the $V$-component. Hence \eqref{ineq:251007c} is automatically satisfied whenever 
	    \begin{equation}
		    T\geq T_n:= s_n-K_1 - \frac{1}{\lambda^*}W_{-1}\big((1+\varepsilon)e^{-\lambda^*K_1} \lambda^* s_n e^{\lambda^*s_n}\big).
	    \end{equation}
	    It follows from the definition of $W_{-1}$ that $W_{-1}(z)=\ln(-z) - \ln\big(-W_{-1}(z)\big)$ for any $z\in (e^{-1}, 0)$. Let $\tilde{C}_\varepsilon:=(1+\varepsilon)e^{-K_1}$, then 
	    \begin{align}
		    \nonumber T_n&=- K_1+s_n-\frac{1}{\lambda^*}W_{-1}\big(\tilde{C}_\varepsilon \lambda^* s_n e^{\lambda^* s_n}\big) \\
		    \nonumber &= -K_1+s_n-\frac{1}{\lambda^*}\left[\ln\big(-\tilde{C}_\varepsilon\lambda^* s_n e^{\lambda s_n}\big) - \ln\big(-W_{-1}\big(\tilde{C}_\varepsilon \lambda^* s_n e^{\lambda^* s_n}\big)\big)\right] \\ 
		    \nonumber &= -K_1+s_n - \frac{1}{\lambda^*}\left[\lambda^* s_n + \ln\big(-\tilde{C}_\varepsilon\lambda^* s_n\big) - \ln\big(-W_{-1}\big(\tilde{C}_\varepsilon \lambda^* s_n e^{\lambda^* s_n}\big)\big)\right]\\
		    \nonumber &=-K_1- \frac{1}{\lambda^*}\ln\big(-\tilde{C}_\varepsilon\lambda^* s_n\big) +\frac{1}{\lambda^*} \ln\Big(-\ln\big(-\tilde{C}_\varepsilon \lambda^* s_n e^{\lambda^* s_n}\big)  +\ln\big(-W_{-1}\big(\tilde{C}_\varepsilon \lambda^* s_n e^{\lambda^* s_n}\big)\big)\Big) \\
		    \nonumber &=-K_1+\frac{1}{\lambda^*}\ln\left[\frac{-1}{\tilde{C}_\varepsilon \lambda^* s_n}\Big(-{\lambda^* s_n}- \ln\big(-\tilde{C}_\varepsilon \lambda^* s_n \big)  +\ln\big(-W_{-1}\big(\tilde{C}_\varepsilon \lambda^* s_n e^{\lambda^* s_n}\big)\big)\Big)\right] \\
		     \label{eq:251007d}&=-K_1+\frac{1}{\lambda^*}\ln\left[\frac{1}{\tilde{C}_\varepsilon} - \frac{1}{\tilde{C}_\varepsilon\lambda^*s_n} \ln\left( \frac{-W_{-1}\big(\tilde{C}_\varepsilon \lambda^* s_n e^{\lambda^* s_n}\big)}{-\tilde{C}_\varepsilon\lambda^* s_n}\right)\right].
%		    &=K_1+\frac{1}{\lambda^*}\ln\left[\frac{1}{\tilde{C}_\varepsilon} + \frac{1}{\tilde{C}_\varepsilon\lambda^*s_n} \ln(-\tilde{C}_\varepsilon\lambda^* s_n) - \frac{1}{\tilde{C}_\varepsilon\lambda^*s_n}\ln\big(-W_{-1}\big(\tilde{C}_\varepsilon \lambda^* s_n e^{\lambda^* s_n}\big)\big)\right]\\
%		    &\geq\frac{1}{\lambda^*}\ln\left[\frac{1}{\tilde{C}_\varepsilon} + \frac{1}{\tilde{C}_\varepsilon\lambda^*s_n} \ln(-\tilde{C}_\varepsilon\lambda^* s_n) \right] = \frac{1}{\lambda^*} \ln\left[\frac{1}{\tilde{C}_\varepsilon}+o_n(1)\right]\xrightarrow[n\to+\infty]{}+\infty.
	    \end{align}
	    Let $0<\mu< e^{-1}$ be given, it follows from elementary computations that $xe^x \geq -e^{(1-\mu) x}$ whenever $x\leq x_\mu:=\frac{1}{\mu}W_{-1}(-\mu)$; thus $W_{-1}(x)\geq \frac{1}{1-\mu}\ln(-x)$ for $x\in (x_\mu e^{x_\mu}, 0)$. Hence,  
	    \begin{align*}
		    -W_{-1}\big(\tilde{C}_\varepsilon \lambda^* s_n e^{\lambda^* s_n}\big) &\leq -\frac{1}{1-\mu} \ln\big( -\tilde{C}_\varepsilon \lambda^* s_n e^{\lambda^* s_n}\big) = -\frac{\lambda^* s_n}{1-\mu}+\frac{-\ln\big(-\tilde{C}_\varepsilon\lambda^*s_n\big)}{1-\mu}.
	    \end{align*}
	    Coming back to \eqref{eq:251007d}, we conclude 
	    \begin{align*}
		    T_n&\leq -K_1+\frac{1}{\lambda^*}\ln\left[\frac{1}{\tilde{C}_\varepsilon}+\frac{-1}{\tilde{C}_\varepsilon\lambda^*s_n} \ln\left(\frac{1}{(1-\mu)\tilde{C}_\varepsilon} + \frac{-\ln\big(-\tilde{C}_\varepsilon\lambda^*s_n\big)}{-(1-\mu)\tilde{C}_\varepsilon\lambda^*s_n}\right) \right] \\ 
		    &=-K_1+\frac{1}{\lambda^*}\ln\left[\frac{1}{\tilde{C}_\varepsilon}+\frac{-1}{\lambda^*s_n} \frac{\ln\big((1-\mu)\tilde{C}_\varepsilon+o_n(1)\big)}{\tilde{C}_\varepsilon}\right] = \frac{1}{\lambda^*}\ln\left[\frac{1}{1+\varepsilon}\right]+o_n(1).
	    \end{align*}
	    This shows that $T_n\to T_*\leq -\frac{1}{\lambda^*}\ln(1+\varepsilon)$ as $n\to+\infty$. Hence, applying Proposition \ref{prop:sweeping-method} and taking the limit $n\to+\infty$ as in the previous step, we obtain 
	    \begin{multline*} 
		    (1+\varepsilon)\left[-(t-T_*)+\frac{\partial_\lambda \varphi^{\lambda}|_{\lambda^*}(x)}{\varphi^{\lambda^*}(x)} -\omega\right]e^{\lambda^* t}\varphi^{\lambda^*}(x)+(1+\varepsilon)^{\frac{\lambda}{\nu}}e^{\nu t} \varphi^\nu(x) \leq  U(t, x)  \\ 
		    \text{ and }(1+\varepsilon)\left[-(t-T_*)+\frac{\partial_\lambda \psi^{\lambda}|_{\lambda^*}(x)}{\psi^{\lambda^*}(x)} -\omega\right]e^{\lambda^* t}\psi^{\lambda^*}(x)+(1+\varepsilon)^{\frac{\lambda}{\nu}}e^{\nu t} \psi^\nu(x) \leq  V(t, x)  ,
	    \end{multline*}
	    for all $t\in\mathbb{R}$ and $x\in\mathbb{R}$, 
	    and in particular 
	    \begin{equation*}
		    \liminf_{t\to-\infty}\inf_{x\in\mathbb{R}}\min\left(\frac{U(t, x)}{-te^{\lambda^* t}\varphi^{\lambda^*}(x)}, \frac{V(t, x)}{-te^{\lambda^* t}\psi^{\lambda^*}(x)}\right) \geq 1+\varepsilon, 
	    \end{equation*}
	    which contradicts \eqref{eq:251007c}. The contradiction proves assertion 2. The proof of Lemma \ref{lem:exponential-decay-minimum} is finished.
\end{proof}

\begin{proof}[Proof of Theorem \ref{thm:TW2}] 
	Let $(u, v)$ be a traveling wave solution for \eqref{eq:main-sys} and $c\geq c^*_R$ be given. We define 
	    \begin{equation*}
		    \big({U}(t, x), {V}(t, x)\big):= \left({u}\left(\frac{t+x}{c}, x\right), {v}\left(\frac{t+x}{c}, x\right)\right)
	    \end{equation*}
	    so that $\big(U(t, x), V(t, x)\big)$ is a classical solution of \eqref{eq:TW-newvar}. 
	We set $e(t)= e^{\lambda_c t}$ and $\lambda=\lambda_c$ if $c>c^*_R$ and $e(t)= -t e^{\lambda^*t}$ and $\lambda=\lambda^*$ if $c=c^*_R$.
	    From Proposition \ref{prop:lower-bound-tail} and proposition \ref{prop:upper-bound-tail} we know that 
	    \begin{equation}\label{eq:251008a}
		    0<\liminf_{t\to-\infty}\inf_{x\in\mathbb{R}}\min\left(\frac{U(t, x)}{e(t)\varphi^\lambda(x)}, \frac{V(t, x)}{e(t)\psi^\lambda(x)}\right) \leq  
		    \limsup_{t\to-\infty}\sup_{x\in\mathbb{R}}\max\left(\frac{U(t, x)}{e(t)\varphi^\lambda(x)}, \frac{V(t, x)}{e(t)\psi^\lambda(x)}\right)<+\infty, 
	    \end{equation}
	    and from Lemma \ref{lem:exponential-decay-minimum}  we know moreover that  
	    \begin{equation} \label{eq:251008d}
		    \lim_{t\to-\infty}\inf_{x\in\mathbb{R}}\min\left(\frac{U(t, x)}{e(t)\varphi^\lambda(x)}, \frac{V(t, x)}{e(t)\psi^\lambda(x)}\right) =1,   
	    \end{equation}
	up to a translation in time. Suppose by contradiction that there exists a $\varepsilon>0$ and a sequence $t_n, x_n$ with $t_n\to -\infty$ and $x_n\in [0, L] $ such that 
	\begin{equation}\label{eq:251007e}
		\max\left(\frac{U(t_n, x_n)}{e(t_n)\varphi^\lambda(x_n)}, \frac{V(t_n, x_n)}{e(t_n)\psi^\lambda(x_n)}\right) \geq 1+\varepsilon.
	\end{equation}
	We define $\big(\tilde{U}_n(t, x), \tilde{V}_n(t, x)\big):= \left(\frac{U(t+t_n, x)}{e(t+t_n)\varphi^\lambda(x)}, \frac{V(t+t_n, x)}{e(t+t_n)\psi^\lambda(x)}\right)$  and $\big(\tilde{u}_n(t, x), \tilde{v}_n(t, x)\big) = \big(\tilde{U}_n(ct-x, x), \tilde{V}_n(ct-x, x)\big)$. Then the functions $\big(\tilde{u}_n(t, x), \tilde{v}_n(t, x)\big)$ can be rewritten as  
	\begin{align*}
	    \tilde{u}_n(t, x)&= \tilde{U}_n(ct-x, x)= \dfrac{U(ct-x+t_n, x)}{e(ct-x+t_n) \varphi^{\lambda}(x)} = 
	    \dfrac{u\left(\frac{ct-x+t_n+x}{c}, x\right)}{e(ct-x+t_n) \varphi^{\lambda}(x)} = 
	    \dfrac{u\left(t+\frac{t_n}{c}, x\right)}{e(ct-x+t_n) \varphi^{\lambda}(x)}, \\ 
	    \tilde{v}_n(t, x)&= \tilde{V}_n(ct-x, x)= \dfrac{V(ct-x+t_n, x)}{e(ct-x+t_n) \psi^{\lambda}(x)} = 
	    \dfrac{v\left(\frac{ct-x+t_n+x}{c}, x\right)}{e(ct-x+t_n) \psi^{\lambda}(x)} = 
	    \dfrac{v\left(t+\frac{t_n}{c}, x\right)}{e(ct-x+t_n) \psi^{\lambda}(x)}, 
	\end{align*}
	and therefore satisfy (we use short notations and omit the function's arguments in the next computation: $e=e(ct-x+t_n)$, $\varphi=\varphi^\lambda(x)$, and the prime to denote a function's first derivative)
	\begin{subequations}\label{eq:251008b}
	\begin{align}
	    \nonumber \partial_t \tilde{u}_n &=\frac{\partial_tu}{e\varphi} - c\tilde{u}_n \frac{e'}{e}=\frac{\partial_x(\sigma \partial_x u)+f_u(x, u,v)}{e\varphi} - c\tilde{u}_n \frac{e'}{e} \\ 
	    \nonumber &= \partial_x (\sigma \partial_x \tilde{u}_n) + 2\sigma \frac{\partial_x (e\varphi)}{e\varphi} \partial_x\tilde{u}_n + \tilde{u}_n \frac{\partial_x (\sigma\partial_x(e\varphi))}{e\varphi}+ \frac{f_u(x, u, v)}{e\varphi} - c\tilde{u}_n\frac{e'}{e} \\
	    \nonumber &= \partial_x (\sigma \partial_x \tilde{u}_n) + 2\sigma \left(-\frac{e'}{e}+\frac{\varphi'}{\varphi}\right) \partial_x\tilde{u}_n + \tilde{u}_n\left[ \frac{e\partial_x (\sigma\partial_x\varphi) - 2e'\sigma \partial_x \varphi +(-e'\partial_x \sigma + \sigma e'' - ce' )\varphi }{e\varphi} \right]\\ 
	     &\quad +  \frac{f_u(x, u, v)}{e\varphi}, \\
	     \nonumber \partial_t \tilde{v}_n &=\partial_x (\sigma \partial_x \tilde{v}_n) + 2\sigma \left(-\frac{e'}{e}+\frac{\psi'}{\psi}\right) \partial_x\tilde{v}_n + \tilde{v}_n\left[ \frac{e\partial_x (\sigma\partial_x\psi) - 2e'\sigma \partial_x \psi +(-e'\partial_x \sigma + \sigma e'' - ce' )\psi }{e\psi} \right]\\ 
	    &\quad +  \frac{f_v(x, u, v)}{e\psi}.
	\end{align}
	\end{subequations}
	From \eqref{eq:251008a} and since $t_n\to -\infty$, the functions $(\tilde{u}_n, \tilde{v}_n)$ are locally bounded, and $\big(u(t+t_n/c, x), v(t+t_n/c, x)\big)\to (0, 0)$ as $n\to+\infty$, locally uniformly in $(t, x)$. Note that 
	\begin{equation*}
	    \frac{e'}{e} =\frac{e'(ct-x+t_n)}{e(ct-x+t_n)}= \lambda +\mathcal{O}\left(\frac{1}{ct-x+t_n}\right) , \qquad \frac{e''}{e}=\lambda^2 +\mathcal{O}\left(\frac{1}{ct-x+t_n}\right),
	\end{equation*}
	and 
	\begin{align*}
	    \frac{f_u(x, u, v)}{e\varphi} &=\dfrac{u\big(r_u-\kappa_u ( u+v)\big) + \mu_v v -\mu_u u }{e\varphi}= \tilde{u}_n \big(r_u-\mu_u-\kappa_u(u+v)\big) +\mu_v \frac{\psi}{\varphi}\tilde{v}_n, \\ 
	    \frac{f_v(x, u, v)}{e\psi} &=\dfrac{v\big(r_v-\kappa_v ( u+v)\big) + \mu_u u -\mu_v v }{e\varphi}= \tilde{v}_n \big(r_v-\mu_v-\kappa_v(u+v)\big) +\mu_u \frac{\varphi}{\psi}\tilde{v}_n, 
	\end{align*}
	so \eqref{eq:251008b} is a linear parabolic equation with locally bounded coefficients. 
	    From standard parabolic estimates, the functions $(\tilde{u}_n, \tilde{v}_n)$ converge in $C^{1, 2}_{t, x; loc}(\mathbb{R}\times\mathbb{R})$, to a vector function $(\tilde{u}_\infty, \tilde{v}_\infty)$ satisfying
	    \begin{subequations}\label{eq:251008c}
	    \begin{align}
		\nonumber \partial_t \tilde{u}_\infty &= \partial_x (\sigma \partial_x \tilde{u}_\infty) + 2\sigma \left(-\lambda +\frac{\varphi'}{\varphi}\right) \partial_x\tilde{u}_\infty + \tilde{u}_\infty\left[ \frac{\partial_x (\sigma\partial_x\varphi) - 2\lambda\sigma \partial_x \varphi +(-\lambda \partial_x \sigma + \lambda^2\sigma  - \lambda c  )\varphi }{\varphi} \right]\\ 
		\nonumber&\quad +\tilde{u}_\infty \big(r_u-\mu_u\big) +\mu_v \frac{\psi}{\varphi}\tilde{v}_\infty\\
		&= \partial_x (\sigma \partial_x \tilde{u}_\infty) + 2\sigma \left(-\lambda +\frac{\varphi'}{\varphi}\right) \partial_x\tilde{u}_\infty + \tilde{u}_\infty\underset{=0}{\underbrace{\big[k(\lambda)-\lambda c\big]}} + \mu_v \frac{\psi}{\varphi}(\tilde{v}_\infty-\tilde{u}_\infty), \label{eq:251008ca} \\ 
		\partial_t \tilde{v}_\infty  
		&= \partial_x (\sigma \partial_x \tilde{v}_\infty) + 2\sigma \left(-\lambda +\frac{\psi'}{\psi}\right) \partial_x\tilde{v}_\infty + \mu_v \frac{\varphi}{\psi}(\tilde{u}_\infty-\tilde{v}_\infty), \label{eq:251008cb} 
	    \end{align}
	    \end{subequations}
	    and moreover (up to further extraction) $x_n\to x_{\infty}\in[0,L]$.
	    Because of \eqref{eq:251008d} we have $\big(\tilde{u}_\infty, \tilde{v}_\infty\big)\geq (1, 1)$ and for each $t\in\mathbb{R}$ there exists $x_0(t)\in\mathbb{R}$ such that either  
	    \begin{subequations}\label{eq:251008e}
		\begin{equation}\label{eq:251008ea}
		\tilde{u}_\infty(ct-x_0(t)-nL, x_0(t)+nL) = 1 ,  \,^\forall n\in\mathbb{Z}, 
	    \end{equation}
	    or 
	    \begin{equation}\label{eq:251008eb}
		\tilde{v}_\infty(ct-x_0(t)-nL, x_0(t)+nL) = 1 , \,^\forall n\in\mathbb{Z}. 
	    \end{equation}
	    \end{subequations} 
	    Suppose that \eqref{eq:251008ea} holds for $t=0$. From the strong minimum principle for cooperative parabolic systems  (see \cite[Chap. 3 Theorem 13 p.190]{Pro-Wei-1984}) we deduce that $\tilde{u}_\infty\equiv 1$, and it follows from \eqref{eq:251008ca} that $\tilde{v}_\infty\equiv 1$ as well. If \eqref{eq:251008eb} holds, a similar argument leads to the same conclusion that $(\tilde{u}_\infty, \tilde{v}_\infty)\equiv (1,1)$.  But from \eqref{eq:251007e} we have $\max\big(\tilde{u}_\infty(0, x_{\infty}),\tilde{v}_\infty(0, x_\infty)\big)\geq 1+\varepsilon$, hence we have reached a contradiction. 

	    We have proved: 
	    \begin{equation*} \label{eq:251008f}
		    \lim_{t\to-\infty}\inf_{x\in\mathbb{R}}\min\left(\frac{U(t, x)}{e(t)\varphi^\lambda(x)}, \frac{V(t, x)}{e(t)\psi^\lambda(x)}\right) =\lim_{t\to-\infty}\sup_{x\in\mathbb{R}}\max\left(\frac{U(t, x)}{e(t)\varphi^\lambda(x)}, \frac{V(t, x)}{e(t)\psi^\lambda(x)}\right)=1.   
	    \end{equation*}
	     This establishes the claim \eqref{eq:decay-U,V} with $\alpha=1$, and the proof of Theorem \ref{thm:TW2} is complete.
\end{proof}

%%%%%%%%%%%%%%%%%%%%
%%%%%%%%%%%%%%%%%%%%

\section{Behavior behind the front}
\label{sec:behind}
In this section, we study the behavior of traveling waves far behind the front, 
under the assumption that the coefficients are spatially homogeneous, for that the spatial period $L$ is very small (the rapidly oscillating case). 
A large part of the arguments here rely on the results we presented in our previous paper \cite{Griette-Matano-2025}.

\subsection{The homogeneous case}
In this subsection we assume that the coefficients of \eqref{eq:main-sys} are spatially homogeneous. 
We begin with the existence of a unique homogeneous equilibrium for \eqref{eq:main-sys}.

\begin{lem}[Existence and uniqueness of stationary state \protect{\cite{Griette-Matano-2025}}]\label{lem:stat-ode-uniqueness}
    Let $ r_u, r_v\in\R$, $\kappa_u>0$, $\kappa_v>0$, and $\mu_u, \mu_v>0$. 
	Assume that $\lambda_A>0$, where $\lambda_A$ denotes the principal eigenvalue of the matrix $A$ in \eqref{A}.
    Then, there exists a unique positive homogeneous steady state $(u^*, v^*)$ for \eqref{eq:main-sys}. 
\end{lem}

\begin{proof}
	The proof follows from algebraic computations on the system satisfied by the homogeneous steady state. We refer to \cite[Lemma 12]{Griette-Matano-2025} for the details.  
\end{proof}

\begin{proof}[Proof of Theorem \ref{thm:asymptotic-behavior}]
It is shown in Theorem 2.4 of our previous paper \cite{Griette-Matano-2025} that any nonnegative solution $(u(t,x),v(t,x))$ to the Cauchy problem for \eqref{eq:main-sys-homo} with nontrivial initial data converges to $(u^*,v^*)$ in the sense that, for each $0<c<c^*_R\,(=c^*_L)$, it holds that
\[
\lim_{t\to+\infty} \sup_{|x|\leq ct} \max\big(|u(t,x)-u^*|, |v(t,x)-v^*|\big)=0.
\]
Hence
\begin{equation}\label{eq:convergence-at-0}
\lim_{t\to+\infty} \max\big(|u(t,0)-u^*|, |v(t,0)-v^*|\big)=0.
\end{equation}
In particular, if $(u(t,x),v(t,x))$ is a right traveling wave of the form $\big(U(x-ct), V(x-ct)\big)$, then it is easily seen that \eqref{eq:convergence-at-0} is equivalent to \eqref{eq:convergence-homogeneous}. Similarly, if  $(u(t,x),v(t,x))$ is a left traveling wave of the form $\big(U(x+ct), V(x+ct)\big)$, \eqref{eq:convergence-at-0} is equivalent to
\[
\lim_{\xi\to +\infty}\big(U(\xi), V(\xi)\big)=(u^*,v^*).
\]
This completes the proof of Theorem \ref{thm:asymptotic-behavior}.
 \end{proof}
 
%%%%%%%%%%%%%%%%%%%%

\subsection{The rapidly oscillating case}

Next, we deal with the case where the spatial period $L=\ep$ is sufficiently small. We begin with the following lemma, which is an analog of Lemma~\ref{lem:stat-ode-uniqueness} in the previous subsection.

\begin{lem}[Existence and uniqueness of rapidly oscillating entire solution \protect{\cite{Griette-Matano-2025}}]\label{lem:rapid-osc-unique}
	Let the assumptions of Theorem \ref{thm:asymptotic-behavior2} hold. Then there exists $\bar{\ep}$ such that for any $0<\ep\leq\bar{\ep}$, the system \eqref{eq:main-sys-ep} possesses a unique entire solution $(u^*_\ep, v^*_\ep)$ that is bounded from above and from below by positive constants, and this entire solution is an $\ep$-periodic stationary solution.  
\end{lem}

\begin{proof}
	The proof of this result has been given in our former paper \cite[Lemma 21]{Griette-Matano-2025}. 
\end{proof}

Now we are ready to prove assertion 2 of Theorem \ref{thm:asymptotic-behavior2}

\begin{proof}[Proof of Theorem \ref{thm:asymptotic-behavior2}]
The claim \eqref{eq:c*-ep} is proved in Lemma 18 of our previous paper \cite{Griette-Matano-2025}. 
By Theorem 2.5 of \cite{Griette-Matano-2025}, any nonnegative solution $\big(u(t,x),v(t,x)\big)$ to the Cauchy problem for \eqref{eq:main-sys-ep} with nontrivial initial data satisfies, for any $c_1, c_2$ with $0<c_1<c^*_{\ep,L}$, $0<c_2<c^*_{\ep,R}$,
\[
\lim_{t\to+\infty} \sup_{-c_1t\leq x\leq c_2 t} \big(|u(t,x)-u^*_\ep(x)|+|v(t,x)-v^*_\ep(x)|\big)=0.
\]
Hence
\begin{equation}\label{eq:convergence-[0,ep]}
\lim_{t\to+\infty} \sup_{0\leq x\leq \ep} \big(|u(t,x)-u^*_\ep(x)+|v(t,x)-v^*_\ep(x)|\big)=0.
\end{equation}
In particular, if  $(u(t,x),v(t,x))$ is a right traveling wave, then by the condition \eqref{eq:TW-propagating},
\[
u\left(t+\frac{\ep}{c}, x\right) = u(t, x-\ep), \; v\left(t+\frac{\ep}{c}, x\right) = v(t, x-\ep), \ \ \text{for all}\ \ (t,x)\in\R^2,
\]
therefore, it is not difficult to see that \eqref{eq:convergence-[0,ep]} implies \eqref{eq:convergence-rapidosc}. 
Left traveling waves can be treated similarly by using \eqref{eq:TW-propagating2}. This completes the proof of Theorem \ref{thm:asymptotic-behavior2}.
\end{proof}

%%%%%%%%%%%%%%%%%%%%%%%%%%%%%%%%%%%%%%%%%%%%%%%%%%%%%%%%%%%%%%%%%%%%%%%%%%%%%%%
%%%%%%%%%%%%%%%%%%%%%%%%%%%%%%%%%%%%%%%%%%%%%%%%%%%%%%%%%%%%%%%%%%%%%%%%%%%%%%%
%%%%%%%%%%%%%%%%%%%%%%%%%%%%%%%%%%%%%%%%%%%%%%%%%%%%%%%%%%%%%%%%%%%%%%%%%%%%%%%
%%%%%%%%%%%%%%%%%%%%%%%%%%%%%%%%%%%%%%%%%%%%%%%%%%%%%%%%%%%%%%%%%%%%%%%%%%%%%%%

%%%%%%%%%%%%%%%%%%%
%%%%%%%%%%%%%%%%%%%%

%%%%%%%%%%%%%%%%%%%%%%%%%%%%
\section{Asymmetric speeds: Proof of Theorem \ref{thm:asymmetric-speeds}}
\label{s:proof-asym-speed}

We first note that the $\lambda$-periodic eigenvalue problem \eqref{eq:lambda-periodic-principal-eigen2} for the system \eqref{eq:RD-asym} is given in the form:
\begin{equation}\label{EP}\tag{EP$_\lambda$}
\begin{cases}
	\varphi''-2\lambda \varphi' +(\lambda^2+1)\varphi -\mu_1(z)\varphi+\mu_2(z)\psi=k(\lambda)\varphi&(z\in\R ),\\[1pt]
	\psi''-2\lambda \psi' +(\lambda^2+r)\psi+\mu_1(x)\varphi-\mu_2(z)\psi =k(\lambda)\psi&(z\in\R ),\\[2pt]
\varphi, \, \psi:\ \hbox{$L$-periodic}.
\end{cases}
\end{equation}
Let $\big(k(\lambda),(\varphi,\psi)\big)$ denote the principal eigenpair of \eqref{EP} satisfying $\varphi>0, \psi>0$. Then, recalling \eqref{eq:speed}, we have the following formula for the right and left spreading speeds:
\begin{equation}\label{c*}
	c^*_R:=c^*_R(\ep, \delta)=\min_{\lambda>0} \frac{k(\lambda)}{\lambda},\quad
	c^*_L:=c^*_L(\ep, \delta)=\min_{\lambda>0} \frac{k(-\lambda)}{\lambda},
\end{equation}
where we indicate the dependence of the speeds on the parameters $\ep, \delta$. 
By summing up the two equations in \eqref{EP} and integrating over the interval $[0,L]$, we obtain
\[
(\lambda^2+1)\int_0^L\varphi \, \dd z+(\lambda^2+r)\int_0^L \psi \,\dd x=k(\lambda)\int_0^L(\varphi+\psi) \dd x.
\]
This, together with the positivity of $\varphi, \psi$, implies (cf. \eqref{k-quadratic})
\begin{equation}\label{k-estimate}
	\lambda^2+1<k(\lambda)<\lambda^2+r\quad ({}^\forall \lambda\in\R ),
\end{equation}
or, equivalently,
\begin{equation}\label{k-estimate2}
\lambda+\frac{1}{\lambda}<\frac{k(\pm\lambda)}{\lambda}<\lambda+\frac{r}{\lambda}\quad ({}^\forall \lambda>0).
\end{equation}
Hence, by \eqref{c*},
\[
\hspace{40pt}2\leq c^*_R<2\sqrt{r},\quad  2\leq c^*_L \leq 2\sqrt{r}
\qquad \hbox{(cf. Proposition~\ref{prop:c*-estimate})}.
\]
Furthermore, from \eqref{k-estimate2} we see that the minima in \eqref{c*} are both attained on the interval
\[
\sqrt{r}-\sqrt{r-1}\leq \lambda \leq \sqrt{r}+\sqrt{r-1}.
\]

%%%%%%%%%%%%%%%%%%
\subsection{Estimates on the eigenvectors}\label{ss:example}
Here we establish some properties of the eigenvectors when $\ep, \delta$ are small and $\ep/\delta$ is bounded. Here, for simplicity, we will use the notation $k:=k(\lambda)$ whenever the context is clear.  \medskip

\noindent
{\bf Analysis on the interval $[0,\ep]$}

\vskip 5pt
On this interval, the solution $(\varphi,\psi)$ of \eqref{EP} satisfies the following:
\begin{equation}\label{I1}
\begin{cases}
\varphi''-2\lambda \varphi' +(\lambda^2+1)\varphi -\dfrac{\mu}{\ep\delta}\varphi=k\varphi&(0<z<\ep),\\[8pt]
\psi''-2\lambda \psi' +(\lambda^2+r)\psi+\dfrac{\mu}{\ep\delta}\varphi=k\psi&(0<z<\ep).
\end{cases}
\end{equation}
The first equation of \eqref{I1} can be solved for $\varphi$ by using the roots of the characteristic equation $X^2-2\lambda X +(\lambda^2+1-\mu/(\ep\delta)-k)=0$. More precisely,
\[
	\varphi(z)=C_1^{+}e^{X_1^{+} z}+C_1^{-}e^{X_1^{-} z}, \text{ where }
X_1^{\pm}=\lambda\pm \sqrt{k+\frac{\mu}{\ep\delta}-1},
\]
and the constants $C_1^+, C_1^-$ satisfy 
$
C_1^+ + C_1^- = \varphi(0)$ and $ C_1^+X_1^+ + C_1^-X_1^- = \varphi'(0).
$
Consequently,
\begin{equation*}
\varphi(z)=C_1^{+}\left(1+X_1^{+} z+O\left((X_1^{+} z)^2\right)\right)+C_1^{-}\left(1+X_1^{-} z+O\left((X_1^{-} z)^2\right)\right)
= \varphi(0)+z\varphi'(0)+O\left(\frac{\ep}{\delta}\right)
\end{equation*}
for all $0\leq z\leq \ep$.
In particular,
\begin{equation}\label{phi(ep)}
\varphi(\ep)=\varphi(0)+O\left(\frac{\ep}{\delta}\right).
\end{equation}
Integrating the first equation in \eqref{I1} over $[0,\ep]$ and using the above estimate, we obtain
\begin{equation}\label{phi'(ep)}
\varphi'(\ep)=\varphi'(0)+2\lambda\left(\varphi(\ep)-\varphi(0)\right)+\frac{\mu}{\delta}\varphi(0)+O(\ep) 
=\varphi'(0)+\frac{\mu}{\delta}\varphi(0)+O\left(\frac{\ep}{\delta}\right).
\end{equation}
Next we note that the function $\Psi(z):=\varphi(z)+\psi(z)$ satisfies the following:
\[
\Psi''-2\lambda \Psi'+(\lambda^2+r-k)\Psi=(r-1)\varphi.
\]
Since the coefficients of the above equation are independent of $\ep,\delta$, we have
\[
\Psi(\ep)=\Psi(0)+O(\ep),\quad \Psi'(\ep)=\Psi'(0)+O(\ep).
\]
Combining this with \eqref{phi(ep)} and \eqref{phi'(ep)}, we obtain
\begin{equation}\label{psi(ep)}
\psi(\ep)=\psi(0)+O\left(\frac{\ep}{\delta}\right),\quad \psi'(\ep)=\psi'(0)-\frac{\mu}{\delta}\varphi(0)+O\left(\frac{\ep}{\delta}\right).
\end{equation}

\vskip 10pt
\noindent
{\bf Analysis on the interval $[\ep,\delta-\ep]$}

\vskip 5pt
On this interval, the solution $(\varphi,\psi)$ of \eqref{EP} simply satisfies the following equations:
\begin{equation}\label{I2}
\begin{cases}
\varphi''-2\lambda \varphi' +(\lambda^2+1)\varphi =k\varphi,\\[2pt]
\psi''-2\lambda \psi' +(\lambda^2+r)\psi=k\psi,
\end{cases}
\end{equation}
along with the initial conditions
\[
\varphi(\ep)=\varphi(0)+O\left(\frac{\ep}{\delta}\right), \ \ \varphi'(\ep)=\varphi'(0)+\frac{\mu}{\delta}\varphi(0)+O\left(\frac{\ep}{\delta}\right),
\]
\[
\psi(\ep)=\psi(0)+O\left(\frac{\ep}{\delta}\right), \ \ \psi'(\ep)=\psi'(0)-\frac{\mu}{\delta}\varphi(0)+O\left(\frac{\ep}{\delta}\right).
\]
Therefore
\[
\begin{split}
\varphi(\delta-\ep)&=\varphi(\ep)+(\delta-2\ep)\varphi'(\ep)+\frac12(\delta-2\ep)^2\varphi''(\ep)+\cdots\\
&=\varphi(0)+(\delta-2\ep)\left(\varphi'(0)+\frac{\mu}{\delta}\varphi(0)\right)+O\left(\frac{\ep}{\delta}\right)+O(\delta^2)\\[2pt]
&=(1+\mu)\varphi(0)+O(\delta)+O\left(\frac{\ep}{\delta}\right),
\end{split}
\]
\[
\begin{split}
\varphi'(\delta-\ep)&=\varphi'(\ep)+(\delta-2\ep)\varphi''(\ep)+\cdots\\
&=\varphi'(\ep)+(\delta-2\ep)\left(2\lambda\varphi'(\ep)-(\lambda^2+1-k)\varphi(\ep)\right)+\cdots\\[2pt]
&=\varphi'(0)+\frac{\mu}{\delta}\varphi(0)+2\lambda\mu\varphi(0)+O(\delta)+O\left(\frac{\ep}{\delta}\right).
\end{split}
\]
Similarly, we have
\[
\begin{split}
\psi(\delta-\ep)&=\psi(0)-\mu\varphi(0)+O(\delta)+O\left(\frac{\ep}{\delta}\right),\\ 
\psi'(\delta-\ep)&=\psi'(0)-\frac{\mu}{\delta}\varphi(0)-2\lambda\mu\varphi(0)+O(\delta)+O\left(\frac{\ep}{\delta}\right).
\end{split}
\]
Combining these, and arguing as in the derivation of \eqref{phi(ep)}, \eqref{phi'(ep)} and \eqref{psi(ep)}, we obtain
\begin{equation}\label{phi(delta)}
\begin{cases}
\, \varphi(\delta)=(1+\mu)\varphi(0)+O(\delta)+O\left(\dfrac{\ep}{\delta}\right),\\[4pt] 
\, \varphi'(\delta)=\varphi'(0)+\dfrac{\mu}{\delta}\left(\varphi(0)-\psi(\delta)\right)+2\lambda\mu\varphi(0)+O(\delta)+O\left(\dfrac{\ep}{\delta}\right),\\[6pt]
\, \psi(\delta)=\psi(0)-\mu\varphi(0)+O(\delta)+O\left(\dfrac{\ep}{\delta}\right),\\[6pt]
\, \psi'(\delta)=\psi'(0)-\dfrac{\mu}{\delta}\left(\varphi(0)-\psi(\delta)\right)-2\lambda\mu\varphi(0)+O(\delta)+O\left(\dfrac{\ep}{\delta}\right).
\end{cases}
\end{equation}

In what follows we assume $\ep=O(\delta^2)$, so the term $O(\delta)+O\left(\frac{\ep}{\delta}\right)$ is simply expressed as $O(\delta)$. Let us normalize $(\varphi,\psi)$ so that
\begin{equation}\label{phi(0)=1}
\varphi(0)=1\ \left(\,=\varphi(L)\,\right).
\end{equation}
The first line of \eqref{phi(delta)} then implies
\begin{equation}\label{phi(delta)=1+mu}
\varphi(\delta)=1+\mu +O(\delta).
\end{equation}
Note that $\varphi$ satisfies the first line of \eqref{I2} on $\delta<z<L$, which we can rewrite as 
\begin{equation}\label{Phi''}
\left(e^{-\lambda z}\varphi\right)''=(k-1)\left(e^{-\lambda z}\varphi\right)\quad (\delta<z<L).
\end{equation}
As $k$ satisfies the inequality \eqref{k-estimate}, we have $k-1>0$. 
Therefore $e^{-\lambda z}\varphi$ is convex on $[\delta,L]$, hence its maximum is attained either at $z=\delta$ or at $z=L$. This and \eqref{phi(0)=1}, \eqref{phi(delta)=1+mu} imply
\begin{equation}\label{phi-bound}
0<\varphi(z)\leq (1+\mu)e^{\lambda L}  + O(\delta) \quad(z\in[\delta,L]).
\end{equation}
Hence $\varphi|_{[\delta,L]}$ is uniformly bounded as $\ep, \delta\to 0$. 
Integrating the first line of \eqref{I2}, we get
\[
\varphi'(L)-\varphi'(\delta)=2\lambda(\varphi(L)-\varphi(\delta))+(k-\lambda^2-1)\int_{\delta}^L\varphi(z)\dd z.
\]
The right-hand side of the above identity is uniformly bounded as $\ep, \delta\to 0$. This and the second line of \eqref{phi(delta)}, 
together with $\varphi'(L)=\varphi'(0)$, implies $\varphi(0)-\psi(\delta)=O(\delta)$, thus
\begin{equation}\label{psi(delta)}
\psi(\delta)=1+O(\delta).
\end{equation}
Combining this and the third line of \eqref{phi(delta)}, and recalling $\ep=O(\delta^2)$, we get 
\begin{equation}\label{psi(0)}
\psi(0)=1+\mu+O(\delta).
\end{equation}
Note also that, by summing up the second and the fourth lines of \eqref{phi(delta)}, we obtain
\begin{equation}\label{phi'+psi'}
\varphi'(\delta)+\psi'(\delta)=\varphi'(0)+\psi'(0)+O(\delta).
\end{equation}

%%%%%%%%%%%%%%%%%%
\subsection{Analysis of the limit problem}\label{ss:analysis-limit}

\vskip 5pt
\noindent
{\bf Uniform derivative estimates}

\vskip 5pt
Let us study the limit problem of \eqref{EP} with \eqref{eq:mu-12} as $\ep, \delta\to 0$. As before, we impose
\[
\ep=O(\delta^2).
\]
For each $\ep, \delta>0$ and $\lambda\in\R $, let  $k(z;\lambda,\ep,\delta)$ and $(\varphi(z;\lambda,\ep,\delta),\psi(z;\lambda,\ep,\delta))$ denote, respectively, the principal eigenvalue and the principal eigenvector of \eqref{EP} with \eqref{eq:mu-12} under the normalization condition $\varphi(0;\lambda,\ep,\delta)=1$ as before (see \eqref{phi(0)=1}). 
Let us first show that $\varphi, \varphi', \psi, \psi'$ are uniformly bounded on $[\delta,L]$ as $\ep, \delta\to 0$. The uniform boundedness of $\varphi$ is already shown in \eqref{phi-bound}. In order to prove the uniform boundedness of $\varphi'$, we prepare the following lemma.

\begin{lemma}\label{lem:Phi-estimates}
Let $\Phi$ be a $C^2$ function on an interval $[a,b]$ satisfying the equation
	\begin{equation}\label{eq:260526a}
\Phi''(z)=F(z)\quad (a<z<b).
	\end{equation}
Then the following estimates hold:
\begin{equation}\label{Phi-estimates}
\begin{split}
& |\Phi(z)|\leq \max\left(|\Phi(a)|,|\Phi(b)|\right) + \frac{(b-z)(z-a)}{b-a}\int_a^b|\Phi(\xi)|\dd\xi,\\
& |\Phi'(z)|\leq \frac{\left|\Phi(b)-\Phi(a)\right|}{b-a}+\int_a^b|\Phi(\xi)|\dd\xi.
\end{split}
\end{equation}
\end{lemma}

\begin{proof}
	The solution of the  equation \eqref{eq:260526a} is given in the following form
\[
\Phi(z)=\frac{(b-z)\Phi(a)+(z-a)\Phi(b)}{b-a}+\int_a^b G(z,\xi)F(\xi)\dd\xi.
\]
where $G(z,\xi)$ denotes the Green's function on the interval $[a,b]$, namely
\[
G(z,\xi)=
\begin{cases}
\,-\dfrac{(\xi-a)(b-z)}{b-a} &(\hbox{if}\ z\geq \xi)\\[6pt]
\,-\dfrac{(z-a)(b-\xi)}{b-a} &(\hbox{if}\ z<\xi)
\end{cases}
\]
The estimates \eqref{Phi-estimates} follow easily from the above expression, so we omit the details.
\end{proof}

Now let us prove the uniform boundedness of $\varphi'$. Since $e^{-\lambda z}\varphi(z)$ satisfies the equation \eqref{Phi''}, we see from the above lemma that $e^{-\lambda z}\varphi(z)$ and $\left(e^{-\lambda z}\varphi(z)\right)'$ are uniformly bounded. Consequently $\varphi'$ is uniformly bounded on $[0,\delta]$ as $\ep, \delta\to 0$.

Next we prove the uniform boundedness of $\psi$, $\psi'$.
Note that $\psi$ satisfies the second line of \eqref{I2} on $\delta<z<L$, which we can rewrite as 
\begin{equation}\label{Psi''}
\left(e^{-\lambda z}\psi\right)''=(k-r)\left(e^{-\lambda z}\psi\right)\quad (\delta<z<L).
\end{equation}
If $k\geq r$, then $\left(e^{-\lambda z}\psi\right)$ is convex. This, together with \eqref{psi(delta)} and \eqref{psi(0)}, implies that $\psi$ is uniformly bounded. Arguing as we did for $\varphi$, we obtain the uniform boundedness of $\psi'$. 

Next, if $k<r$, by summing up the two equations in \eqref{EP} and integrating over $[0,L]$, we obtain
\[
\begin{split}
& (k-\lambda^2-1)\int_0^L \varphi(z)\dd z=(\lambda^2+r-k)\int_0^L \psi(z)\dd z.\\
& \ \ \hbox{Hence}\ \ \ 
(r-\lambda^2-1)\int_0^L \varphi(z)\dd z>\lambda^2 \int_0^L \psi(z)\dd z.
\end{split}
\]
Since the left-hand side of the above inequality is uniformly bounded as $\ep, \delta\to 0$, so is the right-hand side. In particular, $\int_\delta^L \psi(z)\dd z$ is uniformly bounded as $\ep,\delta\to 0$. Combining this with \eqref{psi(delta)}, \eqref{psi(0)}, and applying Lemma \ref{lem:Phi-estimates}, we see that $\psi, \psi'$ are uniformly bounded on $[\delta,L]$ as $\ep.\delta\to 0$.

\vskip 10pt
\noindent
{\bf Formulation of the limit problem}

\vskip 5pt
Now we let $\ep, \delta\to 0$ while keeping $\ep=O(\delta^2)$. The interval $[0,\delta]$ shrinks to the point $z=0$. 
Since $\varphi, \varphi', \psi, \psi'$ are uniformly bounded as $\ep,\delta\to 0$, we see, by the bootstrap argument for \eqref{I2} that $\varphi'', \varphi''', \psi'', \psi'''$ are also uniformly bounded on $[\delta,L]$ as $\ep,\delta\to 0$. Note also that, by \eqref{k-estimate}, the eigenvalue $k(\lambda;\ep,\delta)$ varies within the range
\[
\lambda^2+1<k(\lambda;\ep,\delta)<\lambda^2+r.
\]
Consequently, for any sequence $\ep_n, \delta_n\to 0$ with $\ep_n=O(\delta_n^2)$, we can extract a subsequence $(\ep_{n_j},\delta_{n_j})$ such that $(\varphi(z,\lambda,\ep_{n_j},\delta_{n_j}), \psi(z,\lambda,\ep_{n_j},\delta_{n_j}))$ converges to some function $(\varphi_0(z;\lambda),\psi_0(z;\lambda))$ on $(0,L]$ as $j\to\infty$ in the $C^2$ sense and $k(\lambda,,\ep_{n_j},\delta_{n_j})$ to some constant $k_0$ as $j\to \infty$. 
Needless to say, since $(\varphi(z,\lambda,\ep_{n_j},\delta_{n_j}), \psi(z,\lambda,\ep_{n_j},\delta_{n_j}))$ are $L$-periodic, so is the limit $(\varphi_0,\psi_0)$, but this function has discontinuity at $z=0$. 
In view of \eqref{phi(0)=1}, \eqref{phi(delta)=1+mu}, \eqref{psi(0)}, \eqref{psi(delta)} and \eqref{phi'+psi'}, we see that $(\varphi_0,\psi_0,k_0)$ is a solution of the following limit problem:
\begin{subequations}\label{limit-problem}
\begin{eqnarray}\label{limit-phi}
&\begin{cases}
\,\varphi''-2\lambda \varphi' +(\lambda^2+1)\varphi =k\varphi \ \ \ (0<z<L),\\[2pt]
\,\varphi(+0)=1+\mu,\ \ \varphi(L-0)=1,\\[2pt]
\, \varphi(z)>0\ \ \ (0<z<L),
\end{cases}\\[4pt]
\label{limit-psi}
&\begin{cases}
\,\psi''-2\lambda \psi' +(\lambda^2+r)\psi=k\psi\ \ \ (0<z<L),\\[2pt]
\psi(+0)=1,\ \ \psi(L-0)=1+\mu,\\[2pt]
\, \psi(z)>0\ \ \ (0<z<L),
\end{cases}\\[6pt]
&\label{limit-derivatives}
\varphi'(+0)+\psi'(+0)=\varphi'(L-0)+\psi'(L-0).
\end{eqnarray}
\end{subequations}

Note that \eqref{limit-phi} and \eqref{limit-psi} are separate boundary value problems for $\varphi$ and $\psi$, and they are coupled only through the condition \eqref{limit-derivatives}, which plays a crucial role. 
More precisely, for each $k$ in a certain range, one can solve \eqref{limit-phi} and \eqref{limit-psi} separately, but the resulting pair $(\varphi, \psi)$ satisfies the condition \eqref{limit-derivatives} only for 
a special value of $k$. In this sense, \eqref{limit-problem} can be regarded as an eigenvalue problem with $k$ being the eigenvalue. The following theorem holds:

\begin{thm}[The limit eigenproblem]\label{thm:EP-limit}
	For each $\lambda\in \R $, there is a unique value of $k$, denoted by $k_0(\lambda)$, for which \eqref{limit-problem} has a solution $(\varphi, \psi)$. 
Furthermore, the principal eigenvalue $k(\lambda,\ep,\delta)$ of \eqref{EP} with \eqref{eq:mu-12} converges to $k_0(\lambda)$ as $\ep,\delta\to 0$ locally uniformly in $\lambda$, provided that $\ep=O(\delta^2)$. 
\end{thm}

Before proving the above theorem, let us make some preliminary analysis. 
By summing up the equations in \eqref{limit-phi} and \eqref{limit-psi} and integrating over $0<z<L$ and using \eqref{limit-derivatives}, we obtain
\[
(k_0-\lambda^2-1)\int_0^L \varphi_0(z)\dd z=(\lambda^2+r-k_0)\int_0^L \psi_0(z)\dd z.
\]
Since the two integrals are both positive and since $r>1$, we have 
\begin{equation}\label{k0-estimate}
	\lambda^2+1<k_0<\lambda^2+r\quad ({}^\forall \lambda\in\R ).
\end{equation}
Thus we can retrieve the inequality \eqref{k-estimate} directly from the limit problem \eqref{limit-problem}.

Now we fix $\lambda$ arbitrarily, and let $k$ be a positive constant that is not necessarily the eigenvalue of \eqref{limit-problem}. The equations \eqref{limit-phi}, \eqref{limit-psi} can be rewritten as follows:
\[
\left(e^{-\lambda z}\varphi\right)''=(k-1)\left(e^{-\lambda z}\varphi\right),\ \ \  
\left(e^{-\lambda z}\psi\right)''=(k-r)\left(e^{-\lambda z}\psi\right)\quad (0<z<L).
\]
Therefore, in order for \eqref{limit-phi} and \eqref{limit-psi} to have positive solutions, both $k-1$ and $k-r$ have to be larger than the principal eigenvalue $\nu$ of the following problem, namely $-\pi^2/L^2$\,:
\[
\Phi''=\nu \hspace{1pt}\Phi\ (0<z<L),\quad \ \Phi(0)=\Phi(L)=0.
\]
Hence
\begin{equation}\label{k-range}
k>r-\frac{\pi^2}{L^2}.
\end{equation}
Conversely, for any $k$ satisfying \eqref{k-range}, the boundary value problems \eqref{limit-phi}, \eqref{limit-psi} have unique solutions, which we denote by $\big(\varphi(z;k),\psi(z;k)\big)$. Define
\[
f(k)=\varphi'(L-0;k) - \varphi'(+0;k), \quad g(k)=\psi'(L-0;k) -\psi'(+0,k)
\]
and
\begin{equation}\label{F(k)}
F(k)=f(k)+g(k)=\varphi'(L-0;k)+\psi'(L-0;k) - \varphi'(+0;k)-\psi'(+0,k).
\end{equation}
Then \eqref{limit-derivatives} holds if and only if $F(k)=0$.

\begin{proof}[Proof of Theorem~\ref{thm:EP-limit}]
Fix $\lambda\in\R$ arbitrarily. 
Since we already know the existence of a solution of \eqref{limit-problem} as a limit of the sequence $(\varphi(z,\lambda,\ep_{n_j},\delta_{n_j}), \psi(z,\lambda,\ep_{n_j},\delta_{n_j}))$ and $k(\lambda,\ep_{n_j},\delta_{n_j})$, let us prove the uniqueness of the solution. It suffices to show that $F(k)$ is strictly monotone in $k$.

Choose arbitrary constants $\tilde k>k$ satisfying \eqref{k-range} and let
\[
\varphi(z):=\varphi(z;k),\ \ \widetilde\varphi(z):=\varphi(z,\tilde k),\ \ \psi(z):=\psi(z;k),\ \ \widetilde\psi(z):=\psi(z;\tilde k).
\]
Then
\[
\varphi''-2\lambda \varphi' +(\lambda^2+1)\varphi =k\varphi,\quad
\widetilde\varphi''-2\lambda \widetilde\varphi' +(\lambda^2+1)\widetilde\varphi =\tilde k\widetilde\varphi\geq k\widetilde\varphi.
\]
Therefore $\widetilde\varphi$ is a subsolution of the equation satisfied by $\varphi$, and $\widetilde\varphi(+0)=\varphi(+0)$, $\widetilde\varphi(L-0)=\varphi(L-0)$. 
The function $w=\widetilde\varphi-\varphi$ then satisfies
\[
\left(e^{-\lambda z}w\right)''\geq (k-1)\left(e^{-\lambda z} w\right)\ \ (0<z<L),\quad 
w(+0)=w(L-0)=0.
\]
Since $k-1>-\pi^2/L^2$ by \eqref{k-range}, we can apply the strong maximum principle to show that $w(z)<0\;(0<z<L)$ and that $w'(+0)<0$, $w'(L-0)>0$. In other words,
\[
\widetilde \varphi(z)<\varphi(z)\ \ (0<z<L),\quad \widetilde\varphi'(+0)<\varphi'(+0),\quad 
\widetilde\varphi'(L-0)>\varphi'(L-0).
\]
Hence
\[
f(\tilde k)-f(k)=\widetilde\varphi'(L-0)- \widetilde\varphi'(+0) -\left( \varphi'(L-0)-\varphi'(+0)\right)>0.
\]
This implies that $f(k)$ is strictly monotone increasing. 
Similarly, $g(k)$ is strictly monotone increasing, therefore $F(k)$ is strictly monotone increasing in the range $r-\pi^2/L^2<k<\infty$. 
This proves the uniqueness of the solution of \eqref{limit-problem} for each given $\lambda\in\R$. 

Since any sequence of solutions $(\varphi(z,\lambda,\ep_{n_j},\delta_{n_j}), \psi(z,\lambda,\ep_{n_j},\delta_{n_j})),k(\lambda,\ep_{n_j},\delta_{n_j})$ of \eqref{EP} has a subsequence that converges to a solution of \eqref{limit-problem} and since the solution of \eqref{limit-problem} is unique, we see that the solution  $(\varphi(z;\lambda,\ep,\delta), \psi(z;\lambda,\ep,\delta)\big)$ of \eqref{EP} with \eqref{eq:mu-12} converges to the unique solution of the problem \eqref{limit-problem} as $\ep,\delta\to 0$ without taking a subsequence, and that $k(\lambda,\ep,\delta)\to  k_0(\lambda)$. 
Furthermore, since $k(\lambda,\ep,\delta)$ is convex in $\lambda$, the convergence $k(\lambda,\ep,\delta)\to k_0(\lambda)\; (\ep,\delta\to 0)$ is locally uniform in $\lambda$, 
since a pointwise convergence of a sequence of convex functions is uniform on bounded sets (\cite[Theorem 10.8]{Rockafellar-1970}). This completes the proof of Theorem~\ref{thm:EP-limit}.
\end{proof}

%%%%%%%%%%%
\vskip 8pt
\noindent
{\bf Comparison in the limit problem}

\vskip 5pt
\begin{lem}
Let $k_0(\lambda)$ be as in Theorem~\ref{thm:EP-limit}. Then 
\begin{equation}\label{k(lambda)<k(-lambda)2}
k_0(\lambda)<k_0(-\lambda)\quad \hbox{for all}\ \ \lambda>0.
\end{equation}
\end{lem}

\begin{proof}
For each $\lambda\in\R $ and $k$ satisfying \eqref{k-range}, let $\varphi(z;k,\lambda), \psi(z;k,\lambda)$ denote the solution of \eqref{limit-phi}, \eqref{limit-psi}, and define
\[
f(k;\lambda)=\varphi'(L-0;k,\lambda)-\varphi'(+0;k,\lambda),\quad
g(k;\lambda)=\psi'(L-0;k,\lambda)-\psi'(+0;k,\lambda),
\]
\[
F(k;\lambda)=f(k;\lambda)+g(k;\lambda).
\]
As we have seen above, $F(k;\lambda)$ is strictly increasing in $k$, and $F(k;\lambda)=0$ if and only if $k=k_0(\lambda)$. 
In order to make a finer estimate of $k_0(\lambda)$, we compute $f(k;\lambda)$ and $g(k;\lambda)$ explicitly. 

We start with $f(k;\lambda)$. The characteristic equation for \eqref{limit-phi} is given by
\[
X^2-2\lambda X + \lambda^2+1-k=0. 
\]
In view of \eqref{k0-estimate}, it suffices to consider $k$ in the range $\lambda^2+1<k<\lambda^2+r$. In particular, we may assume $k>1$, hence the solution of the above characteristic equation is given by
\[
X^\pm = \lambda\pm\sqrt{k-1}.
\]
Thus the solution of \eqref{limit-phi} is written in the form
\[
\varphi(z;k,\lambda)=C^+ e^{X^+ z} + C^- e^{X^- z},
\]
with $C^+, C^-$ being determined by the boundary conditions $\varphi(+0)=1+\mu$, $\varphi(L-0)=1$. A simple calculation then shows the following, where $\alpha=\sqrt{k-1}$: 
\begin{equation}\label{f(k)}
f(k;\lambda)=-\mu\lambda+\frac{\alpha}{e^{\alpha L}-e^{-\alpha L}}\left( (\mu+2)\left(e^{\alpha L}+e^{-\alpha L}\right)-2\left((1+\mu)e^{\lambda L}+e^{-\lambda L}\right)\right).
\end{equation}

Next we consider $g(k;\lambda)$. The solution of \eqref{limit-psi} can be expressed by using the roots of $Y^2-2\lambda Y+\lambda^2+r-k=0$, but the expression differs depending on the sign of $k-r$.  

\vskip 5pt
\underbar{(a) The case $k>r$:} In this case, the roots of the characteristic equation are real, so $g(k;\lambda)$ can be computed in the same manner as above, and we have the following, where $\beta=\sqrt{k-r}$:
\begin{equation}\label{g(k)-a}
g(k;\lambda)=\mu\lambda+\frac{\beta}{e^{\beta L}-e^{-\beta L}}\left( (\mu+2)\left(e^{\beta L}+e^{-\beta L}\right)-2\left(e^{\lambda L}+(1+\mu)e^{-\lambda L}\right)\right).
\end{equation}

\vskip 5pt
\underbar{(b) The case $k=r$:} This case is the limit of the above case as $k\to r$ (or $\beta\to 0$), thus
\begin{equation}\label{g(k)-b}
g(k;\lambda)=\mu\lambda+\frac{1}{L}\left(\mu+2-\left(e^{\lambda L}+(1+\mu)e^{-\lambda L}\right)\right).
\end{equation}

\underbar{(c) The case $k<r$:} In this case, we also need use \eqref{k-range} which is a necessary condition for the existence of a positive solution of \eqref{limit-psi}. The solution $\psi$ is expressed by a linear combination of $e^{\lambda z}\cos\omega z$ and $e^{\lambda z}\sin\omega z$, where $\omega=\sqrt{r-k}$, and we get 
\begin{equation}\label{g(k)-c}
g(k;\lambda)=\lambda\mu+\frac{\omega}{\sin\omega L}\left((\mu+2)\cos\omega L - \left(e^{\lambda L}+(1+\mu)e^{-\lambda L}\right)\right).
\end{equation}
Note that the expression of $g(k;\lambda)$ in \eqref{g(k)-c} can also be obtained by replacing $\beta$ in \eqref{g(k)-a} by $i\omega$. So these two expressions are identical.

Now we assume $\lambda>0$, and compare $F(k;\lambda):=f(k;\lambda)+g(k;\lambda)$ and $F(k;-\lambda):=f(k;-\lambda)+g(k;-\lambda)$. In the case where $k>r$, \eqref{f(k)} and \eqref{g(k)-b} give
\[
\begin{split}
F(k,\lambda)& =\frac{\alpha}{e^{\alpha L}-e^{-\alpha L}}\left( (\mu+2)\left(e^{\alpha L}+e^{-\alpha L}\right)-2\left((1+\mu)e^{\lambda L}+e^{-\lambda L}\right)\right)\\
& \hspace{20pt} + \frac{\beta}{e^{\beta L}-e^{-\beta L}}\left( (\mu+2)\left(e^{\beta L}+e^{-\beta L}\right)-2\left(e^{\lambda L}+(1+\mu)e^{-\lambda L}\right)\right),
\end{split}
\]
\[
\begin{split}
F(k,-\lambda)& =\frac{\alpha}{e^{\alpha L}-e^{-\alpha L}}\left( (\mu+2)\left(e^{\alpha L}+e^{-\alpha L}\right)-2\left((1+\mu)e^{-\lambda L}+e^{\lambda L}\right)\right)\\
& \hspace{20pt} + \frac{\beta}{e^{\beta L}-e^{-\beta L}}\left( (\mu+2)\left(e^{\beta L}+e^{-\beta L}\right)-2\left(e^{-\lambda L}+(1+\mu)e^{\lambda L}\right)\right).
\end{split}
\]
Consequently, 
\[
F(k;\lambda)-F(k;-\lambda)=2\mu\left(e^{\lambda L}-e^{-\lambda L}\right)\left(\frac{\beta}{e^{\beta L}-e^{-\beta L}}-\frac{\alpha}{e^{\alpha L}-e^{-\alpha L}}\right).
\]
Since $\lambda>0$, $\alpha>\beta$, and since $y/(e^{y L}-e^{-y L})$ is monotone decreasing in $y$, we have 
\begin{equation}\label{F(lambda)>F(-lambda)}
F(k;\lambda)>F(k;-\lambda)\quad \hbox{for}\ \lambda>0
\end{equation}
if $k>r$. The same inequality also holds for the range $r\geq k> r-\pi^2/L^2$, as one can easily see by using \eqref{g(k)-b} and \eqref{g(k)-c}.  The details are omitted.

Recall that both $F(k;\lambda)$ and $F(k;-\lambda)$ are strictly increasing in $k$, and that  $F(k_0(\lambda),\lambda)=F(k_0(-\lambda),-\lambda)=0$. 
In view of this and \eqref{F(lambda)>F(-lambda)}, we obtain \eqref{k(lambda)<k(-lambda)2}. The lemma is proved.
\end{proof}

We are now in a position to prove Theorem \ref{thm:asymmetric-speeds}.

\begin{proof}[Proof of Theorem \ref{thm:asymmetric-speeds}]
Recall that we have, by \eqref{c*}, 
\begin{equation}\label{c*2}
c^*_R(\ep,\delta)=\min_{\lambda>0} \frac{k(\lambda,\ep,\delta)}{\lambda},\quad
c^*_L(\ep,\delta)=\min_{\lambda>0} \frac{k(-\lambda,\ep,\delta)}{\lambda},
\end{equation}
where $k(\lambda,\ep,\delta)$ denotes the principal eigenvalue of \eqref{EP} with \eqref{eq:mu-12}. 
Furthermore, as shown in \eqref{k-estimate2}, the minima in \eqref{c*2} are attained by $\lambda$ that lies in the following range:
\begin{equation*}\label{lambda-range}
\sqrt{r}-\sqrt{r-1}\leq \lambda \leq \sqrt{r}+\sqrt{r-1}.
\end{equation*}

By Theorem~\ref{thm:EP-limit}, the convergence $k(\lambda,\ep,\delta)\to k_0(\lambda)\;(\ep,\delta\to 0)$ is locally uniform in $\lambda$. 
This, together with \eqref{k(lambda)<k(-lambda)2}, implies that 
\[
k(\lambda,\ep,\delta)<k(-\lambda,\ep,\delta)\quad \hbox{for}\ \,\lambda\in[\sqrt{r}-\sqrt{r-1}, \sqrt{r}+\sqrt{r-1}]
\]
for all sufficiently small $\ep,\delta$ with $\ep=O(\delta^2)$, hence \eqref{eq:speed-asym}. 
This establishes Theorem \ref{thm:asymmetric-speeds}.
\end{proof}
%

%%%%%%%%%%%%%%%%%%%%%%%
%%%%%%%%%%%%%%%%%%%%%%%
\printbibliography
%\bibliographystyle{AIHP_Submission/emss}
%\bibliography{biblio.bib}

%%%%%%%%%%%%%%%%%%%%%%%%
%%%%%%%%%%%%%%%%%%%%%%%%

\end{document}